\documentclass{amsart}  % {{{1
\usepackage{amssymb, tikz, float, enumitem, mathtools, afterpage}
\usetikzlibrary{calc,decorations.markings,matrix,arrows,positioning, backgrounds}
\tikzset{>=latex}
\pdfoutput=1
\usepackage[breaklinks, final]{hyperref}

\newcommand{\A}{\mathbb{A}}
\theoremstyle{plain} \newtheorem{thm}{Theorem}[section]
\newtheorem{prop}[thm]{Proposition}
\newtheorem{lem}[thm]{Lemma}
\newtheorem{cor}[thm]{Corollary}

\theoremstyle{definition} \newtheorem{defn}[thm]{Definition}
\theoremstyle{remark} \newtheorem*{rk}{Remark}
\theoremstyle{plain} \newtheorem{claim}{Claim}
\theoremstyle{plain} 

\newenvironment{claimproof} {
  \begin{proof}[Proof of claim]
  
  } {
  \end{proof}
  }

\DeclareMathOperator{\Sg}{Sg}
\DeclareMathOperator{\Con}{Con}
\DeclareMathOperator{\Clo}{Clo}

\numberwithin{equation}{section}  % number equations within sections
\renewcommand{\phi}{\varphi}
\renewcommand{\epsilon}{\varepsilon}

\usepackage{caption, subcaption, xparse}

\usepackage{framed}

\theoremstyle{definition} 
\theoremstyle{remark} 

\DeclareMathOperator{\cube}{cube}

\DeclareMathOperator{\Line}{line}

\DeclareMathOperator{\faces}{Faces}
\DeclareMathOperator{\glue}{Glue}

\DeclareMathOperator{\tc}{TC}
\DeclareMathOperator{\cut}{Cut}

\DeclareMathOperator{\lines}{Lines}
\DeclareMathOperator{\squares}{Squares}
\DeclareMathOperator{\rect}{Rect}
\DeclareMathOperator{\refl}{Refl}

\DeclareMathOperator{\sym}{Sym}

\DeclareMathOperator{\tol}{Tol}

\DeclareMathOperator{\com}{Com}
\DeclareMathOperator{\fin}{fin}
\DeclareMathOperator{\SD}{SD}
\DeclareMathOperator{\CuBe}{Cube}

\newcommand{\rrot}[3]{{}^{{#1}}{\mathrm{rot}^{#2}_{#3}}}
\newcommand{\ccom}[2]{{\mathrm{cmp}^{#1}_{#2}}}

\def\nat{\mathbb{N}}
\def\A{\mathbb{A}}
\def\var{\mathcal{V}}

\def\Meet{\bigwedge}
\def\Union{\bigcup}
\def\meet{\wedge}
\def\join{\vee}

\def\union{\cup}

\def\finsub{\subseteq_{\fin}}
\def\sd{\SD\text(\meet)}

\tikzset{myStyle/.style={baseline=(center.base), font=\small,
    every node/.style={inner sep=0.25em} }}

\NewDocumentCommand{\LinePic}{ O{} O{} O{1} }{ % {{{
  \begin{tikzpicture}[myStyle, scale=#3*1 ]
    \node (center) at (0,0.5) {\phantom{$\cdot$}}; % black magic
    \path (0,0)  node (s) {$#1$}
        ++(0,1)  node (n) {$#2$};
    \draw (n) -- (s);
  \end{tikzpicture}
}  % }}}

\newcommand{\SquareUnwrapped}[4]{ % {{{
  \node (center) at (0.5,-0.5) {\phantom{$\cdot$}}; % black magic
  \path (0,0)  node (nw) {$#2$}
      ++(1,0)  node (ne) {$#4$}
      ++(0,-1) node (se) {$#3$}
      ++(-1,0) node (sw) {$#1$};
  \draw (nw) -- (ne) -- (se) -- (sw) -- (nw);
}  % }}}
\NewDocumentCommand{\SquareXY}{ O{} O{} O{} O{} O{1} O{1} }{ % {{{
  \begin{tikzpicture}[myStyle, xscale=#5*1, yscale=#6*1 ]
    \SquareUnwrapped{#1}{#2}{#3}{#4}
  \end{tikzpicture}
}  % }}}
\NewDocumentCommand{\Square}{ O{} O{} O{} O{} O{1} }{ % {{{
  \SquareXY[#1][#2][#3][#4][#5][#5]
}  % }}}

\NewDocumentCommand{\SquareAxes}{ O{} O{} O{1} O{0} O{1} }{  % {{{
  \begin{tikzpicture}[myStyle, scale=#3*0.8] % scale to align with squares
    \node (center) at (0.5,0.5) {\phantom{$\cdot$}}; % black magic
    \draw (0,0) -- node[above]{$#1$} (1,0) node[right]{$#4$}
      (0,0) -- node[left]{$#2$} (0,1) node[above]{$#5$};
  \end{tikzpicture}
}  % }}}
\NewDocumentCommand{\CubeAxes}{ O{} O{} O{} O{1} O{0} O{1} O{2} }{  % {{{
  \begin{tikzpicture}[myStyle, scale=#4*0.85] % scale to align with cubes
    \node (center) at (0.5,0.75) {\phantom{$\cdot$}}; % black magic
    \draw (0,0) -- node[above]{$#1$} (1,0) node[right]{$#5$}
      (0,0) -- node[left]{$#2$} (0,1) node[above]{$#6$}
      (0,0) -- node[below left=-0.25em]{$#3$} (0.5,-0.5) node[below right=-0.2em]{$#7$};
  \end{tikzpicture}
}  % }}}

\newcommand{\CubeNodes}[8]{  % {{{
  \node at (0.75,-0.75) (center) {\phantom{$\cdot$}}; % black magic
  \path (0,0)  node (back_nw)      {$#2$}
      ++(1,0)  node (back_ne)      {$#4$}
      ++(0,-1) node (back_se)      {$#3$}
      ++(-1,0) node (back_sw)      {$#1$}
        (0.5,-0.5) node (front_nw) {$#6$}
      ++(1,0)  node (front_ne)     {$#8$}
      ++(0,-1) node (front_se)     {$#7$}
      ++(-1,0) node (front_sw)     {$#5$};
}  % }}}
\newcommand{\CubeUnwrapped}[8]{ % {{{
  \CubeNodes{#1}{#2}{#3}{#4}{#5}{#6}{#7}{#8}
  \draw (back_nw) -- (back_ne) -- (back_se) -- (back_sw) -- (back_nw)
    (front_nw) -- (front_ne) -- (front_se) -- (front_sw) -- (front_nw)
    (back_nw) -- (front_nw)
    (back_ne) -- (front_ne)
    (back_se) -- (front_se)
    (back_sw) -- (front_sw);
}  % }}}
\newcommand{\CubeDUnwrapped}[8]{ % {{{
  \CubeNodes{#1}{#2}{#3}{#4}{#5}{#6}{#7}{#8}
  \draw (front_nw) -- (front_ne) -- (front_se) -- (front_sw) -- (front_nw)
    (back_nw) -- (back_ne) (back_sw) -- (back_nw)
    (back_nw) -- (front_nw)
    (back_ne) -- (front_ne)
    (back_sw) -- (front_sw);
  \draw[densely dotted] (back_ne) -- (back_se) -- (back_sw)
    (back_se) -- (front_se);
}  % }}}
\NewDocumentCommand{\Cube}{ O{} O{} O{} O{} O{} O{} O{} O{} O{1} }{  % {{{
  \begin{tikzpicture}[myStyle, scale=#9*1 ]
    \CubeUnwrapped{#1}{#2}{#3}{#4}{#5}{#6}{#7}{#8}
  \end{tikzpicture}
}  % }}}

\NewDocumentCommand{\CubeDeep}{ O{} O{} O{} O{} O{} O{} O{} O{} O{1}  }{  % {{{
  \begin{tikzpicture}[myStyle, xscale=#9*1, yscale=1.5 ]
    \CubeUnwrapped{#1}{#2}{#3}{#4}{#5}{#6}{#7}{#8}
  \end{tikzpicture}
}  % }}}
\NewDocumentCommand{\CubeD}{ O{} O{} O{} O{} O{} O{} O{} O{} O{1} }{  % {{{
  \begin{tikzpicture}[myStyle, scale=#9*1 ]
    \CubeDUnwrapped{#1}{#2}{#3}{#4}{#5}{#6}{#7}{#8}
  \end{tikzpicture}
}  % }}}

\NewDocumentCommand{\DeltaZeroCubeD}{ O{} O{} O{} O{} O{} O{} O{} O{} O{1} }{  % {{{
  \begin{tikzpicture}[myStyle, scale=#9*1]
    \CubeDUnwrapped{#1}{#2}{#3}{#4}{#5}{#6}{#7}{#8}
    \draw (back_sw)  to[out=30,in=180-30] (back_se)
      (back_nw)  to[out=30,in=180-30] node[auto]{$\delta$} (back_ne)
      (front_sw) to[out=30,in=180-30] (front_se);
    \draw[dashed] (front_nw) to[out=30,in=180-30] (front_ne);
  \end{tikzpicture}
}  % }}}

\NewDocumentCommand{\DeltaOneCubeD}{ O{} O{} O{} O{} O{} O{} O{} O{} O{1} }{  % {{{
  \begin{tikzpicture}[myStyle, scale=#9*1]
    \CubeDUnwrapped{#1}{#2}{#3}{#4}{#5}{#6}{#7}{#8}
    \draw (back_sw)  to[out=120,in=240]node[left]{$\delta$} (back_nw)
      (back_se)  to[out=120,in=240]  (back_ne)
      (front_sw) to[out=120,in=240] (front_nw);
  \end{tikzpicture}
}  % }}}
\NewDocumentCommand{\DeltaTwoCubeD}{ O{} O{} O{} O{} O{} O{} O{} O{} O{1} }{  % {{{
  \begin{tikzpicture}[myStyle, scale=#9*1]
    \CubeDUnwrapped{#1}{#2}{#3}{#4}{#5}{#6}{#7}{#8}
    \draw (back_sw) to[out=180+30,in=180] node[auto,swap]{$\delta$} (front_sw)
      (back_se) to[out=0,in=30] (front_se)
      (back_nw) to[out=180+30,in=180] (front_nw);
    \draw[dashed] (back_ne) to[out=0,in=30] (front_ne);
  \end{tikzpicture}
}  % }}}

\DeclareMathOperator{\M}{\text{M}}  % for matrix relations
\begin{document}
% title and abstract {{{1
\title{Supernilpotent Taylor Algebras are Nilpotent}
\author{ Andrew Moorhead}

\address[Andrew Moorhead]{
  Department of Mathematics;
  Vanderbilt University;
  Nashville, TN;
  U.S.A.}
\email[Andrew Moorhead]{andrew.p.moorhead@vanderbilt.edu}

\thanks{ This work was supported by the National Science
  Foundation grant no.\ DMS 1500254 and the Austrian Science Fund (FWF):P29931}

\begin{abstract}
We develop the theory of the higher commutator for Taylor varieties. A new higher commutator operation called the hypercommutator is defined using a type of invariant relation called a higher dimensional congruence. The hypercommutator is shown to be symmetric and satisfy an inequality relating nested terms. For a Taylor algebra the term condition higher commutator and the hypercommutator are equal when evaluated at a constant tuple, and it follows that every supernilpotent Taylor algebra is nilpotent. We end with a characterization of  congruence meet-semidistributive varieties in terms of the  neutrality of the higher commutator.

\end{abstract} 
\maketitle % }}}1

\section{Introduction} \label{sec:intro} % {{{1

In this article we study centrality conditions for general algebras. Our goal is to further develop the theory of a congruence lattice operation called the higher commutator, which is a higher arity generalization of the binary commutator. Higher commutator operations are significant because they are used to detect structure that cannot be described with nested binary commutators. An important example is the distinction between nilpotence, which is a condition that is defined using nested binary commutators, and supernilpotence, which is a condition defined using the higher commutator. Until recently, it was not known if supernilpotent algebras are necessarily nilpotent. The answer in general is no \cite{mooremoorhead}. However, we prove here that if a supernilpotent algebra satisfies a nontrivial idempotent equational condition, then the answer is yes. An interesting byproduct of the proof we give is an elementary theory of what we call a higher dimensional congruence.

% In the seventies, Smith discovered a language independent congruence lattice operation for Mal'cev varieties of algebras that interprets as the classical commutator in each of many well known classes, e.g.\ groups, rings, and Lie algebras \cite{jdhsmith}. Smith defined the commutator using the language of category theory. An aspect of the structure enforced by a Mal'cev operation is that any `reasonable' definition of the commutator will reproduce Smith's commutator. Hagemann and Herrmann extended these ideas to the context of modular varieties. Their development was purely algebraic \cite{hh} and led to the definition of what is now called the \emph{term condition} commutator. 
 
We begin with a broad outline of the ideas underlying the results of this paper. In 1954 Mal'cev \cite{malcevterm} observed that a variety of algebras $\var$ has permuting congruences exactly when there is a $\var$-term $q$ satisfying the identities
\[
q(x,x,y) \approx q(y,x,x) \approx x.
\]
A term satisfying these identities is called a Mal'cev operation. His discovery initiated a continuing line of research into the relationship between algebraic structure and equational conditions. Indeed, many important structural features (e.g.,\ congruence modularity, congruence distributivity, etc.) are now known to be enforced by equational laws. The strength of a particular condition may be measured by its position in what is called the lattice of interpretability types of varieties \cite{garciataylor}. The collection of all idempotent equational conditions forms a sublattice of the interpretability lattice. Taylor observed \cite{taylortermhomlaws} that any idempotent variety that does not interpret into the bottom element of this sublattice must have a term satisfying a certain generic package of identities (see the beginning of Section \ref{sec:higherarities}). Such a term is called a Taylor term. 

Taylor terms have received a lot of attention recently because of their connection to the Constraint Satisfaction Problem. The CSP Dichotomy Conjecture has been independently confirmed by Bulatov \cite{bulatovdich} and Zhuk \cite{zhukdich}. Roughly, each proof demonstrates that if the algebra of operations that preserve a set of finitary relations $\mathcal{R}$ has a Taylor term, then there is a polynomial time algorithm that decides the CSP for $\mathcal{R}$. Using some of algebraic theory that came out of investigating the CSP, Ol\v{s}\'ak recently proved that any package of Taylor identities force the existence of a particular $6$-ary Taylor term. The results of this article establish that the condition of having a Taylor term has strong consequences for the behavior of higher commutators.

The commutator establishes a useful connection between the possible configurations of an algebra's invariant relations and its clone of polynomial operations. Smith was the first to articulate such a connection. Using the language of category theory, he developed a signature independent commutator for Mal'cev varieties that interprets as the classical commutator in each of many well known classes, e.g.,\ groups, rings, and Lie algebras \cite{jdhsmith}. Smith's idea is a particularly nice example of the kind of insight a study of general algebra provides. The basic operations of an algebra can be thought of as instructions for building structure, and the same structure can be produced in different ways (e.g.,\ a group can be specified in the standard way or as an algebra with a single division operation.) The invariant relations of an algebra are indifferent to the manner in which they are generated, and therefore are the natural place to look for a structural definition of centrality. The language specific definitions of abelianness, solvability, and nilpotence for a particular variety are then consequences of this definition. The success of this viewpoint is demonstrated by Herrmann's celebrated classification of the abelian algebras belonging to a modular variety as exactly those algebras that are polynomially equivalent to a module \cite{heraffine}.

Hagemann and Herrmann were the first to extend Smith's commutator beyond the domain of Mal'cev varieties. Their development avoided category theory \cite{hh} and led to the definition of what is now called the \emph{term condition} commutator. While the term condition is independent of signature, it is nevertheless a syntactic condition. Freese and McKenzie study commutators for modular varieties in \cite{fm}. One of their early conclusions is that all `reasonable' definitions of a commutator for a modular variety are equivalent, and the remainder of the theory developed in the text favors the term condition commutator. A contrasting development of the modular commutator is found in Gumm's book \cite{gumm}, where the development of the modular commutator is guided by geometrical intuition. 

The term condition commutator for a variety that is not modular need not be symmetric, and it follows that two different centrality conditions that are equivalent in the modular case are not equivalent in general. Much is known in spite of this difficulty. In \cite{kearnesszendreirel}, Kearnes and Szendrei prove that the symmetric term condition commutator is equal to the linear commutator for a Taylor variety, and they use this equivalence to prove that any abelian Taylor algebra is polynomially equivalent to a subalgebra of a reduct of a module. We refer the reader to the monograph of Kearnes and Kiss \cite{kearneskiss} for a thorough treatment of the nonmodular binary commutator. 

Bulatov defines a higher arity generalization of the term condition commutator in \cite{buldef}. While for groups and rings Bulatov's higher commutator is term definable from the binary commutator, for other Mal'cev algebras it is not (for example, different expansions of a group may share congruences and binary commutators, but have different higher commutator operations). In
\cite{aichmud}, Aichinger and Mudrinksi develop analogues of those
properties shown to be essential for the binary commutator for the higher
commutator in a Mal'cev variety. In the same paper the higher commutator is used to define a special subclass of nilpotent Mal'cev algebras, which they call \emph{supernilpotent} Mal'cev algebras. Using earlier results of Kearnes \cite{smallfreespec}, they go on to show that the finite members of this class are exactly those algebras that are the product of prime power order nilpotent Mal'cev algebras. Supernilpotence has important connections to the free spectrum of an algebra (see for example \cite{aichfreespec}) and the equation solvability problem (see \cite{IdziakKrz} and\cite{KompatscherSupNil}.)  Equation solvability and related problems emphasize
the need to understand the differences between nilpotence and
supernilpotence.

In \cite{orsalrel}, Opr{\v s}al develops properties of the higher commutator in Mal'cev varieties by establishing a connection between the term condition and certain invariant relations. The theory of the higher commutator has been recently extended to varieties
that are not Mal'cev. In \cite{moorheadHC}, the author extends most
of the theory of the higher commutator to congruence modular varieties.  In
\cite{delta3v}, the author develops a relational description of the
modular ternary commutator and uses this to show that $(2)$-step
supernilpotence implies $(2)$-step nilpotence in a congruence modular
variety. In Wires \cite{Wires}, several properties of higher commutators are
developed outside of the context of congruence modularity. Implicit in the
results of Wires is that supernilpotence implies nilpotence for congruence
modular varieties. More recently, Kearnes and Szendrei have announced that
any \emph{finite} supernilpotent algebra is nilpotent, which is to appear in \cite{finitesupnil}.

In the context of current research into the properties of supernilpotent algebras, the main contribution of this article is indicated by its title. However, the machinery that is developed contributes something to the discussion of what a `good' notion of centrality is. In view of the approach to commutator theory taken here, the term condition can be thought of as a local method to check a global condition corresponding to the hypercommutator. This can be compared to the relationship between a tolerance and a congruence. In a Mal'cev variety, these two kinds of relations are the same, but in general one must take the transitive closure of a tolerance to produce a congruence. This is the $(1)$-dimensional instance of the main idea in this article, which is to extend the notion of transitive closure to a relation that is coordinatized by a hypercube. The success of this local to global principle is determined by the identities in the variety to which it is applied.

The paper is structured as follows. In Section \ref{sec:notation} we state some basic definitions and develop enough machinery to define two commutators, which are
\begin{enumerate}
\item the \emph{term condition commutator}, which is written as
$[\cdot, \dots, \cdot]_{TC} $, and 
\item the \emph{ hypercommutator}, which is written as
$[\cdot, \dots, \cdot]_H.$
\end{enumerate}
In Section \ref{sec:higherarities}, we prove the two main components of the proof that supernilpotent Taylor algebras are nilpotent. We call these components
\begin{enumerate}[leftmargin=2cm]
\item[H=TC:] 
$
[ \theta, \dots, \theta]_H = [\theta, \dots, \theta]_{TC},
$
 where $\theta$ is a congruence of a Taylor algebra $\A$, and 
 \item[HHC8:] for any algebra $\A$, $[\theta_0, \dots, \theta_{m-1}[\theta_m, \dots, \theta_{n-1}]_H]_H \leq [\theta_0, \dots, \theta_{n-1}]_H$, where $(\theta_0, \dots, \theta_{n-1})\in \Con(\A)^n$ (cf.\ HC8 in \cite{aichmud}.)
\end{enumerate}
Section \ref{sec:lowaritycase} is included to illustrate the proof method for few dimensions, and Section \ref{sec:sdmeet} examines the behavior of the hypercommutator in a congruence meet-semidistributive variety.

\section{Basic Concepts}\label{sec:notation}

\subsection{Notation}

We use the following basic notations. It is convenient for us to always think of the natural numbers as the set of all finite ordinals ordered by set membership. This means we will usually write $ i \in n$ instead of $i < n$ and $n$ instead of $\{0, \dots, n-1\}$. We will usually use the notation $f \in B^A$ to indicate that $f \subseteq A \times B$ is a function. We will often (but not always) use subscript notation to indicate images of functions:

\[ \langle a, b \rangle \in f \iff f_a =b.\]
If $Q \subseteq A $, then $f\vert_{Q}$ is the notation we use to restrict $f$ to $Q$.
In case the domain of a function is an interval of natural numbers $\{m, m+1, \dots, n-1\}$ we will also write a function $f \in A^{n\setminus m}$ as the tuple $(f_m, f_{m+1}, \dots, f_{n-1})  $.
\subsection{Cubes}

Let $n \geq 0$. One of the basic objects we study here are relations of arity $2^n$. Such relations inherit the structure of an $n$-dimensional cube. This viewpoint allows us to articulate structural properties that would otherwise remain obscure if we considered relations of arity $2^n$ as unstructured tuples.  

More generally, let $S \finsub \nat$ be a finite set of cardinality $n \geq 0$. An \textbf{$(n)$-dimensional cube} is the graph with vertices belonging to the set of functions $2^S$, with two functions $f,g \in 2^S$ connected by an edge when there is exactly one $i\in S$ such that $f(i) \neq g(i)$. So, a $(0)$-dimensional cube is a single vertex, a $(1)$-dimensional cube is two vertices connected by an edge, and so on. 

%Now, $2^S$ is naturally identified with the universe of the boolean algebra of subsets of $S$. We will make use of the symmetric difference of two subsets, which when identified with indicator functions is defined as 
%\begin{align*}
%+: (2^S)^2 &\to 2^S\\
%(f,g) &\mapsto \{ \langle i, f_i +_2 g_i \rangle : i \in S \}.
%\end{align*}
We name some constants of $2^S$. Denote by $\textbf{i}$ the indicator function that takes value one for $i \in S$ and zero elsewhere. Also, denote by $\textbf{1}$ the function that takes constant value $1$ and $\textbf{0}$ the function that takes constant value $0$. It should be clear what the domain is for these constants from the context in which they are used. 

Now let $A$ be a nonempty set. Formally, every $\gamma \in A^{2^S}$ is a collection of pairs and this collection of pairs inherits the structure of a $(|S|)$-dimensional cube. That is, let

$$ G(\gamma) = \langle \{ \langle f, \gamma_f \rangle  : f\in 2^S \} ; E \rangle $$
be the graph with vertex set $\gamma$, where $ \langle f, \gamma_f \rangle $ is connected by an edge to $\langle g, \gamma_g \rangle $ if and only if $f$ and $g$ are connected in $2^S$. We will call such a graph a \textbf{labeled $(|S|)$-dimensional cube}. The $(|S|)$-dimensional cube $2^S$ is a \textbf{coordinate system} for $\gamma$ and the value $\gamma_f$ is called the \textbf{label} of the function $f \in 2^S$. We will usually not be so formal and refer to $\gamma$ instead of $G(\gamma)$. We denote by \textbf{$\gamma$-pivot} the vertex label $\gamma_{\textbf{1}}$. All other vertex labels are called \textbf{$\gamma$-supporting}. Sometimes we call the $\gamma$-supporting vertex label $\gamma_\textbf{0}$ the \textbf{$\gamma$-antipivot}.

By elementary properties of exponents we may decompose any vertex labeled $(|S|)$-dimensional cube into a cube of cubes. That is, let $Q \subseteq S$ and define the map 

\begin{align*} 
\cut_Q: A^{2^S} &\to(A^{2^{S\setminus Q} } )^{2^Q}\\
\gamma &\mapsto \bigg\{ \big\langle f , \{\langle g , \gamma_{ g \cup f} \rangle : g \in 2^{S\setminus Q} \} \big\rangle  : f \in 2^Q \bigg\}.
\end{align*}
So, $\cut_Q(\gamma)$ is a labeled $(|Q|)$-dimensional cube, where each vertex is labeled by a labeled $(|S \setminus Q|)$-dimensional cube which is called a \textbf{$Q$-cut} of $\gamma$.

It is easy to see that $\cut_Q$ has an inverse, which is defined as 
\begin{align*} 
\glue_Q:(A^{2^{S\setminus Q} } )^{2^Q} &\to A^{2^S} \\
\gamma &\mapsto \bigg\{ \big\langle f , (\gamma_{f\vert_Q})_{f\vert_{S\setminus Q}} \big\rangle  : f \in 2^S \bigg\}. 
\end{align*}
Therefore, every labeled $(n)$-dimensional cube may be represented as a labeled cube of lower dimension, where the vertices of this lower dimensional cube are vertex labeled cubes, and every such cube of cubes may be `glued' back together. It is illustrative to draw pictures of these different representations and we provide some in Figure \ref{fig:4cube}. Note that the labels of some of the vertices are missing to improve readability.  

The $\cut_Q$ with $Q$ such that $|Q| = 1$ or $|S \setminus Q | = 1,2$ are used often enough to merit names: 

\begin{enumerate}
\item $\cut_{\{i\}} $ is called $\faces_i$,
\item $\cut_{S \setminus \{i\} } $ is called $\lines_i$, and
\item $\cut_{S \setminus \{i,j\}}$ is called $\squares_{i,j}$.

\end{enumerate}

\begin{figure}
\centering
\includegraphics[scale=.8]{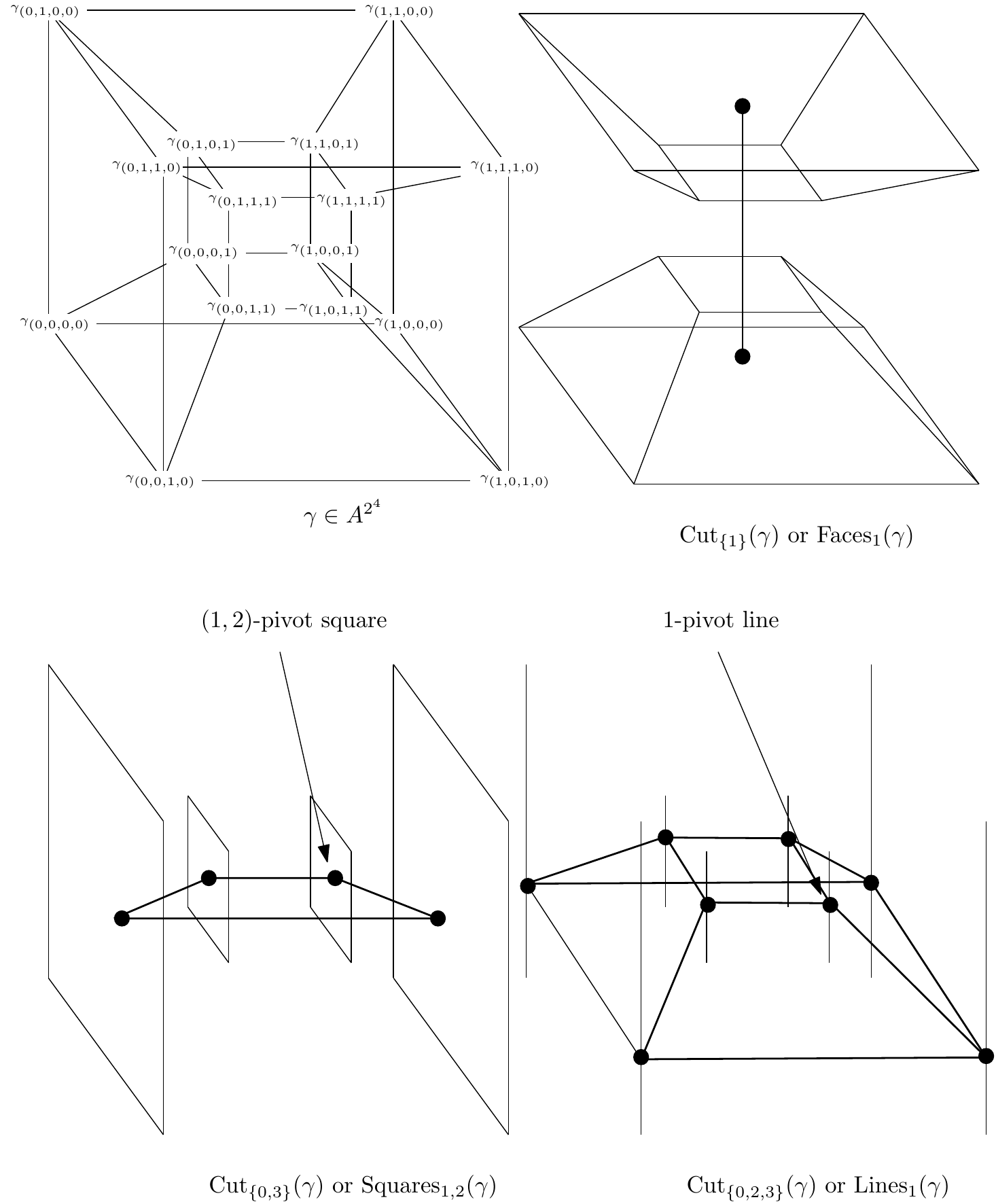}
\caption{Equivalent representations of a labeled $4$-dimensional cube.}\label{fig:4cube}
\end{figure}

Now, let $A$ be a nonempty set and let $R \subseteq A^{2^S}$. In this situation we say that $R$ is a \textbf{$(|S|)$-dimensional relation}. The $(|S|)$-dimensional cube is a coordinate system for $R$ and we think of the elements belonging to $R$ as labeled $(|S|)$-dimensional cubes. If $f \in 2^S$, we use $R_f$ to denote the set $\{\gamma_f: \gamma \in R\}$. 

To make the notation less cumbersome, we adopt the following convention. If $S,T \finsub \nat$ with $|S| = |T|$, then $2^S$ and $2^T$ are isomorphic coordinate systems in the sense  any bijection from $S$ onto $T$ lifts to a graph isomorphism from $2^S$ onto $2^T$. We will often make use of this fact without mentioning it explicitly. 

For example, for $i \in S$ let

$$ \faces_i(R) = \{ \faces_i(\gamma) : \gamma \in R \}$$
be the image of $R$ under $\faces_i$. Now, $\faces_i(R) \subseteq (A^{2^{S \setminus \{i \}}})^{2^{\{i\}}}$ and there is an obvious bijection between $2^{\{i\}}$ and $2$. Therefore, we informally treat $\faces_i(R)$ as a binary relation on the set $A^{2^{S \setminus \{i \}}}$. In this case we will use a superscript to specify a face, i.e.\
\[
\faces_i^0(\gamma) = \faces_i(\gamma)_{\langle i, 0 \rangle}
\qquad \text{and} \qquad
\faces_i^1(\gamma) = \faces_i(\gamma)_{\langle i,1 \rangle}.
\]

Similarly, let $ \gamma \in R$. Because $\lines_i(\gamma) \in (A^{2^{\{i\}}})^{2^{ S\setminus \{i \}}}$, we informally treat it as a $(|S|-1)$-dimensional cube with vertices labeled by elements of $A^2$. We will sometimes refer to the vertex labels of $\lines_i(\gamma)$ as \textbf{$(i)$-cross section lines}. In this case we call $\lines_i(\gamma)_\textbf{1}$ the \textbf{$(i)$-pivot line}, we call $\lines_i(\gamma)_\textbf{0}$ the \textbf{$(i)$-antipivot line}, and we call any $\lines_i(\gamma)_f$ an \textbf{$(i)$-supporting line} when $f \neq \textbf{1}$. 

Continuing along these lines, every $\squares_{i,j}(\gamma) \in (A^{2^{\{i,j\}}})^{2^{S\setminus \{i, j \}}}$ may be treated as a $(|S|-2)$-dimensional cube with vertices labeled by elements of $A^{2^2}$.  We will sometimes refer to the vertex labels of $\squares_{i,j}(\gamma)$ as \textbf{$(i,j)$-cross section squares}. We call $\squares_{i,j}(\gamma)_\textbf{1}$-pivot the \textbf{$(i,j)$-pivot square} of $\gamma$, we call $\squares_{i,j}(\gamma)_\textbf{0}$ the \textbf{$(i,j)$-antipivot square}, and we call $\squares_{i,j}(\gamma)_f$ an \textbf{$(i,j)$-supporting square} when $f\neq \textbf{1}$. 

\textbf{Important convention:} Whenever we draw a square belonging to $\squares_{i,j}(R)$, it is always oriented like this picture of $2^2$
$
\Square[(0,0)][(0,1)][(1,0)][(1,1)][1.3],
$
along with the convention that $i$ corresponds to $ 0 \in 2$ and $j$ corresponds to $1 \in 2$. According to this scheme, a picture of an element in $\squares_{i,j}(R)$ is the transpose of a picture of an element in $\squares_{j,i}(R)$.

\subsection{Higher Dimensional Congruence Relations}

\begin{defn}
Let $B$ be a nonempty set and let $R \subseteq B^2$ be a binary relation on $B$. We say that $R$ is a quasiequivalence relation on $B$ provided that each of the following conditions hold:
\begin{enumerate}
\item $\langle a, b \rangle \in R$ implies $\langle a,a \rangle, \langle b,b \rangle \in R$ (quasireflexivity),
\item $\langle a,b \rangle \in R$ if and only if $\langle b,a \rangle \in R$. (symmetry), and
\item $\langle a,b \rangle, \langle b,c \rangle \in R$ imply that $\langle a,c \rangle \in R$ (transitivity).
\end{enumerate} 
\end{defn}

\begin{defn}\label{def:highercon}
Let $\A $ be an algebra with underlying set $A$ and let $R \subseteq A^{2^S}$ be a $(|S|)$-dimensional relation for some $S\finsub \nat$.
\begin{enumerate}
\item  $R$ is said to be \textbf{$(S)$-reflexive}, \textbf{$(S)$-symmetric}, or \textbf{$(S)$-transitive} if \newline $\faces_i(R)$ is respectively quasireflexive, symmetric, or transitive on $A^{2^{S\setminus \{i\}}}$ for each $i \in S$.

\item $R$ is said to be a \textbf{$(|S|)$-dimensional equivalence relation} provided \newline $\faces_i(R)$ is a quasiequivalence relation on $A^{2^{S\setminus \{i\}}}$ for each $i \in S$. 
\item $R$ is said to be a \textbf{$(|S|)$-dimensional congruence} of $\A$ if it is a $(|S|)$-dimensional equivalence that is also compatible with the basic operation of $\A$.
\item $R$ is said to be a \textbf{$(|S|)$-dimensional tolerance} of $\A$ if it is $(S)$-reflexive, $(S)$-symmetric, and compatible with the basic operations of $\A$.

\end{enumerate}

\end{defn}

The higher dimensional versions of reflexivity and symmetry can be described in terms of certain unary operations. Let $A$ be a nonempty set, $S \finsub \nat$, and $j\in 2$. For each $i \in S$, we define the maps $\refl_i^j : A^{2^S} \to A^{2^S}$ and $\sym_i^j: A^{2^S} \to A^{2^S}$
by 
\begin{align*}
\refl_i^j(h) &= \glue_{\{i\}}(\langle \faces_i^j, \faces_i^j \rangle ) \text{ and }\\
\sym_i(h) &= \glue_{\{i\}}(\langle \faces_i^1, \faces_i^0 \rangle).
\end{align*}
The following lemma is an easy consequence of the definitions. 

\begin{lem}\label{lem:reflsymclosure}
Let $A$ be a nonempty set and $S \finsub \nat$. Let $R \subseteq A^{2^S}$ be a $|S|$-dimensional relation. The following hold:
\begin{enumerate}
\item $R$ is $(S)$-reflexive if and only if $R$ is closed under $\refl_i^j$ for all $(i,j) \in S \times 2$, and
\item $R$ is $(S)$-symmetric if and only if $R$ is closed under $\sym_i$ for every $i \in S$.

\end{enumerate}

\end{lem}

 We observed earlier that any vertex labeled $(|S|)$-dimensional cube can be interpreted as a cube of cubes. Such an interpretation may be used to formulate weaker versions of higher dimensional symmetry, reflexivity and transitivity. The following lemma makes this precise. The proof, which involves a direct application of the definitions, is left to the reader. 
 
 \begin{lem}\label{lem:higherconofhighercon}
Let $A$ be a nonempty set and $S\finsub\nat$. Let $R \subseteq A^{2^S}$ and suppose $Q \subseteq S$. Each of the following implications holds.

\begin{enumerate}
\item If $R$ is $(S)$-symmetric, then $\cut_Q(R)$ is $(Q)$-symmetric.
\item If $ R$ is $(S)$-reflexive, then $\cut_Q(R)$ is $(Q)$-reflexive.
\item If $R$ is $(S)$-transitive, then $\cut_Q(R)$ is $(Q)$-transitive. 
\end{enumerate}
\end{lem} 
 
\begin{cor}\label{lem:hyperreflexive}
Let $A$ be a nonempty set, $S \finsub \nat$, and $R \subseteq A^{2^S}$ be a $(|S|)$-dimensional relation. Let $\gamma \in R$, $Q \subseteq S$ and $f \in 2^Q$. Let $\alpha \in  A^{2^S}$ be defined by
$
\cut_Q(\alpha)_g = \cut_Q(\gamma)_f,
$
for every $g \in 2^Q$.
If $R$ is $(S)$-reflexive, then $\alpha \in R$.
\end{cor}

\begin{proof}
Suppose $S = \{i_0, \dots, i_{n-1} \}$ for $n = |S|$. We first prove the lemma in the special case when $Q = S$. In this case $\cut_Q(R) = R$. Let $f \in 2^S$. The lemma is asserting that $\alpha \in R$, where $\alpha \in 2^S$ is defined by $\alpha_g = \gamma_f$ for all $g \in 2^S$. Indeed, it is clear that 
\[
\alpha = \refl_{i_0}^{f_{i_0}}(\refl_{i_1}^{f_{i_1}}(\dots (\refl_{i_{n-1}}^{f_{i_{n-1}}}(\gamma)\dots)).
\]
Because $R$ is assumed to be $(S)$-reflexive, it follows from Lemma \ref{lem:reflsymclosure} that $\alpha \in R$. 

For the general case we apply the special case we just handled to the situation where $A' =  A^{2^{S\setminus Q}}$,  $S' = Q$, $R' = \cut_Q(R)$, and $\gamma ' = \cut_Q(\gamma)$.  Now let $\alpha' \in (A')^{2^Q}$ be defined by $(\alpha')_g = (\gamma')_f$. We suppose that $R$ is $(S)$-reflexive, so Lemma \ref{lem:higherconofhighercon} shows that $R'$ is $(Q)$-reflexive. All of the assumptions we made in the special case are satisfied, so we conclude that $\alpha ' \in R'$. Because $\alpha' = \cut_Q(\alpha)$, we have shown that $\cut_Q(\alpha) \in \cut_Q(R)$, or equivalently, that $\alpha \in R$.

\end{proof}

The properties of $(S)$-symmetry, reflexivity, and transitivity are each preserved by projecting onto a set of coordinates that determines a lower dimensional cube. This feature, which is made precise in the next lemma, is in a sense dual to the situation described in Lemma \ref{lem:higherconofhighercon}. 

\begin{lem}\label{lem:projectingontosubcubes}
Let $A$ be a nonempty set and $S\finsub\nat$. Let $R \subseteq A^{2^S}$ and suppose $Q \subseteq S$. Take $f \in 2^Q$. Each of the following implications holds.
\begin{enumerate}
\item If $R$ is $(S)$-symmetric, then $\cut_Q(R)_f$ is $(S\setminus Q)$-symmetric.
\item If $ R$ is $(S)$-reflexive, then $\cut_Q(R)_f$ is $(S \setminus Q)$-reflexive.
\item If $R$ is $(S)$-transitive and $(S)$-reflexive, then $\cut_Q(R)_f$ is $(S \setminus Q)$-transitive. 
\end{enumerate}

\end{lem}

\begin{proof}
The proof of \emph{(1)} and \emph{(2)} is left to the reader. We prove \emph{(3)}. Suppose the conditions of the lemma and \emph{(3)} hold and let $\gamma, \lambda \in \cut_Q(R)_f$ be such that $\faces_i^1(\gamma) = \faces_i^0(\lambda)$ for some $i \in S \setminus Q$. We show that $\glue_{\{i\}}(\langle \faces_i^0(\gamma), \faces_i^1(\lambda) \rangle ) \in \cut_Q(R)_f$. Let $\alpha, \beta \in A^{2^S}$ be defined by
$
\cut_Q(\alpha)_g = \gamma
$ and  
$
\cut_Q(\beta)_g = \lambda
$, for all $g \in 2^{Q}$.
Applying Corollary \ref{lem:hyperreflexive} to this situation shows that $ \alpha, \beta \in R$. 

We claim that $\faces_i^1(\alpha) = \faces_i^0(\beta)$. Indeed, let $h \in 2^{S\setminus \{i\}}$. We can decompose $h$ as the union of two partial functions $h' \in 2^Q$ and $h'' \in 2^{S \setminus Q \union \{i\}}$. The computation  

\begin{align*}
\faces_i^1(\alpha)_h 
&= \alpha_{h' \union h'' \union \langle i, 1\rangle }\\
&= (\cut_Q(\alpha)_{h'})_{h^{''} \union \langle i, 1\rangle }\\
&= \gamma_{h^{''} \union \langle i, 1\rangle }\\
&= \lambda_{h^{''} \union \langle i, 0\rangle }\\
&= (\cut_Q(\beta)_{h'})_{h^{''} \union \langle i, 0\rangle }\\
&= \beta_{h' \union h'' \union \langle i, 0\rangle }\\
&= \faces_i^0(\beta)_h 
\end{align*}
establishes our claim. 
Therefore, 
$\eta = \glue_{\{i\}}(\langle \faces_i^0(\alpha) , \faces_i^1(\beta)\rangle ) \in R$ and so $\cut_Q(\eta)_f \in \cut_Q(R)_f$. A computation similar to the one above shows that $\cut_Q(\eta)_f = \glue_{\{i\}}(\langle \faces_i^0(\gamma), \faces_i^1(\lambda) \rangle )$, as desired. 

\end{proof}

\begin{cor}\label{cor:deltaedgesarecongruences}
Let $\A$ be an algebra and $S \finsub \nat$. Let $R \leq \A^{2^S}$ be a $(|S|)$-dimensional tolerance of $\A$. For every $Q \subseteq S$, 
\begin{enumerate}
\item $\cut_Q(R)_f$ is a $(|S \setminus Q|)$-dimensional tolerance of $\A$, and 
\item $\cut_Q(R)_f = \cut_Q(R)_g$ for all $f, g \in 2^Q$.
\end{enumerate} 
Additionally, the same statement holds if the word `tolerance' is replaced by `congruence'.
\end{cor}

\begin{proof}
The first item \emph{(1)} of the lemma follows from \emph{(1)} and \emph{(2)} of Lemma \ref{lem:projectingontosubcubes}. To show \emph{(2)}, suppose $f, g \in 2^Q$ and take $\gamma \in \cut_Q(R)_f$. By Corollary \ref{lem:hyperreflexive}, there exists $\alpha \in R$ so that $\cut_Q(\alpha)_g = \gamma$. Therefore, $\cut_Q(R)_f \subseteq \cut_Q(R)_g$. The same argument shows that $\cut_Q(R)_g \subseteq \cut_Q(R)_g$.

If $R$ is assumed to be a $(|S|)$-dimensional congruence, then \emph{(3)} of Lemma \ref{lem:projectingontosubcubes} indicates that $\cut_Q(R)_f$ is also $(|S\setminus Q|)$-transitive for every $f \in 2^Q$. This establishes the final statement of the lemma. 
\end{proof}

\begin{defn}\label{def:higherconlattice}
Let $\A$ be an algebra. For each $S \finsub \nat$, set

\[ \Con_S(\A) \coloneqq \langle \{R \subseteq A^{2^S}: R \text{ is an $(|S|)$-dimensional congruence} \}; \join, \meet, 0, 1 \rangle
, \text{ where }\]
\begin{enumerate}

\item $R\meet T \coloneqq R \cap T$ 
 \item $R \join T \coloneqq \Meet \{Z: R \cup T \subseteq Z \}$.
 
\end{enumerate}

\end{defn}

It is an easy exercise to show that each of these lattices is algebraic. The definition we give contains many redundancies, because $\Con_Q(\A)$ and $\Con_S(\A)$ encode exactly the same information whenever $|Q| = |S|$. The reader may wonder why we do not instead use the canonical choice of coordinates which produces the following sequence of lattices:
\[
\Con_0(\A), \Con_1(\A), \dots, \Con_n(\A), \dots 
\]
Our choice is motivated by a wish to avoid changing coordinate systems when we consider nested commutator expressions. 

We remark that $\Con_1(\A)$ is different from $\Con(\A)$, because we require only quasireflexivity of our relations. This relaxation of reflexivity has the consequence that  $\Con_1(\A)$ contains all congruences of subalgebras of $\A$. The ordinary congruence lattice of $\A$ is isomorphic to the interval above the full diagonal relation in $\Con_1(\A)$. We also remark that $\Con_0(\A)$ is the lattice of subuniverses of $\A$ and that all of these lattices may have the empty relation as the least element in the event that $\A$ has no smallest subalgebra. There are some appealing extensions of classical results pertaining to congruences to higher dimensional congruences. Most notably, an $(n)$-dimensional equivalence relation of an algebra $\A$ is a compatible relation if and only if it is compatible with those $n$-ary polynomials of the subalgebra determined by its intersection with the diagonal of $A$ in $A^{2^n}$. These ideas will be presented in a companion article.

We now describe the generation of higher dimensional congruences. Take $S\finsub \nat$ (the case $|S| =0$ is generation of a subalgebra), and let $X\subseteq A^{2^S}$. We respectively define the $(|S|)$-dimensional congruence and $(|S|)$-dimensional tolerance of $\A$ generated by $X$ as

\begin{align*}
\Theta_S(X) &= \Meet \{R: R \text{ is a $(|S|)$-dimensional congruence and } X\subseteq R \}\\
\tol_S(X) &= \Meet \{R: R \text{ is a $(|S|)$-dimensional tolerance and } X\subseteq R \}.
\end{align*}

The notion of a transitive closure of a binary relation generalizes to higher dimensions in the obvious way. Suppose $S = \{i_0, \dots, i_{n-1}\} \finsub \nat$ for some $n \in \nat$. Let $Y \subseteq A^{2^S}$ be a $(|S|)$-dimensional relation. For $i \in S$ set 

\[ Y^{\circ_i } =  \glue_{\{i\}}( \faces_i(Y)^{\circ}),\]
where $\faces_i(Y)^{\circ}$ is the transitive closure of $\faces_i(Y)$ when interpreted as a binary relation. We recursively define

\begin{enumerate}
\item $\tc_{0}(Y) = Y^{\circ_{i_0}}$, and 
\item $\tc_{j}(Y) = \left(\tc_{j-1}(Y)\right)^{\circ_{i_{j\mod n}}}$, for $j > 0$.

\end{enumerate}
Finally, set $\tc(Y) = \Union_{j \in \nat} \tc_{i_j}(Y)$. The proof of the following proposition is left to the reader.

\begin{lem}\label{prop:hcongenerate} 
Let $\A$ be an algebra and $S \finsub \nat$. The following hold.
\begin{enumerate}
\item If $R$ is a $(|S|)$-dimensional tolerance of $\A$, then $R^{\circ_i}$ is a $(|S|)$-dimensional tolerance of $\A$ and $R \subseteq R^{\circ_i}$.
\item
$\Theta_S(X) = \tc(\tol_S(X))$, for all $X \subseteq A^{2^S}$.
\item 
$\cut_Q(\tc(X))_f \subseteq \tc(\cut_Q(X)_f)$, for all $X \subseteq A^{2^S}$, $Q \subseteq S$, and $f \in 2^Q$.
\end{enumerate}

\end{lem}

\subsection{Centrality Conditions}

We now use this machinery to develop the commutator theory for the congruences of an algebra. It is interesting to note that the scope of this theory could be enlarged to include all higher dimensional congruences. It is unclear if such a broad generalization of commutator theory has any practical application, so we limit our development to congruences.

In this section we define two centralizer conditions that are used to define two distinct higher arity commutators. The first is due to Bulatov and is a natural extension of the so-called term condition. The second is a new condition and is used to define what we call the hypercommutator. 

The definition of the $(n)$-ary commutator as formulated by Bulatov in \cite{buldef} can be restated as a condition on a certain $(n)$-dimensional invariant relation, elements of which are often referred to as matrices. We do not state the original definition here, but refer the reader to \cite{moorheadHC} for details on the equivalence between our definition of the term condition higher commutator and that given by Bulatov. 

Let $\A$ be an algebra, $S \finsub \nat$, and $ m \in |S|$. Corollary \ref{cor:deltaedgesarecongruences} associates to any $(|S|)$-dimensional congruence $R$ a collection of $(|S|-m)$-dimensional congruences indexed by the subsets of $S$ of cardinality $m$, i.e.\ the set
\[
P_m(R)=
\{ \cut_Q(R)_\textbf{1} : Q \in S^{[m]}\}.
\] 
For any such indexed set of higher dimensional congruences, i.e.
\[
\{T_Q \in \Con_{S \setminus Q}(\A) \}_{Q \in S^{[m]}},
\]
there exists (as can be easily verified) a maximal $(|S|)$-dimensional congruence
\[
\rect\left(\{T_Q \in \Con_{S \setminus Q}(\A) \}_{Q \in S^{[m]}} \right) \in \Con_S(\A)
\]
that satisfies 
\[
P_m\left
(\rect\left(\{T_Q \in \Con_{S \setminus Q}(\A) \}_{Q \in S^{[m]}} \right)
\right) = \{T_Q \in \Con_{S \setminus Q}(\A) \}_{Q \in S^{[m]}} .
\]

We call this maximal relation the \textbf{$\left(\{T_Q \in \Con_{S \setminus Q}(\A) \}_{Q \in S^{[m]}}\right)$-rectangles}.
In the special case that $m = |S|-1$, we have that 
$
P_{|S|-1}(R) = \{\lines_i(R)_\textbf{1}: i \in S \}
$
is an $S$-indexed family of $(1)$-dimensional congruences, and  $\rect(\left(\{\lines_i(R)_\textbf{1}: i \in S \}\right)$ is the $(|S|)$-dimensional congruence consisting of all those $\gamma \in A^{2^S}$ satisfying $\lines_i(\gamma)_f \in \lines_i(R)_\textbf{1}$, for all $i \in S$ and $f \in 2^{S\setminus \{i\}}$. If it is also the case that $S = \{l, \dots, n-1\} = n\setminus l$, we use the notation
\begin{align*}
(\theta_l, \dots, \theta_{n-1}) &\text{ instead of } 
P_{n-l-1}(R) \text{, and}\\
\rect(\theta_l, \dots, \theta_{n-1}) &\text{ instead of }
\rect(P_{n-l-1}(R)).
\end{align*}

We are still in the situation where $\A$ is an algebra and $S \finsub \nat$. Assume also that $|S|\geq 1$. For each $i \in S$ define 
$
\cube_i: A^2 \to A^{2^S}
$
by 
\[
\cube_i(x,y)_f = 
\begin{cases}
x & \text{if } f_i = 0 \text{, and}\\
y & \text{if } f_i =1.
\end{cases}
\]
From the context it should be clear what the dimension of $\cube_i(x,y)$ is.

\begin{defn}\label{def:tolmatdeltadefinition}
Let $\A$ be an algebra and $S \finsub \nat$ with $|S|\geq 1$. Let $\{\theta_i\}_{i \in S} \subseteq \Con(\A)$ be an $S$-indexed set of congruences. Set

\begin{align*}
M(\{\theta_i\}_{i \in S}) &= \tol_S\bigg( \Union_{i\in S} \cube_i(\theta_i)
 \bigg) \text{, and}\\
\Delta(\{\theta_i\}_{i \in S}) &= \Theta_S\bigg( \Union_{i\in S} \cube_i(\theta_i) \bigg).
\end{align*}
Notice that if $|S| =1$, then 
$
M(\{\theta_i\}_{i \in S}) = \Delta(\{\theta_i\}_{i \in S}),
$
and this relation is equal to $\theta_i$ (up to a trivial change of coordinates).
In case $S = \{m, \dots, n-1\} = n\setminus m$, we will use the notation $(\theta_m, \dots, \theta_{n-1})$, $M(\theta_m, \dots, \theta_{n-1})$, and $\Delta(\theta_m, \dots, \theta_{n-1})$ for $\{\theta_i\}_{i \in S}$, $M(\{\theta_i\}_{i \in S})$, and $\Delta(\{\theta_i\}_{i \in S})$, respectively.
\end{defn}
 
\begin{figure}[!ht]
\includegraphics[scale=1]{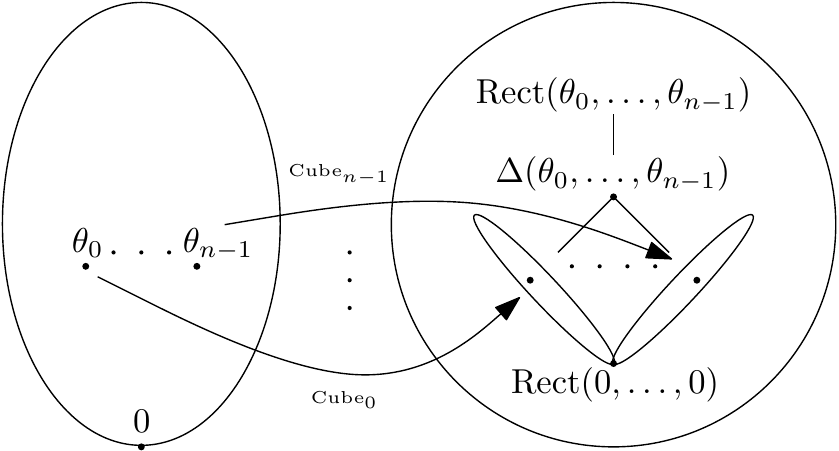}
\caption{Embeddings of $\Con(\A)$ into $\Con_n(\A)$}\label{fig:deltapicture}
\end{figure}

\begin{rk}
Let $n \geq 2$. For each $i \in n$, the map 
\begin{align*}
\CuBe_i: \mathcal{P}(A^{2}) &\to \mathcal{P}(A^{2^n})\\
R &\mapsto \cube_i(R)
\end{align*}
when restricted to $\Con(\A)$ is a lattice embedding into $\Con_n(\A)$. Denote the least congruence of $\A$ by $0$. Any two distinct such embeddings intersect only at their shared bottom element, which is $\rect(0, \dots, 0)$. See Figure \ref{fig:deltapicture} for a picture that shows the relationship between these embeddings and Definition \ref{def:tolmatdeltadefinition}.
\end{rk}

For historical reasons, we call $M(\{\theta_i\}_{i \in S})$ the \textbf{algebra of $\{\theta_i\}_{i \in S}$ matrices}. The following lemmas establish some basic properties of these two relations. Each statement is referring to the situation established in Definition \ref{def:tolmatdeltadefinition}.

\begin{lem}\label{lem:basicdelta0}
$ M(\{\theta_i\}_{i \in S}) \leq \Delta(\{\theta_i\}_{i \in S}) \leq \rect(\{\theta_i\}_{i \in S}).
$
\end{lem}
\begin{proof}
The first containment follows from the fact that any $(|S|)$-dimensional congruence is also a $(|S|)$-dimensional tolerance. The second containment follows from the observation that $\cube_i(\theta_i) \leq \rect(\{\theta_i\}_{i \in S})$ and that $\rect(\{\theta_i\}_{i \in S}) \in \Con_S(\A)$. 
\end{proof}

\begin{lem}\label{lem:basicdelta1}
$M(\{\theta_i\}_{i \in S}) = \Sg_{A^{2^S}}\bigg( \Union_{i\in S} \cube_i(\theta_i) \bigg).
$
\end{lem}
\begin{proof}
Because $\theta_i$ is a congruence for each $i \in S$, the relation $\Union_{i\in S} \cube_i(\theta_i)$ is both $(S)$-symmetric and $(S)$-reflexive. It follows that 
\[
\Sg_{A^{2^S}}\bigg( \Union_{i\in S} \{\cube_i(x,y) : \langle x,y \rangle \in \theta_i \} \bigg)
\]
is already an $(|S|)$-dimensional tolerance, and is therefore equal to $M(\{\theta_i\}_{i \in S})$.
\end{proof}

\begin{lem}\label{lem:basicdelta2}
 $\Delta(\{\theta_i\}_{i \in S}) = \tc(M(\{\theta_i\}_{i \in S}))$.

\end{lem}
\begin{proof}
This is an immediate consequence of Lemma \ref{prop:hcongenerate}.
\end{proof}

\begin{lem}\label{lem:deltaprojectstodelta}
For every $Q \subseteq S$ and $f \in 2^Q$,
\begin{enumerate}
\item 
$\cut_Q(M(\{\theta_i\}_{i\in S}))_f = M(\{\theta_i\}_{i \in S \setminus Q })$, and 
\item 
$\cut_Q(\Delta(\{\theta_i\}_{i\in S}))_f = \Delta(\{\theta_i\}_{i \in S \setminus Q })$.
\end{enumerate}
\end{lem}

\begin{proof}
We first notice that $\cut_Q(\text{ })_f$ commutes with the term operations of $\A$. That is,
$
\cut_Q(s(\gamma_0, \dots, \gamma_{\sigma(s)-1}))_f 
= s(\cut_Q(\gamma_0)_f, \dots, \cut_Q(\gamma_{\sigma(s)-1})_f)
$
for every term $s \in \Clo(\A)$ and $\gamma_0, \dots, \gamma_{\sigma(s)-1} \in \Delta(\{\theta_i\}_{i \in S})$. Furthermore, we compute 
\[
\cut_Q(\cube_i(x,y))_f =
\begin{cases}
\cube_i(x,y) & \text{ if $i \in S\setminus Q$,}\\
\text{constant cube with value $x$}       & \text{ if $i \in Q$ and $f_i =0$, and}\\
\text{constant cube with value $y$}        & \text{ if $i \in Q$ and $f_i =1$}.
\end{cases}
\]
To establish \emph{(1)}, we apply Lemma \ref{lem:basicdelta1} and conclude that 
\begin{align*}
\cut_Q(M(\{\theta_i\}_{i \in S}))_f &=
\cut_Q \left(\Sg_{A^{2^S}}\bigg( \Union_{i\in S} \cube_i(\theta_i) \bigg)
\right)_f\\
&= \Sg_{A^{2^{S\setminus Q}}}
\bigg( \Union_{i\in S} \cut_Q(\cube_i(\theta_i))_f \bigg)\\
&= \Sg_{A^{2^{S\setminus Q}}}
\bigg( \Union_{i\in S\setminus Q} \cube_i(\theta_i) \bigg)\\
&= M(\{\theta_i\}_{i \in S \setminus Q}).
\end{align*}

To establish \emph{(2)}, we show that each of the two relations contains the other. Suppose that $\gamma \in \Delta(\{\theta_i\}_{i \in S})$. It follows from Lemma \ref{lem:basicdelta2} that $\gamma \in \tc(M(\{\theta_i\}_{i \in S}))$. We now apply \emph{(3)} of Lemma \ref{prop:hcongenerate} and conclude that 

\begin{align*}
\cut_Q(\gamma)_f &\in 
\cut_Q(\tc(M(\{ \theta_i \}_{i \in S}))_f\\
&\subseteq \tc(M(\{\theta_i\}_{i \in S \setminus Q}))\\
& = \Delta(\{\theta_i\}_{i \in S\setminus Q}).
\end{align*}
Therefore, $\cut_Q(\Delta(\{\theta_i\}_{i\in S}))_f \subseteq \Delta(\{\theta_i\}_{i \in S \setminus Q })$. 
For the other containment, we first note that 
$
M(\{\theta_i\}_{i \in S\setminus Q}) = \cut_Q(M(\{\theta_i\}_{i \in S}))_f \subseteq 
\cut_Q(\Delta(\{\theta_i\}_{i \in S}))_f 
$, so 
\[
\Delta(\{\theta_i\}_{i \in S\setminus Q}) = 
\tc(M(\{\theta_i\}_{i \in S\setminus Q})) \subseteq 
\tc(\cut_Q(\Delta(\{\theta_i\}_{i \in S}))_f).
\]
Corollary \ref{cor:deltaedgesarecongruences} indicates that $\cut_Q(\Delta(\{\theta_i\}_{i\in S})_f$ is a $(|S \setminus Q|)$-dimensional congruence of $\A$. Therefore, $\Delta(\{\theta_i\}_{i \in S \setminus Q }) \subseteq \cut_Q(\Delta(\{\theta_i\}_{i\in S}))_f .$

\end{proof}

%Now we will say a little about basic coordinate transformations. The following lemma suffices for this article.  
%\begin{lem}\label{lem:coordinatetransform}
%Let $\A$ be an algebra, $n \geq 2 $, and $(\theta_0, \dots, \theta_{n-1}) \in \Con(\A)^n$. Let $ Q = \{i_0, \dots, i_{|Q|-1} \}\subseteq n$ and let $f \in 2^{n \setminus Q}$. Suppose $\phi: |Q| \to Q$ is defined by $\phi(j) = i_j$. Let $\hat{\phi}: A^{2^{|Q|}} \to A^{2^Q}$ be defined so that $(\hat{\phi}(\gamma))_g = \gamma_{g\circ\phi}$, for every $g \in 2^Q$. Then, $\hat{\phi}$ is an isomorphism such that
%\begin{enumerate}
%\item $\hat{\phi}(\refl_i(\gamma)) = \refl_{\phi(i)}(\hat{\phi}(\gamma))$, and 
%\item $\hat{\phi}(\sym_i(\gamma)) = \sym_{\phi(i)}(\hat{\phi}(\gamma))$, for every $\gamma \in A^{2^{|Q|}}$ and $i \in |Q|$.
%\end{enumerate}
%The following also hold:
%\begin{enumerate}

%\item[(3)]  $\hat{\phi}: M(\theta_{i_0}, \dots, \theta_{i_{|Q|-1}}) \simeq \cut_{n \setminus Q}(M(\theta_0, \dots, \theta_{n-1}))_f $, and 
%\item[(4)] $\hat{\phi}: \Delta(\theta_{i_0}, \dots, \theta_{i_{|Q|-1}}) \simeq \cut_{n \setminus Q}(\Delta(\theta_0, \dots, \theta_{n-1}))_f $.
%end{enumerate}
%end{lem}
%\begin{proof}
%TODO:SKETCH A PROOF, revise if necessary!!!!!!!!!!!
%\end{proof}

\begin{defn}\label{def:centrality}
Let $\A$ be an algebra, $\delta \in \Con(\A)$, $S\finsub \nat$ with $|S| \geq 2$, and $i\in S$. We say that a $(|S|)$-dimensional relation $R$ on $A$ has  \textbf{$(\delta,i)$-centrality} if there is no $\gamma\in R$ such that exactly $2^{|S|-1}-1$ many vertices of $\Line_i(\gamma)$ are labeled by $\delta$-pairs.
\end{defn}

The relations that we consider here are usually $(S)$-symmetric. In this situation, the following lemma provides a useful method to check centrality. The proof is left to the reader.

\begin{lem}\label{lem:checkcentralitywithpivot}
Let $\A$ be an algebra, $\delta \in \Con(\A)$, $S\finsub \nat$ with $|S| \geq 2$, and $i\in S$. Suppose that a $(|S|)$-dimensional relation $R \subseteq A^{2^S}$ is $(S)$-symmetric. Then $R$ has $(\delta, i)$-centrality  if and only if the following condition holds:
\begin{enumerate}
\item[*] If $\gamma \in R$ is such that every $(i)$-supporting line of $\gamma$ is a $\delta$-pair, then the $(i)$-pivot line of $\gamma$ is also a $\delta$-pair.
\end{enumerate}

\end{lem}

We now define two commutators. They share one essential feature: both are defined with respect to a centrality condition that is quantified over an $(n)$-dimensional relation for some $n\geq 2$.

\begin{defn}\label{def:tccomm}
Let $\A$ be an algebra and $S \finsub \nat$ with $|S|\geq 2$. Let $\{\theta_i\}_{i \in S} \subseteq \Con(\A)$ be an $S$-indexed set of congruences. Let $k$ be the greatest element of $S$. We define

\begin{align*}
[\{\theta_i\}_{i \in S}]_{TC} &= 
\Meet \{\delta: M(\{\theta_i\}_{i \in S}) \text{ has } (\delta, k)\text{-centrality} \} \\
[\{\theta_i\}_{i \in S}]_{H} &= \Meet \{\delta: \Delta(\{\theta_i\}_{i \in S}) \text{ has } (\delta, k)\text{-centrality} \}.
\end{align*}
We call these operations the $(|S|)$-ary \textbf{term condition commutator} and \textbf{hypercommutator}, respectively. In case $S = \{m, \dots, n-1\} =n\setminus m$, we use the notation
$[\theta_m, \dots, \theta_{n-1}]_{TC}$ and $[\theta_m, \dots, \theta_{n-1}]_H$ for these operations.

\end{defn}

\begin{thm}\label{thm:basicpropertiescommutator}
Let $\A$ be an algebra, $n\geq 2$, and $\theta_0, \dots, \theta_{n-1}, \gamma_0, \dots, \gamma_{n-1} \in \Con(\A)$ with $\theta_0 \leq \gamma_0, \dots, \theta_{n-1} \leq \gamma_{n-1}$. The following hold for both the term condition commutator and the hypercommutator: 
\begin{enumerate}
\item $[\theta_0, \dots, \theta_{n-1}] \leq \Meet_{i\in n} \theta_i$,
\item $[\theta_0, \dots, \theta_{n-1}] \leq [\gamma_0, \dots, \gamma_{n-1}] $, and
\item $[\theta_0, \dots, \theta_{n-1}] \leq [\theta_1, \dots, \theta_{n-1}].$
\end{enumerate}
The following also holds: 
\begin{enumerate}
\item[(4)] $[\theta_0, \dots, \theta_{n-1}]_{TC} \leq [\theta_0, \dots, \theta_{n-1}]_H$
\end{enumerate}
\end{thm}

\begin{proof}
Properties \emph{(1)}-\emph{(3)} are already known to hold for the term condition commutator, see \cite{aichmud}. Let us establish that they hold for the hypercommutator. 

To show \emph{(1)}, set $\delta = \Meet_{i\in n} \theta_i$. We must verify that $\Delta(\theta_0, \dots, \theta_{n-1})$ has $(\delta, n-1)$-centrality, and will apply the criterion established by Lemma \ref{lem:checkcentralitywithpivot} to do so. Take \newline $\mu \in \Delta(\theta_0, \dots, \theta_{n-1})$ with the property that every $(n-1)$-supporting line is a $\delta$-pair. We want to show that the $(n-1)$-pivot line of $\mu$ is a $\theta_j$-pair, for every $j \in n$. This holds for $j = n-1$, because Lemma \ref{lem:basicdelta0} indicates that $\mu \in \rect(\theta_0, \dots, \theta_{n-1})$. For $j \neq n-1$, consider the $(j, n-1)$-pivot square
\[
\squares_{j, n-1}(\mu)_\textbf{1} = 
\begin{tikzpicture}
    [ baseline=(center.base), font=\small,
      every node/.style={inner sep=0.25em}, scale=1 ]
    \node at (0.5,-0.5) (center) {\phantom{$\cdot$}}; % phantom = black magic
    \path (0,0)   node (nw) {$d$}
      -- ++(1,0)  node (ne) {$c$}
      -- ++(0,-1) node (se) {$b$}
      -- ++(-1,0) node (sw) {$a$};
    \draw (nw) -- (ne) -- (se) -- (sw) -- (nw);
    
    \draw (sw) edge[bend left] node[left] {$\delta$} (nw);
    
    \end{tikzpicture}.
\]
The pair $\langle a,d \rangle $ is an $(n-1)$-supporting line of $\mu$ and is therefore a $\delta$-pair. We have indicated this with a curved line. The $(n-1)$-pivot line of $\mu$ is the pair $\langle b,c \rangle$. Because $\gamma \in \rect(\theta_0, \dots, \theta_{n-1})$, it follows that $\langle a,b \rangle, \langle c, d \rangle \in \theta_j$. Therefore, $\langle b, c \rangle \in \theta_j$. 

To show \emph{(2)} and \emph{(4)}, it is enough to note that 
\[M(\theta_0, \dots, \theta_{n-1}) \leq \Delta(\theta_0, \dots, \theta_{n-1}) \leq \Delta(\gamma_0, \dots, \gamma_{n-1}),
\] and that the set $\{ R \subseteq A^{2^n}: \text{$R$ has $(\delta, n-1)$-centrality} \}$ is downward closed, for every $\delta \in \Con(\A)$.  

To see that \emph{(3)} holds, suppose that $\delta \in \Con(\A)$ is such that $\Delta(\theta_1, \dots, \theta_{n-1})$ has $(\delta, n-1)$-centrality. Take $\gamma \in \Delta(\theta_0, \dots, \theta_{n-1})$ such that every $(n-1)$-supporting line of $\gamma$ is a $\delta$-pair. It follows that every $(n-1)$-supporting line of $\faces_0^1(\gamma)$ is a $\delta$-pair. Lemma \ref{lem:deltaprojectstodelta} indicates that $\faces_0^1(\Delta(\theta_0, \dots, \theta_{n-1})) = \Delta(\{\theta_i\}_{i \in n \setminus \{0\}}) = \Delta(\theta_1, \dots, \theta_{n-1})$. We apply the assumption that $\Delta(\theta_1, \dots, \theta_{n-1})$ has $(\delta, n-1)$-centrality and conclude that $\lines_{n-1}(\faces_0^1(\gamma))_\textbf{1} = \lines_{n-1}(\gamma)_\textbf{1}$ is a $\delta$-pair. We have shown that $\Delta(\theta_0, \dots, \theta_{n-1})$ also has $(\delta, n-1)$-centrality, so the proof is finished. 
\end{proof}

\subsection{Nilpotence and Supernilpotence}\label{subsec:nilpotenceandsupernilpotence}

Let $\A$ be an algebra and let $\theta \in \Con(\A)$. Recursively define
over $\nat$ the congruences $[\theta)_0 \coloneqq \theta$, and
\[
  (\theta]_{n+1} \coloneqq \big[ \theta, (\theta]_n \big]_{TC}
\]
to produce a descending chain called the \textbf{lower central series} of $\theta$: 
\[
  (\theta_0] \geq (\theta]_1 \geq \dots \geq (\theta]_n \geq \dots.
\]
If $(\theta]_n=0$, then $\theta$ is said to be
\textbf{$(n)$-step left nilpotent}.
A congruence $\theta$ of $\A$ is said to be \textbf{$(n)$-step
supernilpotent} if it satisfies 
\[
  [\underbrace{\theta, \dots, \theta}_{(n-1)}]_{TC} = 0.
\]

\section{The binary and ternary cases}\label{sec:lowaritycase}

\subsection{Proof of H=TC for the binary and ternary cases}\label{subsec:biterH=TC}

Theorem \ref{thm:basicpropertiescommutator} indicates that the hypercommutator is always an upper bound for the term condition commutator of the same arity. In this section we will show that 

\[
[\theta, \theta]_H \leq [\theta, \theta]_{TC} 
\qquad \text{and} \qquad 
[\theta, \theta, \theta]_H \leq [\theta, \theta, \theta]_{TC}
\]
if $\theta$ is a congruence of a Taylor algebra (see the beginning of Section \ref{sec:higherarities}.) Indeed, we will demonstrate that $\Delta(\theta, \theta)$ has $(\delta, i)$-centrality for each $i \in 2$ whenever $M(\theta, \theta)$ has $(\delta, i)$-centrality for each $i \in 2$. The idea for the proof will generalize to any dimension. We want to point out that the key to the argument is inspired by Lemma 4.4 in \cite{kearnesszendreirel}.

\begin{lem}\label{lem:binaryh=tc}
Let $\var $ be a variety with Taylor term $t$. Let $\A \in \var$, $\theta, \delta \in \Con(\A)$, and $j \in 2$. Suppose $R$ is a $2$-dimensional tolerance of $\A$ such that $M(\theta, \theta) \leq R \leq \rect(\theta, \theta)$ and $R$ has $(\delta, i)$-centrality for each $i \in 2$. Then, $R^{\circ_j}$ has $(\delta, i)$-centrality for each $i\in 2$.

\end{lem}

\begin{proof}

We assume without loss assume that $j =0$. The proof will refer to the items listed in Figure \ref{fig:binTC=HC}. Before we begin, we remark that item $(0)$ shows the orientation of coordinates, and that any pair of elements that belongs to $\delta$ is connected with a curved line. A typical element of $R^{\circ_0}$ is shown in item $(1)$. Now assume that $\langle a,c \rangle \in \delta$, as shown in item $(2)$. An induction using that $R$ has $(\delta, 1)$-centrality is illustrated with dotted curved lines, and it follows that $\langle b,d \rangle \in \delta$. 
Therefore, $R^{\circ_0}$ has $(\delta, 1)$-centrality.

Next we show that $R^{\circ_0}$ has $(\delta, 0)$-centrality. Assume that $\langle c, d\rangle \in \delta$, as depicted on the left-hand side of the implication depicted in item $(3)$. Suppose that the Taylor identity that $t$ satisfies in its first coordinate is given by

$$ t(x, \phi(x, y)) \approx t(y, \psi(x,y)),$$
where $\phi(x,y)$ and $\psi(x,y)$ denote tuples in the variables $x,y$. It follows from the compatibility, $(2)$-reflexivity, and $(2)$-symmetry of $R$ that the right-hand side of the implication depicted in item $(3)$ belongs to $R^{\circ_0}$. We observed earlier that $R^{\circ_0}$ has $(\delta, 1)$-centrality, and this along with the equality $t(b, \phi(b,d)) = t(d, \psi(b,d))$ implies that $\langle t(a, \phi(b,d)), t(c, \psi(b,d)) \rangle \in \delta$. Therefore, all of the labels of this square belong to the same $\delta$-class. In particular, we conclude that $\langle t(a, \phi(b,d)) , t(b, \phi(b,d) \rangle $ is a $\delta$-pair.

Now, let $\eta(a,b) \in \{a,b\}^{\sigma(t)-1}$. Because $R \leq \rect(\theta, \theta)$, we know that $a,b,d$ all belong to the same $\theta$-class. We assume also that $M(\theta, \theta) \leq R$, hence the square shown in item \emph{(4)} belongs to $R$. Because $R$ is assumed to have $(\delta, 0)$-centrality, we conclude that $\langle t(a, \eta(a,b)), t(b, \eta(a,b)) \rangle $ is a $\delta$-pair.

This line of reasoning can be duplicated for each coordinate of the Taylor term $t$. Therefore, we construct a $\delta$-chain that connects $a$ to $b$, see item $(5)$. This demonstrates that $R^{\circ_0}$ has $(\delta,0)$-centrality.

\end{proof}

%\begin{prop}\label{prop:deltahastccenbinary}
%Let $\var$ be a Taylor variety and let $\A \in \var$. Let $\theta \in \Con(\A)$ and set $\delta = [\theta, \theta]_{TC}$. The relation $\Delta(\theta, \theta)$ has $(\delta, i)$-centrality for each $i\in 2$. 
%\end{prop}

%\begin{proof}
%By the definition of the term condition commutator, the relation $M(\theta, \theta)$ has $(\delta, i)$-centrality for each $i\in 2$. Furthermore, $M(\theta, \theta)$ is both $2$-reflexive and $2$-symmetric. It follows by an induction using Lemma \ref{lem:binaryh=tc} that $\tc_i(M(\theta, \theta))$ has these same properties, so it follows that 

%\[
%\Delta(\theta, \theta) = \Union_{i\in \omega} \tc_i(M(\theta, \theta))
%\]
%has $(\delta, i)$-centrality for each $i\in 2$.
%\end{proof}

\begin{thm}\label{thm:binarytc=h}
For $\var$ be a Taylor variety, $\A \in \var$, and $\theta \in \Con(\A)$,
\[
[\theta, \theta]_H = [\theta, \theta]_{TC}.
\]
\end{thm}

\begin{proof} 
By Theorem \ref{thm:basicpropertiescommutator}, the binary hypercommutator always lies above the binary term condition commutator. We show that $[\theta, \theta]_H \leq [\theta, \theta]_{TC}$. Set $\delta = [\theta, \theta]_{TC}$. It suffices to check that $\Delta(\theta, \theta)$ has $(\delta, i)$-centrality, for each $i \in 2$. 

We proceed by induction. For each $j \in \nat$ set  $R_j = \tc_j(M(\theta, \theta))$. It follows inductively from \emph{(1)} of Lemma \ref{prop:hcongenerate} that each $R_j$ is a $(2)$-dimensional tolerance such that $M(\theta, \theta) \leq R \leq \rect(\theta, \theta)$. Using this, it follows inductively from Lemma \ref{lem:binaryh=tc} that each $R_j$ has $(\delta, i)$-centrality for all $i \in 2$. Because 
$
\Delta(\theta, \theta) = \Union_{j \in \nat} R_j,
$
the proof is finished. 
\end{proof}

The proof of Theorem \ref{thm:binarytc=h} has a structure which provides a template for the higher arity cases. The following is a list of the essential steps and their names.

\begin{enumerate}
 
\item \textbf{Inductive Assumption:} Assume that $R$ is an $(n)$-dimensional tolerance such that $\M(\theta, \dots, \theta) \leq R \leq \rect(\theta, \dots, \theta) \leq \A^{2^n}$ and $R$ has $(\delta, l)$-centrality for all $l\in n$.

\item \textbf{Perpendicular Stage:} Establish that $R^{\circ_j} $ has $(\delta, i)$-centrality for $i \neq j $.
\item  \textbf{Parallel Stage:} Establish that $R^{\circ_j}$ has $(\delta, j)$-centrality.

\end{enumerate}
Next, we illustrate this proof template in the $(3)$-dimensional case. 

\begin{lem}\label{lem:ternaryh=tc}
Let $\var $ be a variety with Taylor term $t$ and let $\A \in \var$. Let $\theta, \delta \in \Con(\A)$ and $j \in 3$. Let $R$ be a $(3)$-dimensional relation such that $ M(\theta, \theta, \theta) \leq R \leq \rect(\theta, \theta, \theta)$ and $R$ has $(\delta, i)$-centrality for each $i\in 3$. Then, the relation $R^{\circ_j}$ 
has $(\delta, i)$-centrality for each $i\in 3$.

\end{lem}

\begin{proof}
The main steps of the proof are illustrated in Figures \ref{fig:ternaryperp} and \ref{fig:ternarypar}. Without loss, we assume that $j=0$. We begin with the perpendicular stage and refer to Figure \ref{fig:ternaryperp}. Item $(0)$ illustrates the orientation of coordinates. We want to show that $R$ has $(\delta, i)$-centrality for each $i\neq 0$. Without loss, take $i=1$. A typical element of $R^{\circ_0}$ is depicted in item $(1)$. The left hand side of the implication in item $(2)$ illustrates the assumption that 
\[
\langle a_0, c_0 \rangle, \langle b_0, d_0 \rangle, \langle a_1, c_1 \rangle \in \delta.
\]
We want to show that $\langle b_1, d_1 \rangle \in \delta$. Suppose that the identity that $t$ satisfies in the first coordinate is given by

$$ t(x, \phi(x, y)) \approx t(y, \psi(x,y)),$$
where $\phi(x,y)$ and $\psi(x,y)$ denote tuples in the variables $x,y$. The right hand side of the implication in item $(2)$ depicts a sequence of elements of $R$, the corners of which determine a cube that belongs to $R^{\circ_0}$. Each solid curved line indicates that the corresponding vertex labels determine a $\delta$-pair, while the symbol along each top row indicates an equality that results from an application of the Taylor identity. The curved dotted lines also indicate $\delta$-pairs. Their existence is deduced left-to-right, first by the transitivity of $\delta$, then by an application of the $(\delta,2)$-centrality of $R$, and last by an application of the transitivity of $\delta$. We conclude that 

\[
\langle t(b_1, \psi(b_0, b_1)), t(d_1,\psi(b_0, b_1)) \rangle \in \delta
\]

Let $\eta(b_1, d_1) \in \{b_1, d_1\}^{\sigma(t)-1}$. The labeled cube depicted in item $(3)$ is an element of $R$. This follows because the labeled cube determined by the first argument of $t$ belongs to $R$ (because $R$ is $3$-symmetric), as do the labeled cubes determined by each of the remaining arguments of $t$ (because $M(\theta, \theta, \theta) \leq R \leq \rect(\theta, \theta, \theta)$.) The two columns belonging to the back face determine $\delta$-pairs because 
$\langle b_0, d_0 \rangle \in \delta$, and it has been shown that the left column of the front face also determines a $\delta$-pair. Because $R$ has $(\delta, 1)$-centrality, we conclude that 
\[
\langle t(b_1, \eta(b_1, d_1)), t(d_1, \eta(b_1, d_1)) \rangle \in \delta.
\]
Item (4) finishes the argument in a manner identical to the end of the proof of Lemma \ref{lem:binaryh=tc}. This finishes the perpendicular stage of the argument.

We proceed to the \textbf{parallel stage} and refer to Figure \ref{fig:ternarypar}. The left hand side of the implication in item $(2')$ illustrates the assumption that 
\[
\langle a_0, b_0 \rangle, \langle c_0, d_0 \rangle, \langle c_1, d_1 \rangle \in \delta
\]
We want to show that $\langle a_1, b_1 \rangle \in \delta$. As before, we present an argument involving the first argument of the Taylor term. The right hand side of the implication in item $(2')$ depicts a sequence of elements of $R$, the corners of which determine a cube that belongs to $R^{\circ_0}$. A solid curved line indicates a $\delta$-pair whose existence follows from the initial assumptions. The dotted curved lines also indicate $\delta$-pairs. The existence of the bottom dotted curved line follows from the transitivity of $\delta$,
while the existence of the top dotted curved line follows from our earlier completion of the perpendicular stage. We conclude that 
\[
\langle t(a_1, \psi(b_0, b_1)) , t(b_1, \psi(b_0,b_1)) \in \delta.
\]

Now, let $\eta(a_1,b_1) \in \{a_1, b_1 \}^{\sigma(t)-1}$. As before, our goal is to show that 
\[
\langle t(a_1, \eta(a_1, b_1)), t(b_1, \eta(a_1, b_1)) 
\rangle \in \delta.
\]
We need to produce an element of $R$ to which we may apply the assumption that $R$ has $(\delta, i)$-centrality for each $i\in 3$. This is possible provided we assume that 
\[
\Square[b_0][a_0][b_1][a_1] \in M(\theta, \theta),
\]
as illustrated in item $(3')$. The remainder of the argument in this case is similar to the perpendicular stage. 

In general, we may only produce the sequence of elements of $R$ shown in item $(5')$. Because this is another instance of the parallel stage, it appears as though no progress has been made. However, note that there is a symmetric version of $(3')$ in which we assume that 

\[
\Square[c_1][a_1][d_1][b_1] \in M(\theta, \theta).
\]
This new instance satisfies assumptions of this symmetric version of $(3')$, so we conclude that 

\[
\langle t(a_1, \eta(a_1, b_1)), t(b_1, \eta(a_1, b_1)) 
\rangle \in \delta.
\]
This finishes the proof of the parallel stage. 
\end{proof}

The analogue of Theorem \ref{thm:binarytc=h} immediately follows. Because it is a special case of Theorem \ref{thm:allarityH=TC}, we omit the proof.

\begin{thm}\label{thm:ternarytc=h}
Let $\var$ be a Taylor variety, $\A \in \var$, and $\theta \in \Con(\A)$. In this situation,
\[
[\theta, \theta, \theta]_H = [\theta, \theta, \theta]_{TC}.
\]
\end{thm}

\subsection{Proof of HHC8 for the binary and ternary case}\label{subsec:ternHHC8}
Let $\A$ be any algebra and take $\theta_0, \theta_1, \theta_2 \in \Con(\A)$. We will show that
\[
[\theta_0, [\theta_1, \theta_2]_H ]_H\leq [\theta_0, \theta_1, \theta_2]_H
\]
We begin by developing a relational characterization of both the binary and ternary hypercommutators. Both Propositions \ref{prop:binhypchar} and \ref{prop:ternaryhypchar} are special cases of Theorem \ref{thm:charhypercomm}.

\begin{prop}\label{prop:binhypchar}
Let $\A$ be an algebra and take $\theta_0, \theta_1 \in \Con(\A)$. The following are equivalent.
\begin{enumerate}
\item $\langle x,y \rangle \in [\theta_0, \theta_1]_H$.
\item $\Square [x][x][x][y] \in \Delta(\theta_0, \theta_1)$.
\item $\Square[a][a][x][y] \in \Delta(\theta_0, \theta_1)$ for some $a \in A$.

\item $\Square[b][x][b][y] \in \Delta(\theta_0, \theta_1)$ for some $b \in A$.

\end{enumerate}
\end{prop}

\begin{proof}
The proof of this Proposition is the $(2)$-dimensional version of the proof provided for Proposition \ref{prop:ternaryhypchar}.
\end{proof}
\begin{prop}\label{prop:ternaryhypchar}
Let $\A$ be an algebra and take $\theta_0, \theta_1, \theta_2 \in \Con(\A)$. The following are equivalent.
\begin{enumerate}
\item $\langle x,y \rangle \in [\theta_0, \theta_1, \theta_2]_H$.
\item $\Cube [x][x][x][x][x][x][x][y] \in \Delta(\theta_0, \theta_1, \theta_2)$.
\item $\Cube[a_0][b_0][a_0][b_0][c_0][x][c_0][y] \in \Delta(\theta_0, \theta_1, \theta_2)$ for some $a_0,b_0,c_0 \in A$.
\item $\Cube[a_1][a_1][b_1][b_1][c_1][c_1][x][y] \in \Delta(\theta_0, \theta_1, \theta_2)$ for some $a_1,b_1,c_1 \in A$.

\item $\Cube[a_2][c_2][b_2][x][a_2][c_2][b_2][y] \in \Delta(\theta_0, \theta_1, \theta_2)$ for some $a_2,b_2,c_2 \in A$.

\end{enumerate}
\end{prop}

\begin{proof}
We first show that \emph{(2),(3),(4),(5)} are equivalent. It is clear that \emph{(2)} implies each of \emph{(3)}, \emph{(4)}, \emph{(5)}. Assume that \emph{(3)} holds and refer to Figure \ref{fig:thcchar1}. Item $(0)$ provides the orientation of coordinates. Items $(1)$ and $(2)$ illustrate that \emph{(2)} holds, where each step follows from the $(3)$-symmetry, reflexivity, and transitivity of $\Delta(\theta_0, \theta_1, \theta_2)$. The proof that \emph{(4)} or \emph{(5)} imply \emph{(2)} is similar and is omitted. 

Now we show \emph{(1}) holds if and only if \emph{(2)} holds. Set 
\[
\delta = \{ \langle x,y \rangle : \text{\emph{(2)} holds} \}.
\]
It is clear that $\delta \subseteq [\theta_0,\theta_1,\theta_2]$, establishing that \emph{(2)} implies \emph{(1)}. To establish the other direction it suffices to show that $\delta$ is a congruence, which we leave to the reader, and also that $\Delta(\theta_0, \theta_1, \theta_2)$ has $(\delta, 2)$-centrality, which we prove now.

We refer to Figure \ref{fig:thcchar2}. Item $(0)$ provides the orientation of coordinates. In item $(1)$ a typical element of $\Delta(\theta_0, \theta_1, \theta_2)$ is depicted with every $(1)$-supporting line determining a $\delta$-pair. We need to show that the $(1)$-pivot line $\langle f,h \rangle $ is also a $\delta$-pair. The result of items $(2)$-$(5)$ is that

\[
\DeltaOneCubeD[a][c][b][d][e][g][f][h]
\in \Delta(\theta_0, \theta_1, \theta_2) 
\implies
\DeltaOneCubeD[a][a][b][b][e][g][f][h]
\in \Delta(\theta_0, \theta_1, \theta_2).
\]

A similar argument may be applied to this new cube to produce the cube shown in item $(6)$. We know that \emph{(4)} implies \emph{(2)}, so $\langle f, h \rangle \in \delta$.
\end{proof}

We remark that Propositions \ref{prop:binhypchar} and \ref{prop:ternaryhypchar} imply that both the binary and ternary hypercommutator are symmetric, i.e.\ their output does not depend on the order of their arguments. The following is a less obvious consequence.

\begin{thm}[Binary-ternary HHC8]\label{thm:bthhc8}
If $\A$ is an algebra and $\theta_0, \theta_1, \theta_2 \in \Con(\A)$, then
\[
[\theta_0, [\theta_1, \theta_2]_H]_H \leq [\theta_0, \theta_1, \theta_2]_H.
\]
\end{thm}

\begin{proof}
We use the same orientation of coordinates as in the other proofs. Take $\langle x,y\rangle \in [\theta_0, [\theta_1,\theta_2]_H]_H$. We will show that $\langle x, y \rangle \in [\theta_0, \theta_1,\theta_2]_H$. By Propositions \ref{prop:binhypchar} and \ref{prop:ternaryhypchar}, this amounts to showing that 
\[
\Square[x][x][x][y] \in \Delta(\theta_0,[\theta_1, \theta_2]_H) \implies
\Cube[x][x][x][x][x][x][x][y] \in \Delta(\theta_0, \theta_1, \theta_2).
\]
To this end, set 
\[
N(\theta_0, \theta_1, \theta_2) = \left\{ h: h\in \Delta(\theta_0, \theta_1,\theta_2) \text{ and } h= \Cube[a][a][c][c][a][b][c][d] \text{ for some } a,b,c,d \in A \right\}.
\]

We claim that $\Delta(\theta_0, [\theta_1, \theta_2]_H) \leq \faces_2^1(N(\theta_0, \theta_1, \theta_2)) = R $. To prove it, we will show that $R$ contains the generators of $M(\theta_0, [\theta_1, \theta_2]_H)$ and is a $(2)$-dimensional congruence. 

Indeed, suppose that $\langle x,y \rangle \in [\theta_1, \theta_2]_H$. Proposition \ref{prop:binhypchar} shows that 
\[
\mu = \Square[x][x][x][y] \in \Delta(\theta_1, \theta_2).
\]
On the other hand, Lemma \ref{lem:deltaprojectstodelta} indicates that $\mu \in \faces_0^0(\Delta(\theta_0, \theta_1, \theta_2)$. Because $\Delta(\theta_0, \theta_1, \theta_2)$ is $(3)$-reflexive, we have shown that

\[
\langle x,y\rangle  \in [\theta_1, \theta_2] \iff 
\Cube[x][x][x][x][x][y][x][y] \in \Delta(\theta_0, \theta_1, \theta_2)
\]
and therefore 
\[
\cube_1([\theta_1, \theta_2]_H)=
\left\{ \Square[x][y][x][y]: \langle x,y \rangle \in [\theta_1, \theta_2]_H \right\} \subseteq \faces_2^1(N(\theta_0, \theta_1, \theta_2)).
\]
Also, $\faces_2^1(\cube_0(\theta_0)) = \cube_0(\theta_0)$ (the relation on the left is $(3)$-dimensional, the relation on the right is $(2)$-dimensional,) so 
\[\cube_0(\theta_0) = 
\left\{ \Square[x][x][y][y]: \langle x,y \rangle \in \theta_0 \right\} \subseteq \faces_2^1(N(\theta_0, \theta_1, \theta_2)).
\]

We have shown that the generators of $M(\theta_0, [\theta_1, \theta_2])$ belong to $R$. It remains to verify that $R$ is a $(2)$-dimensional congruence. We show here that $\faces_1(R)$ is transitive (the proof of the other conditions is similar.) So, take 
\[
\lambda = \Cube[a][a][c][c][a][b][c][d] , \epsilon =
\Cube[b][b][d][d][b][e][d][f] \in \Delta(\theta_0, \theta_1, \theta_2).
\]
Now, $\Delta(\theta_0, \theta_1, \theta_2)$ is a $(3)$-dimensional congruence, so we have that

\[
 \begin{tikzpicture}[scale=.8]	
	 \tikzstyle{vertex}=[circle,minimum size=10pt,inner sep=0pt]
	 \tikzstyle{selected vertex} = [vertex, fill=red!24]
	 \tikzstyle{selected edge} = [draw,line width=5pt,-,red!50]
	 \tikzstyle{edge} = [draw,-,black]
	 \node[fill=white] (v0) at (0,0) {$a$};
	 \node[vertex] (v1) at (0,1) {$a$};
	 \node[vertex] (v2) at (1,0) {$c$};
	 \node[vertex] (v3) at (1,1) {$c$};
	 
	 \node[vertex] (v4) at (0.5,-.5) {$b$};
	 \node[vertex] (v5) at (0.5,.5) {$b$};
	 \node[vertex] (v6) at (1.5,-.5) {$d$};
	 \node[vertex] (v7) at (1.5,.5) {$d$};
	 
	 \node[vertex] (v8) at (1.7,-1.7) {$b$};
	 \node[vertex] (v9) at (1.7,-.7) {$b$};
	 \node[vertex] (v10) at (2.7,-1.7) {$d$};
	 \node[vertex] (v11) at (2.7,-.7) {$d$};
	 
	 \node[vertex] (v12) at (2.2,-2.2) {$b$};
	 \node[vertex] (v13) at (2.2,-1.2) {$e$};
	 \node[vertex] (v14) at (3.2,-2.2) {$d$};
	 \node[vertex] (v15) at (3.2,-1.2) {$f$};
	 
	 \node[vertex] (v16) at (0,-1.7) {$a$};
	 \node[vertex] (v17) at (0,-.7) {$a$};
	 \node[fill=white] (v18) at (1,-1.7) {$c$};
	 \node[vertex] (v19) at (1,-.7) {$c$};
	 
	 \node[vertex] (v20) at (2.2,-3.9) {$a$};
	 \node[fill=white] (v21) at (2.2,-2.9) {$b$};
	 \node[fill=white] (v22) at (3.2,-3.9) {$c$};
	 \node[vertex] (v23) at (3.2,-2.9) {$d$};

	 \node[vertex] (implies) at (5.5, -1.5) 
	 {$\in \Delta(\theta_0, \theta_1, \theta_2) \implies$};
	 
	 \node[vertex] (v24) at (7.7,-1.7) {$a$};
	 \node[vertex] (v25) at (7.7,-.7) {$a$};
	 \node[vertex] (v26) at (8.7,-1.7) {$c$};
	 \node[vertex] (v27) at (8.7,-.7) {$c$};
	 
	 \node[vertex] (v28) at (8.2,-2.2) {$a$};
	 \node[vertex] (v29) at (8.2,-1.2) {$e$};
	 \node[vertex] (v30) at (9.2,-2.2) {$c$};
	 \node[vertex] (v31) at (9.2,-1.2) {$f$};
	 
	  \node[vertex] (impliess) at (11, -1.5) 
	 {$\in \Delta(\theta_0, \theta_1, \theta_2). $};
	 
	 \begin{scope}[on background layer]   
     \draw[edge] (v0) -- (v1) -- (v3) -- (v2) -- (v0);
	 \draw[edge] (v4) -- (v5) -- (v7) -- (v6) -- (v4);
	 \draw[edge] (v8) -- (v9) -- (v11) -- (v10) -- (v8);
	 \draw[edge] (v12) -- (v13) -- (v15) -- (v14) -- (v12);
	 \draw[edge] (v16) -- (v17) -- (v19) -- (v18) -- (v16);
	 \draw[edge] (v20) -- (v21) -- (v23) -- (v22) -- (v20);

	 \draw[edge] (v0) -- (v4);  \draw[edge] (v1) -- (v5);
	 \draw[edge] (v2) -- (v6);  \draw[edge] (v3) -- (v7);
	 
	 \draw[edge] (v8) -- (v12);  \draw[edge] (v9) -- (v13);
	 \draw[edge] (v10) -- (v14);  \draw[edge] (v11) -- (v15);
	 
	 \draw[edge] (v16) -- (v20);  \draw[edge] (v17) -- (v21);
	 \draw[edge] (v18) -- (v22);  \draw[edge] (v19) -- (v23);
	 
	  \draw[edge] (v24) -- (v25) -- (v27) -- (v26) -- (v24);
	 \draw[edge] (v28) -- (v29) -- (v31) -- (v30) -- (v28);
	 
	 \draw[edge] (v24) -- (v28);  \draw[edge] (v25) -- (v29);
	 \draw[edge] (v26) -- (v30);  \draw[edge] (v27) -- (v31);
	 
	 \end{scope}

 \end{tikzpicture}
\]
It is easy to see that each of the three cubes on the left hand side of the above implication belong to $\Delta(\theta_0, \theta_1, \theta_2)$. An application of $(3)$-transitivity produces the desired result.

Finally, suppose that $\langle x,y \rangle \in [\theta_0, [\theta_1, \theta_2]_H]_H$. By Proposition \ref{prop:binhypchar}, we know that 
\[
\Square[x][x][x][y] \in \Delta(\theta_0, [\theta_1, \theta_2]_H),
\]
and we have demonstrated that $\Delta(\theta_0, [\theta_1, \theta_2]_H) \leq R$. It follows that
\[
\Cube[x][x][x][x][x][x][x][y] \in \Delta(\theta_0, \theta_1, \theta_2).
\]
We apply Proposition \ref{prop:ternaryhypchar} and conclude that $\langle x,y \rangle \in [\theta_0, \theta_1, \theta_2]$.

\end{proof}

\begin{cor}\label{cor:ternarynestingproperty}
If $\theta$ is a congruence of a Taylor algebra $\A$, then 
\[
[\theta, [\theta, \theta]_{TC}]_{TC} \leq [\theta, \theta, \theta]_{TC}.
\] 
\end{cor}

\begin{figure}[H]
\includegraphics[scale=1]{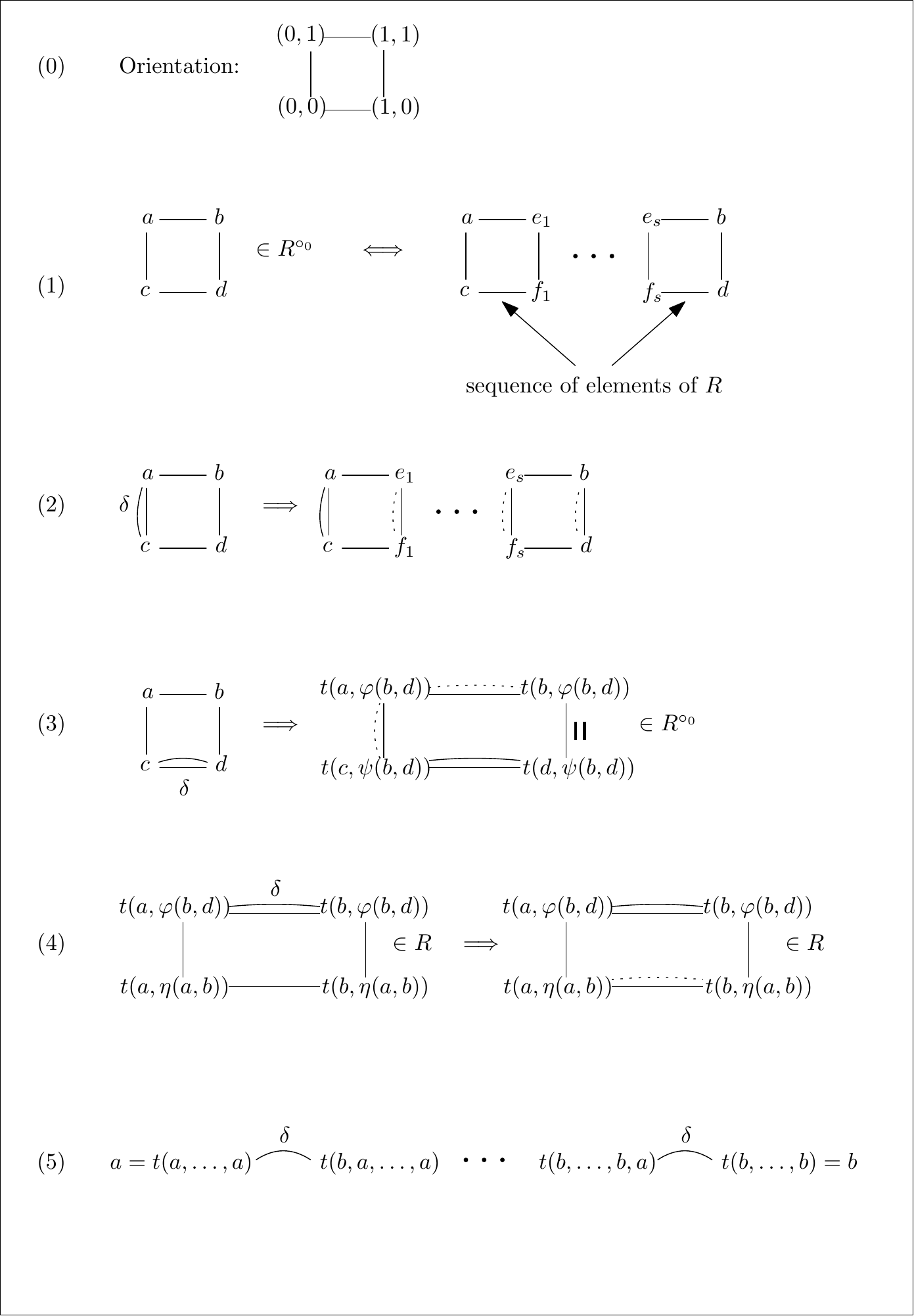}
\caption{Binary Case}\label{fig:binTC=HC}
\end{figure}

\begin{figure}[H]
\includegraphics[scale=1]{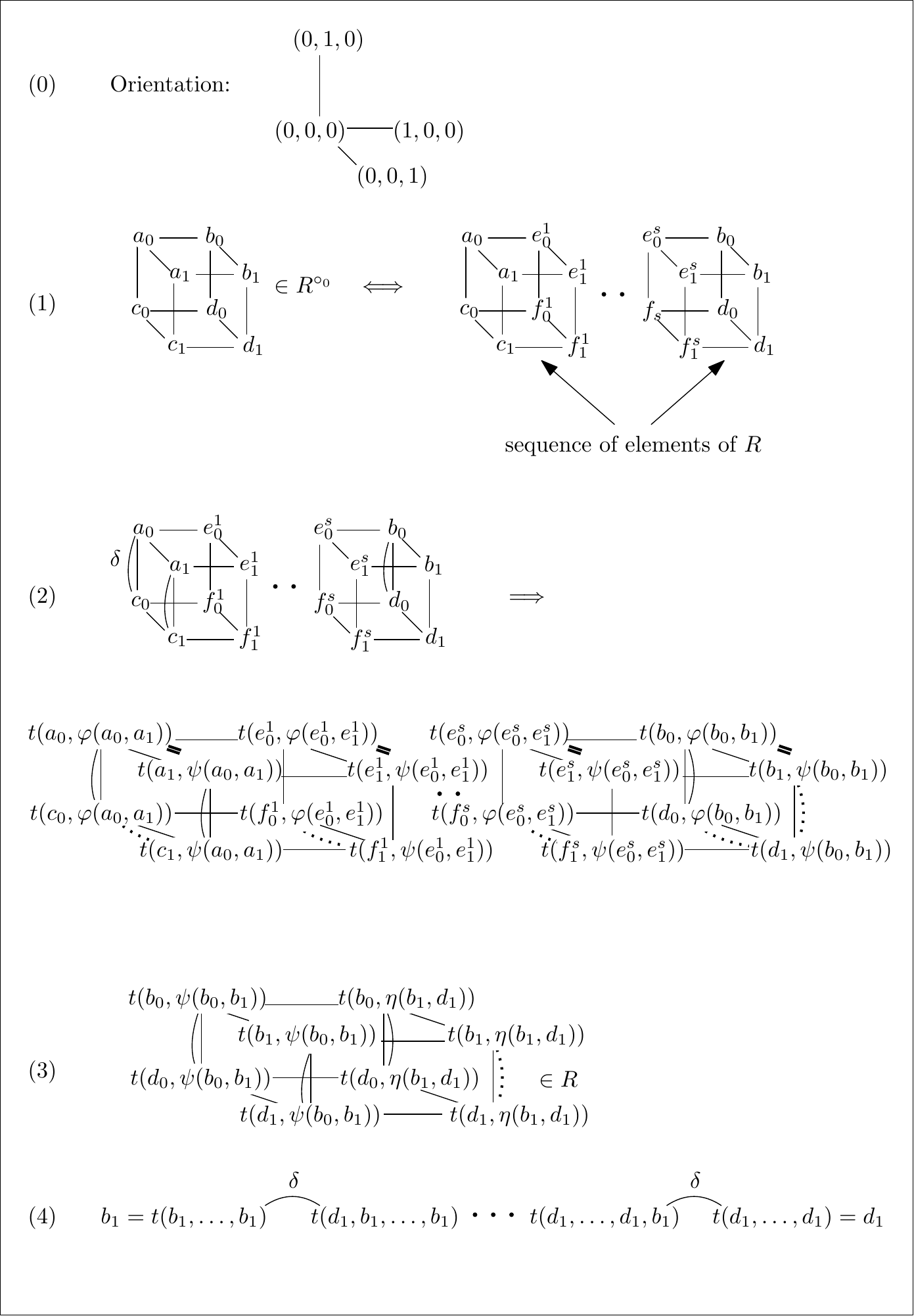}
\caption{Ternary perpendicular stage}\label{fig:ternaryperp}
\end{figure}
\begin{figure}[H]
\includegraphics[scale=1]{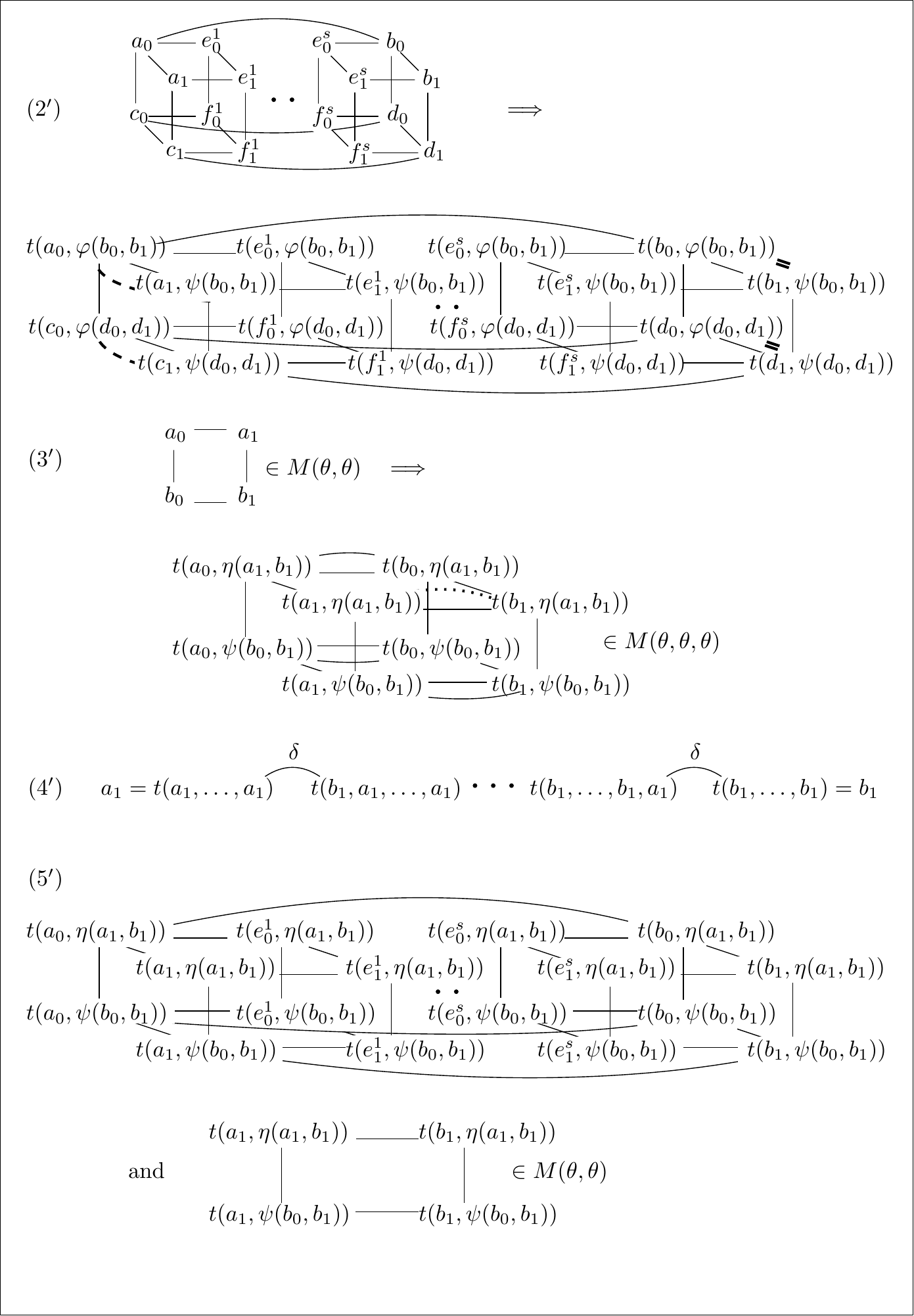}
\caption{Ternary parallel stage}\label{fig:ternarypar}
\end{figure} 

\begin{figure}[H]
\begin{center}
\includegraphics[scale=1]{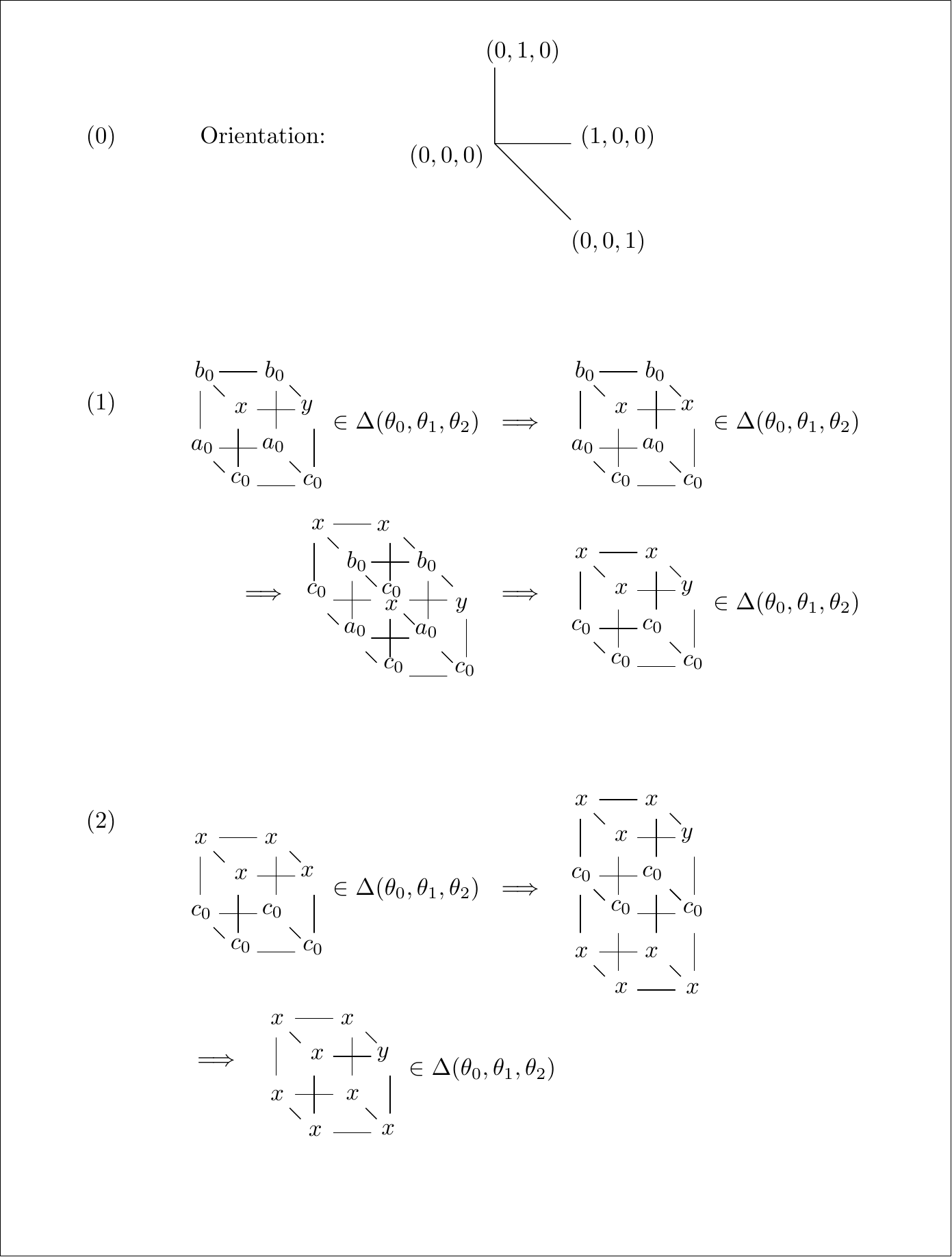}
\end{center}
\caption{Ternary hypercommutator characterization \emph{(3)} implies \emph{(2)}}\label{fig:thcchar1}
\end{figure}

\begin{figure}[H]
\begin{center}
\includegraphics[scale=1]{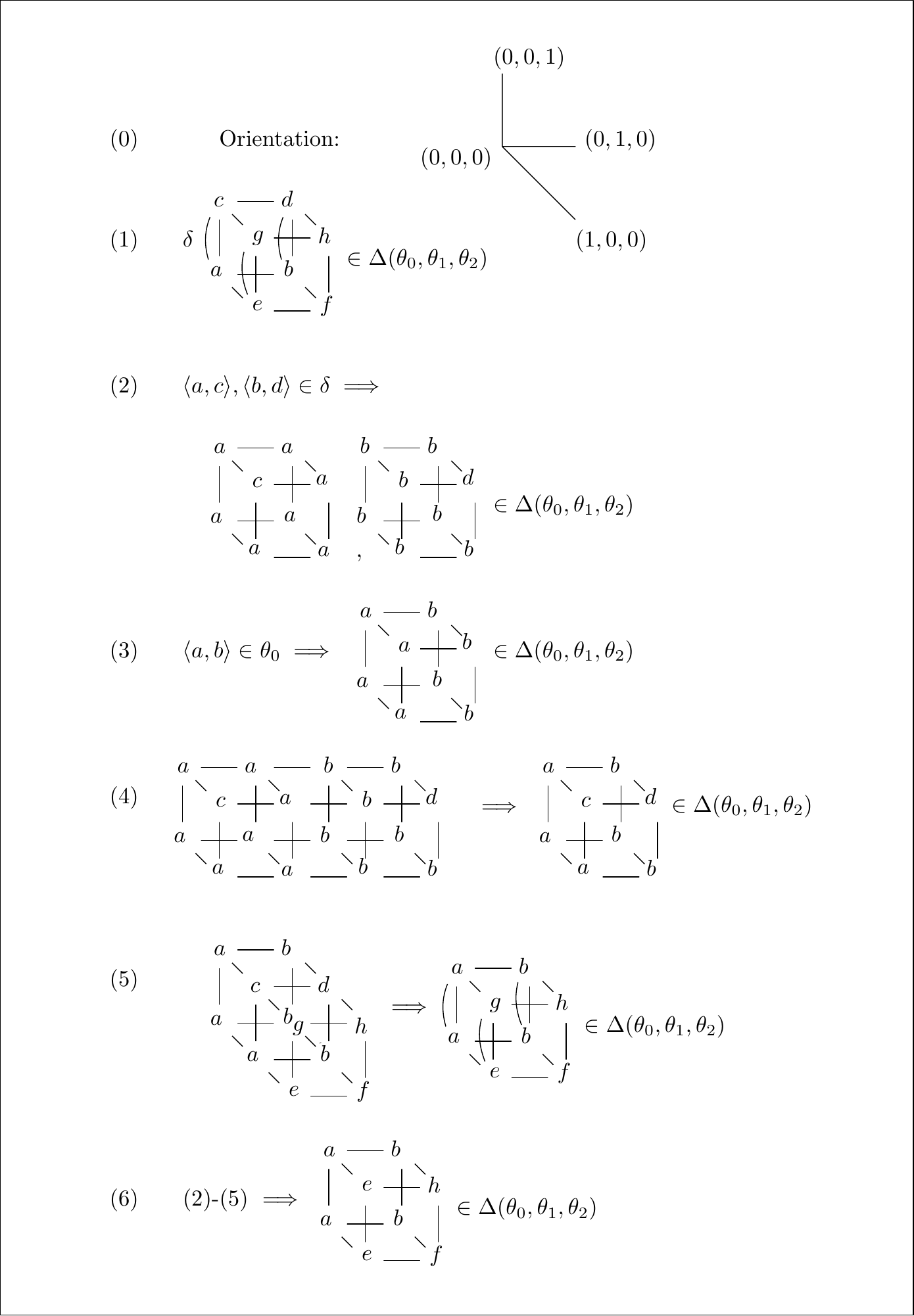}
\end{center}
\caption{Ternary hypercommutator characterization \emph{(1)} iff \emph{(2)}}\label{fig:thcchar2}
\end{figure}

\begin{proof}
The result follows from the existence of the following increasing sequence of congruences of $\A$:
\[
[\theta, [\theta, \theta]_{TC}]_{TC} 
\leq 
[\theta, [\theta, \theta]_H]_H
\leq 
[\theta, \theta, \theta]_H
= 
[\theta, \theta, \theta]_{TC}.
\]
Indeed, the first bound is a consequence of Theorem \ref{thm:basicpropertiescommutator}, the second bound is a consequence of Theorem \ref{thm:bthhc8}, and the third equality is a consequence of Theorem \ref{thm:ternarytc=h}.
\end{proof}

\section{Higher arities}\label{sec:higherarities}
This section extends the results of Section \ref{sec:lowaritycase} to any finite dimension. The basic ideas here are essentially the same as the ideas that worked for few dimensions. The term condition commutator and the hypercommutator measure two extremes of a hierarchy of centralizer conditions. To study this hierarchy for a Taylor algebra, we use the Taylor term to produce large families of cubes that connect stronger centralizer conditions to weaker ones. The argument is more complex for two reasons, the first being that cubes of dimension greater than three are not easily visualized, and the second being that the Taylor term must be composed with itself many times when the dimensional of the relations is large. 

The section is structured as follows: Subsection \ref{subsec:rotandcomp} develops machinery, and Subsections \ref{subsec:H=TC} and \ref{subsec:HHC8} extend Subsections \ref{subsec:biterH=TC} and \ref{subsec:ternHHC8} , respectively.

\subsection{Rotations and Companions}\label{subsec:rotandcomp}

Assume that $\var$ is a variety with a Taylor term, which is an idempotent term $t$ of arity $\sigma(t)$ that satisfies a package of identities of the form

$$ t\left(\begin{array}{ccccccc}
x & z_{0,1} &\cdot & \cdot & \cdot & z_{0,\sigma(t)-1} \\
z_{1,0} & x &\cdot & \cdot & \cdot & z_{1,\sigma(t)-1}  \\
\cdot  & \cdot &x & \cdot & \cdot & \cdot \\
\cdot & \cdot  &\cdot & \cdot & \cdot & \cdot \\
\cdot  & \cdot  &\cdot & \cdot & \cdot & \cdot\\
z_{\sigma(t)-1,0} & \cdot  &\cdot & \cdot & \cdot & x
\end{array} \right)  \approx t\left(\begin{array}{ccccccc}
y & w_{0,1} &\cdot & \cdot & \cdot & w_{0,\sigma(t)-1} \\
w_{1,0} & y  &\cdot & \cdot & \cdot & w_{1,\sigma(t)-1}  \\
\cdot  & \cdot &y & \cdot & \cdot & \cdot \\
\cdot & \cdot  &\cdot & \cdot & \cdot & \cdot \\
\cdot  & \cdot  &\cdot & \cdot & \cdot & \cdot\\
w_{\sigma(t)-1,0} & \cdot  &\cdot & \cdot & \cdot & y
\end{array} \right),$$
where $z_{i,j}, w_{i,j} \in \{x,y\}$ and the diagonal entries of the left and right matrices are $x$ and $y$, respectively.  

For notational convenience, we prefer to work with terms $t_0, \dots, t_{\sigma(t)-1}$, each of which is derived from the Taylor term $t$ by a permutation of variables. For each $e \in \sigma(t)$ let $t_e$ be the term such that 

$$t_e(z_e, z_0, \dots, z_{e-1}, z_{e+1}, \dots, z_{\sigma(t)-1}) = t(z_0, \dots, z_{\sigma(t)-1}). $$
Therefore, each $t_e$ satisfies an identity of the form
$$ t_e(x, \phi_e(x,y)) \approx t_e(y, \psi_e(x,y)),$$
where $\phi_e(x,y)$ and $\psi_e(x,y)$ are the tuples of length $\sigma(t)-1$ in the variables $x,y$ obtained by deleting the $e$th entry from the $e$th row of the left and right hand matrices, respectively.

Our goal is to show that the $n$-ary term condition commutator and hypercommutator are equal in a Taylor variety when evaluated at a constant tuple. To do this, we must establish a connection between the two $(n)$-dimension tolerances that are used to define each of these commutators.
Two types of $(n)$-cube, which we will call \textbf{rotations} and \textbf{companions}, are crucial to our arguments. We now define these cubes and establish their basic properties.

\begin{defn}[rotations]\label{def:rotations}
Let $\var$ be a Taylor variety with Taylor term $t$ and associated terms $t_0, \dots, t_{\sigma(t)-1}$. Let $\A \in \var$ and $n \geq 2$. For each $e \in \sigma(t)$ and $j\neq l \in n$ define  \textbf{$e$-th $(j,l)$ rotation of $\gamma \in A^{2^n}$}
as 
\[ \rrot{n}{e}{j,l}(\gamma) = t_e(\gamma, \epsilon_0, \dots, \epsilon_{\sigma(t)-2}),
\]
where for each $ s \in \sigma(t)-1$, 

\[ \epsilon_s = \begin{cases} \refl_j^0(\gamma) & \text{ if the $s$th variables of $\phi_e$ and $\psi_e$  are respectively $x$ and $y$, } \\[20pt]
\sym_l(\refl_j^0(\gamma)) & \text{ if the $s$th variables of $\phi_e$ and $\psi_e$  are respectively $y$ and $x$, }\\[20pt]
\refl_l^0(\refl_j^0(\gamma)) & \text{ if the $s$th variables of $\phi_e$ and $\psi_e$  are respectively $x$ and $x$, }
 \\[20pt]
\refl_l^1(\refl_j^0(\gamma)) & \text{ if the $s$th variables of $\phi_e$ and $\psi_e$  are respectively $y$ and $y$. }
\end{cases}
\]

\end{defn}

\begin{lem}[Basic rotation properties]\label{lem:shiftrotdef}

Let $\var$ be a Taylor variety with Taylor term $t$ and associated terms $t_0, \dots, t_{\sigma(t)-1}$. Let $\A \in \var$, $n \geq 2$, and $j \neq l \in n$. The $e$-th $(j,l)$ rotation satisfies the following properties:

\begin{enumerate}

\item 
 
Let $f \in 2^{n \setminus{ \{j,l \}}}$. If $\gamma \in A^{2^n}$ has the $(j,l)$-cross section square 
\[
 \squares_{j,l}(\gamma)_f = 
 \begin{tikzpicture}
    [ baseline=(center.base), font=\small,
      every node/.style={inner sep=0.25em}, scale=1 ]
    \node at (0.5,-0.5) (center) {\phantom{$\cdot$}}; % phantom = black magic
    \path (0,0)   node (nw) {$c_f$}
      -- ++(1,0)  node (ne) {$d_f$}
      -- ++(0,-1) node (se) {$b_f$}
      -- ++(-1,0) node (sw) {$a_f$};
    \draw (nw) -- (ne) -- (se) -- (sw) -- (nw);
  \end{tikzpicture},
 \]
then $\rrot{n}{e}{j,l}(\gamma)$ has the $(j,l)$-cross-section square

\[
 \squares_{j,l}(\rrot{n}{e}{j,l}(\gamma))_f = 
  \begin{tikzpicture}
    [ baseline=(center.base), font=\small,
      every node/.style={inner sep=0.25em}, scale=1 ]
    \node at (0.5,-0.5) (center) {\phantom{$\cdot$}}; % phantom = black magic
    \path (0,0)   node (nw) {$t_e(c_f, \psi_e(a_f,c_f))$}
      -- ++(3.3,0)  node (ne) {$t_e(d_f, \psi_e(a_f,c_f))$}
      -- ++(0,-1) node (se) {$t_e(b_f, \phi_e(a_f,c_f))$}
      -- ++(-3.3,0) node (sw) {$t_e(a_f, \phi_e(a_f,c_f))$};
    \draw (nw) -- (ne) -- (se) -- (sw) -- (nw);
    
    %\draw (sw) edge[bend left] (nw);
    %\draw (nw) edge[bend left] (ne);
    %\draw (ne) edge[bend left] (se);
    %\draw (se) edge[bend left] (sw);
    
    %\draw (sw) edge[bend left] (nw);
    %\draw (nw) edge[bend left] (ne);
    %\draw (ne) edge[bend left] (se);
    %\draw (se) edge[bend left] (sw);
    
    %\draw (se) edge node[below] {$=$} (sw);
    \draw (sw) edge node[left] {$\parallel$} (nw);
    \end{tikzpicture}.
\]

\item For  $j \neq l \in n-1$, 
\[\rrot{n}{e}{j,l}(\gamma) = 
\glue_{\{n-1\}} \left( \rrot{n-1}{e}{j,l}( \faces_{n-1}^0(\gamma)), \rrot{n-1}{e}{j,l}( \faces_{n-1}^1(\gamma)) \right).
\]

\item If $R \leq \A^{2^n}$ is an $(n)$-dimensional tolerance, then
$ \rrot{n}{e}{j,l}: R \to R. $

\item Let $\delta \in \Con(\A)$. If each $(j)$-supporting line of $\gamma \in \A^{2^n}$ is a $\delta$-pair, then each $(l)$-supporting line of $\rrot{n}{e}{j,l}(\gamma)$ is a $\delta$-pair.

\item Let $\delta \in \Con(\A)$. If each $(j)$-cross section line of $\gamma \in \A^{2^n}$ is a $\delta$-pair, then each $(l)$-cross section line of $\rrot{n}{e}{j,l}(\gamma)$ is a $\delta$-pair.

\end{enumerate}
\end{lem}

\begin{proof}
Each of these properties follows directly from Definition \ref{def:rotations}. Let us establish them in order. Take $\gamma$ and $f$ as in the assumptions of \emph{(1)} and $\epsilon_0, \dots, \epsilon_{\sigma(t) -2}$ as in Definition \ref{def:rotations}. We compute 
\begin{align*}
\squares_{j,l}(\rrot{n}{e}{j,l}(\gamma))_f 
&= \squares_{j,l}(t_e(\gamma, \epsilon_0, \dots, \epsilon_{\sigma(t)-2}))_f \\
&= t_e(\squares_{j,l}(\gamma)_f, \squares_{j,l}(\epsilon_0)_f, \dots, \squares_{j,l}(\epsilon_{\sigma(t) -2})_f) \\
&=  \begin{tikzpicture}
    [ baseline=(center.base), font=\small,
      every node/.style={inner sep=0.25em}, scale=1 ]
    \node at (0.5,-0.5) (center) {\phantom{$\cdot$}}; % phantom = black magic
    \path (0,0)   node (nw) {$t_e(c_f, \psi_e(a_f,c_f))$}
      -- ++(3.3,0)  node (ne) {$t_e(d_f, \psi_e(a_f,c_f))$}
      -- ++(0,-1) node (se) {$t_e(b_f, \phi_e(a_f,c_f))$}
      -- ++(-3.3,0) node (sw) {$t_e(a_f, \phi_e(a_f,c_f))$};
    \draw (nw) -- (ne) -- (se) -- (sw) -- (nw);    
    \draw (sw) edge node[left] {$\parallel$} (nw);
    \end{tikzpicture}.
\end{align*}
This establishes \emph{(1)}. To establish \emph{(2)}, it is enough to notice that the two $(n)$-cubes in question have the same $(j,l)$-cross section squares.

To establish \emph{(3)}, suppose $R \leq \A^{2^n}$ is an $(n)$-dimensional tolerance and $\gamma \in R$. The $(n)$-reflexivity and symmetry of $R$ imply that each of the 
$
 \epsilon_0, \dots, \epsilon_{\sigma(t)-2}
$ from Definition \ref{def:rotations}
also belong to $R$. Because $R$ is an $\A$-admissible relation, it follows that $\rrot{n}{e}{j,l}(\gamma) \in R$. 

Let $\gamma \in \A^{2^n}$ be such that every $(j)$-supporting line is a $\delta$-pair. To establish \emph{(4)} and \emph{(5)} we analyze the $(j,l)$-cross section squares of $\gamma$. Let $f \in 2^{n \setminus \{j,l\} } \setminus \{\textbf{1} \}$ and suppose that 
\[
 \squares_{j,l}(\gamma)_f = 
 \begin{tikzpicture}
    [ baseline=(center.base), font=\small,
      every node/.style={inner sep=0.25em}, scale=1 ]
    \node at (0.5,-0.5) (center) {\phantom{$\cdot$}}; % phantom = black magic
    \path (0,0)   node (nw) {$c_f$}
      -- ++(1,0)  node (ne) {$d_f$}
      -- ++(0,-1) node (se) {$b_f$}
      -- ++(-1,0) node (sw) {$a_f$};
    \draw (nw) -- (ne) -- (se) -- (sw) -- (nw);
    \draw (se) edge[bend left]  (sw);
    \draw (ne) edge[bend right] (nw);
  \end{tikzpicture},
 \]
where the each curved line indicates a $\delta$-pair.
It follows that 

\[
 \squares_{j,l}(\rrot{n}{e}{j,l}(\gamma))_f = 
  \begin{tikzpicture}
    [ baseline=(center.base), font=\small,
      every node/.style={inner sep=0.25em}, scale=1 ]
    \node at (0.5,-0.5) (center) {\phantom{$\cdot$}}; % phantom = black magic
    \path (0,0)   node (nw) {$t_e(c_f, \psi_e(a_f,c_f))$}
      -- ++(3.3,0)  node (ne) {$t_e(d_f, \psi_e(a_f,c_f))$}
      -- ++(0,-1) node (se) {$t_e(b_f, \phi_e(a_f,c_f))$}
      -- ++(-3.3,0) node (sw) {$t_e(a_f, \phi_e(a_f,c_f))$};
    \draw (nw) -- (ne) -- (se) -- (sw) -- (nw);
    
    \draw (sw) edge node[left] {$\parallel$} (nw);
    \draw (se) edge[bend left]  (sw);
    \draw (ne) edge[bend right] (nw);
    
    \end{tikzpicture}.
\]
Because $\delta$ is transitive, we conclude that all of the vertex labels of the above square belong to the same $\delta$-class. In particular, each column determines a $\delta$-pair.  The remaining $(l)$-supporting line of $\rrot{n}{e}{j,l}(\gamma)$ is the $(l)$-supporting line of $\squares_{j,l}(\rrot{n}{e}{j,l}(\gamma))_\textbf{1}$, which is constant and therefore is a $\delta$-pair. Therefore, \emph{(4)} holds. Similar reasoning shows that if the $(j)$-pivot line of $\gamma$ is also a $\delta$-pair, then so is the $(l)$-pivot line of $\rrot{n}{e}{j,l}(\gamma)$. This proves \emph{(5)}.

\end{proof}

\begin{defn}[companions]\label{def:companions}
Let $\var$ be a Taylor variety with Taylor term $t$ and associated terms $t_0, \dots, t_{\sigma(t)-1}$. Let $\A \in \var$, $n \geq 2$, $e \in \sigma(t)$, and $j, k, l \in n$ with $j \neq l$ and $j\neq k$. Let $\gamma \in A^{2^n}$ and suppose that the $(j,l)$-pivot square of $\gamma$ is 
\[
\squares_{j,l}(\gamma)_\textbf{1} = \Square[a][c][b][d].
\]
Define  \textbf{$e$-th $(j,l, k)$ companion of $\gamma$}
as 
\[ \ccom{e}{j,l,k}(\gamma) = t_e(\refl_k^1(\gamma), \epsilon_0, \dots, \epsilon_{\sigma(t)-2}),
\]
where for $ s \in e $, 

\[ \epsilon_s = \begin{cases} 
\cube_k(a,d) & \text{ if the $s$th variable of $\psi_e$ is $x$, and} \\[20pt]
\cube_k(c,d) & \text{ if the $s$th variable of $\psi_e$ is $y$, }
\end{cases}
\]
and for $ s \in \sigma(t)\setminus{e}$
\[ \epsilon_s = \begin{cases} 
\cube_k(a,c) & \text{ if the $s$th variable of $\psi_e$ is $x$, and} \\[20pt]
\cube_k(c,c) & \text{ if the $s$th variable of $\psi_e$ is $y$. }
\end{cases}
\]

\end{defn}

\begin{lem}[Basic companion properties]\label{lem:companionproperties}
Let $\var$ be a Taylor variety with Taylor term $t$ and associated terms $t_0, \dots, t_{\sigma(t)-1}$. Let $\A \in \var$, $n \geq 2$, $\gamma \in A^{2^n}$,  and $j,k,l \in n$ with $j \neq l$ and $j \neq k$.
\begin{enumerate}
\item If the $(j)$-pivot line of $\ccom{e}{j,l,k}(\gamma)$ is a $\delta$-pair for all $e \in \sigma(t)$, then the $(j)$-pivot line of $\gamma$ is a $\delta$-pair.

\item For each $e \in \sigma(t)$, the $(j)$-pivot line of $\faces_k^0(\ccom{e}{j,l,k}(\gamma))$ is equal to the $(j)$-pivot line of $\rrot{n}{e}{j,l}(\gamma)$.
\item Let $e\in \sigma(t)$ and $\delta \in \Con(\A)$. If every $j$-supporting line of $\faces_k^1(\gamma)$ is a $\delta$-pair, 
then the $(j)$-supporting lines of $\faces_k^0(\ccom{e}{j,l,k}(\gamma))$ and \newline$\faces_k^1(\ccom{e}{j,l,k}(\gamma))$ are $\delta$-pairs.
\item If $\theta \in \Con(\A)$ and  $R \leq \A^{2^n}$ is an $n$-dimensional tolerance of $\A$ such that $M(\theta, \dots, \theta) \leq R \leq \rect(\theta, \dots, \theta)$, then 
$\ccom{e}{j,l,k}: R \to R
$ for all $e \in \sigma(t)$. 
\end{enumerate}

\end{lem}

\begin{proof}

We will prove \emph{(1)}, \emph{(2)}, and \emph{(3)} by analyzing the $(k,j)$-cross section squares of $\ccom{e}{j,l,k}(\gamma)$. Suppose that the $(j,l)$-pivot square of $\gamma$ is
\[
\squares_{j,l}(h)_\textbf{1} = \Square[a][c][b][d].
\]
To prove \emph{(2)} and \emph{(3)} we analyze the $(k,j)$-cross section squares of $\ccom{e}{j,l,k}(\gamma)$. With a proof of \emph{(3)} in mind, assume that every $(j)$-supporting line of $\faces_i^1(\gamma)$ is a $\delta$-pair. This means that the $(k,j)$-cross-section squares of $\gamma$ are of the form

\[
 \squares_{k,j}(\gamma) = 
\underbrace{ 
 \left\{
 \begin{tikzpicture}
    [ baseline=(center.base), font=\small,
      every node/.style={inner sep=0.25em}, scale=1 ]
    \node at (0.5,-0.5) (center) {\phantom{$\cdot$}}; % phantom = black magic
    \path (0,0)   node (nw) {$w_f$}
      -- ++(1,0)  node (ne) {$z_f$}
      -- ++(0,-1) node (se) {$v_f$}
      -- ++(-1,0) node (sw) {$u_f$};
    \draw (nw) -- (ne) -- (se) -- (sw) -- (nw);
    \draw (ne) edge[bend left] node[right] {$\delta$} (se);
  \end{tikzpicture}
: f \in 2^{n \setminus{\{k,j\}}} \setminus{ \{\textbf{1}\}} \right\}  
}_{\text{supporting}}
\union
\underbrace{
\left\{
 \begin{tikzpicture}
    [ baseline=(center.base), font=\small,
      every node/.style={inner sep=0.25em}, scale=1 ]
    \node at (0.5,-0.5) (center) {\phantom{$\cdot$}}; % phantom = black magic
    \path (0,0)   node (nw) {$w_{\textbf{1}}$}
      -- ++(1.5,0)  node (ne) {$z_\textbf{1} = d$}
      -- ++(0,-1) node (se) {$v_\textbf{1} =c$}
      -- ++(-1.5,0) node (sw) {$u_{\textbf{1}}$};
    \draw (nw) -- (ne) -- (se) -- (sw) -- (nw);
  \end{tikzpicture}
 \right\}
}_{\text{pivot}} 
 ,
\]
where a curved line indicates a $\delta$-pair.
By definition, the $(k,j)$-cross-section squares of $\refl_k^1(\gamma)$ are 
\[
 \squares_{k,j}(\refl_k^1(\gamma)) = 
 \left\{
 \begin{tikzpicture}
    [ baseline=(center.base), font=\small,
      every node/.style={inner sep=0.25em}, scale=1 ]
    \node at (0.5,-0.5) (center) {\phantom{$\cdot$}}; % phantom = black magic
    \path (0,0)   node (nw) {$z_f$}
      -- ++(1,0)  node (ne) {$z_f$}
      -- ++(0,-1) node (se) {$v_f$}
      -- ++(-1,0) node (sw) {$v_f$};
    \draw (nw) -- (ne) -- (se) -- (sw) -- (nw);
    \draw (ne) edge[bend left] node[right] {$\delta$} (se);
    \draw (nw) edge[bend right]  (sw);
  \end{tikzpicture}
: f \in 2^{n \setminus{\{k,j\}}} \setminus{ \{\textbf{1}\}} \right\}  
\union
\left\{
 \begin{tikzpicture}
    [ baseline=(center.base), font=\small,
      every node/.style={inner sep=0.25em}, scale=1 ]
    \node at (0.5,-0.5) (center) {\phantom{$\cdot$}}; % phantom = black magic
    \path (0,0)   node (nw) {$d$}
      -- ++(1.5,0)  node (ne) {$d$}
      -- ++(0,-1) node (se) {$c$}
      -- ++(-1.5,0) node (sw) {$c$};
    \draw (nw) -- (ne) -- (se) -- (sw) -- (nw);

  \end{tikzpicture}
 \right\}.
\]
Set $\theta_e(c,d) = (\underbrace{d, \dots, d}_{\text{length}=e}, c, \dots, c ) \in A^{\sigma(t)-1}$.
It follows from Definition \ref{def:tccomm}  that the $(k,j)$-cross-section squares of $\ccom{e}{j,l,k}(\gamma)$ are

\begin{align*}
 \squares_{k,j}(\ccom{e}{j,l,k}(\gamma))
 &= 
\underbrace{
\left\{
  \begin{tikzpicture}
    [ baseline=(center.base), font=\small,
      every node/.style={inner sep=0.25em}, scale=1 ]
    \node at (0.5,-0.5) (center) {\phantom{$\cdot$}}; % phantom = black magic
    \path (0,0)   node (nw) {$t_e(z_f, \psi_e(a,c))$}
      -- ++(3.1,0)  node (ne) {$t_e(z_f, \theta_e(c,d))$}
      -- ++(0,-1) node (se) {$t_e(v_f, \theta_e(c,d))$}
      -- ++(-3.1,0) node (sw) {$t_e(v_f, \psi_e(a,c))$};
    \draw (nw) -- (ne) -- (se) -- (sw) -- (nw);
      \draw (ne) edge[bend left] node[right] {$\delta$} (se);
    \draw (nw) edge[bend right] (sw);
    \end{tikzpicture}
 : f \in 2^{n \setminus{\{k,j\}}} \setminus{\{\textbf{1}\}} 
 \right\}
 }_{\text{supporting}} \\
&\union
\underbrace{
\left\{
   \begin{tikzpicture}
    [ baseline=(center.base), font=\small,
      every node/.style={inner sep=0.25em}, scale=1 ]
    \node at (0.5,-0.5) (center) {\phantom{$\cdot$}}; % phantom = black magic
    \path (0,0)   node (nw) {$t_e(d, \psi_e(a,c))$}
      -- ++(3.1,0)  node (ne) {$t_e(d, \theta_e(c,d))$}
      -- ++(0,-1) node (se) {$t_e(c, \theta_e(c,d))$}
      -- ++(-3.1,0) node (sw) {$t_e(c, \psi_e(a,c))$};
    \draw (nw) -- (ne) -- (se) -- (sw) -- (nw);
    %\draw (nw) edge[bend right]  (sw);
    \end{tikzpicture}
 \right\}
 }_{\text{pivot}} .
\end{align*}
The set of left columns appearing above is precisely the set of $(j)$-cross section lines of $\faces_k^0(\ccom{e}{j,l,k}(\gamma))$. Similarly, the set of right columns is precisely the set of $(j)$-cross section lines of $\faces_k^1(\ccom{e}{j,l,k}(\gamma)$. Moreover, this identification respects the property of a line being supporting or pivot. This proves \emph{(2)} and \emph{(3)}. 

Now we establish \emph{(1)}. 
We want to show that $\langle c, d \rangle  \in \delta$, assuming that the $(j)$-pivot line of $\ccom{e}{j,l,k}(\gamma)$ is a $\delta$-pair for every $e \in \sigma(t)$. Evidently, the $(j)$-pivot line of $\ccom{e}{j,l,k}(\gamma)$ is 
\begin{align*}
\lines_j(\ccom{e}{j,l,k}(\gamma))_\textbf{1}& = 
\langle t_e(c,\underbrace{d, \cdots, d}_{\text{length } e}, c, \cdots, c), t_e(d, \underbrace{d, \cdots, d}_{\text{length } e}, c, \cdots, c) \rangle \\
&= \langle t( \underbrace{d , \cdots, d}_{\text{length }e}, c, c, \cdots, c) , t( \underbrace{d , \cdots, d}_{\text{length }e}, d, c, \cdots, c)
\rangle,
\end{align*}
where the second equality follows from the fact that $t_e$ is obtained from $t$ by switching the $0$th and $e$th coordinates. Each such pair belongs to $\delta$, so the elements of the chain 
\[
c = t(c,c, \cdots, c) \equiv_\delta t(d, c, \dots, c) \equiv_\delta \cdots \equiv_\delta t(d, \cdots, d,d) = d
\]
all belong to the same $\delta$-class,
where the outermost equalities follow from the idempotence of $t$. 

Now we prove \emph{(4)}. Suppose that the assumptions hold and take $e \in \sigma(t)$. We want to show that $\ccom{e}{j,l,k}(\gamma) \in R$ for $\gamma \in R$. Take $\epsilon_0, \dots, \epsilon_{\sigma(t)-2}$ as in Definition \ref{def:companions}. The assumption that $R \leq \rect (\theta, \dots, \theta)$ implies that each of these $(n)$-cubes belongs to $M(\theta, \dots, \theta)$, which is assumed to be a subset of $R$. Also, the assumption that $R$ is an $(n)$-dimensional tolerance implies that $\refl_k^1(\gamma) \in R$. Therefore,  $
\ccom{e}{j,l,k}(\gamma) = t_e(\refl_k^1(\gamma), \epsilon_0, \dots, \epsilon_{\sigma(t) -2} ) \in R.
$

\end{proof}

\subsection{Proof of H=TC}\label{subsec:H=TC}

\begin{lem}\label{lem:shiftrotpivotchain}

Let $\var$ be a variety with a Taylor term $t$ and associated terms $t_0, \dots, t_{\sigma(t) -1}$. Let $\A \in \var$, let $\theta, \delta \in \Con(\A)$, and suppose $i,j,l \in n$ are distinct. Let $n\geq 2$ and suppose $R$ is an $n$-dimensional tolerance of $\A$ such that
$ R$ has $(\delta, j)$-centrality and
$M(\theta, \dots, \theta) \leq R \leq \rect(\theta, \dots, \theta)$.
  Let $\mu \in R$ have the property that every $(j)$-supporting line of 
$
\faces_i^1(\mu)
$
is a $\delta$-pair. If for all $e \in \sigma(t)$ the $(l)$-pivot line of $\rrot{n}{e}{j,l}(\mu)$ is a $\delta$-pair, then the $(j)$-pivot line of $\mu$ is a $\delta$-pair. . 
\end{lem}

\begin{proof}

By \emph{(1)} of Lemma \ref{lem:companionproperties}, it suffices to show that the $(j)$-pivot line of $\ccom{e}{j,l,i}(\mu)$ is a $\delta$-pair for all $e \in \sigma(t)$. By \emph{(4)} of Lemma \ref{lem:companionproperties} and the assumption that $R$ has $(\delta, j)$-centrality, we will be finished if we can show that every $(j)$-supporting line of 
$\ccom{e}{j,l,i}(\mu)$ is a $\delta$-pair, for all $e \in \sigma(t)$. In view of \emph{(2)} and \emph{(3)} of Lemma \ref{lem:companionproperties}, we need only show that the $(j)$-pivot line of $\rrot{n}{e}{j,l}(\mu)$ is a $\delta$-pair, for all $e \in \sigma(t)$.

Suppose that the $(j,l)$-pivot square of $\mu$ is 
\[
\begin{tikzpicture}
    [ baseline=(center.base), font=\small,
      every node/.style={inner sep=0.25em}, scale=1 ]
    \node at (0.5,-0.5) (center) {\phantom{$\cdot$}}; % phantom = black magic
    \path (0,0)   node (nw) {$c$}
      -- ++(1,0)  node (ne) {$d$}
      -- ++(0,-1) node (se) {$b$}
      -- ++(-1,0) node (sw) {$a$};
    \draw (nw) -- (ne) -- (se) -- (sw) -- (nw);
    \draw (sw) edge[bend right] node[below] {$\delta$} (se);
  \end{tikzpicture},
\]
where the curved line indicating a $\delta$-pair follows from the assumptions. By \emph{(1)} of Lemma \ref{lem:shiftrotdef}, the $(j,l)$-pivot square of 
$\rrot{n}{e}{j,l}(\mu) $ is 
\[
  \begin{tikzpicture}
    [ baseline=(center.base), font=\small,
      every node/.style={inner sep=0.25em}, scale=1 ]
    \node at (0.5,-0.5) (center) {\phantom{$\cdot$}}; % phantom = black magic
    \path (0,0)   node (nw) {$t_e(c ,\psi_e(a,c))$}
      -- ++(3.3,0)  node (ne) {$t_e(d, \psi_e(a,c))$}
      -- ++(0,-1) node (se) {$t_e(b, \phi_e(a,c))$}
      -- ++(-3.3,0) node (sw) {$t_e(a, \phi_e(a,c))$};
    \draw (nw) -- (ne) -- (se) -- (sw) -- (nw);
    %\draw (sw) edge[bend left] (nw);
    %\draw (nw) edge[bend left] (ne);
    %\draw (ne) edge[bend left] (se);
    \draw (se) edge[bend left] node[below] {$\delta$} (sw);
   \draw (sw) edge node[left] {$\parallel$} (nw);
    \end{tikzpicture} ,
\]
for each $e \in \sigma(t)$.
The right column and top row of the above square are respectively the $(l)$-pivot line and $(j)$-pivot line of $\rrot{n}{e}{j,l}(\mu)$. We assume that the $(l)$-pivot line is a $\delta$-pair, so it follows that  the $(j)$-pivot line of $\rrot{n}{e}{j,l}(\mu)$ is a $\delta$-pair. 

\end{proof}

We need to consider certain compositions of rotations and will use finite trees for bookkeeping. Assume that $t$ is a Taylor term of arity $\sigma(t)$ for some variety $\var$ and let $n \geq 2$. Set

\[
\mathbb{D}_n = \langle  \sigma(t)^{ < n}; \leq \rangle,
\]
where $ \sigma(t)^{ < n} = \Union \{ \sigma(t)^i : i \in n \}$  and two sequences $d_1$, $d_2$ are $\leq$-related when $d_1 \subseteq d_2$. Note that $\mathbb{D}_n$ has the empty sequence $\emptyset$ as its root. For $\A \in \var$ and $\gamma \in A^{2^n}$, set $\gamma^\emptyset = \gamma$. We recursively define $\gamma^d = \rrot{n}{d_i}{i,i+1}(\gamma^c)$, where $d =(d_0, \dots , d_i) \in \mathbb{D}_{n}$ is non-empty and $c$ is the predecessor of $d$.

\begin{lem}\label{lem:treesuccessorrotations}
Let $\var$ be a variety with a Taylor term $t$ and associated terms \newline $t_0, \dots, t_{\sigma(t)-1}$. Let $\A \in \var$ and $R \leq \A^{2^n}$ be an $n$-dimensional tolerance for some $n\geq 2$. If $d \in \mathbb{D}_{n}$ is a tuple of length $i \in n$, then
\begin{enumerate}

\medskip
\item 
$\gamma^d \in R$, and 
\medskip

\item
if $f \in 2^{n \setminus \{i\}}$ satisfies $f(j) = 0$ for some $j\in i$, then the $(i)$-cross-section line of $\gamma^d$ at $f$
is a constant pair.

\end{enumerate}

\end{lem}

\begin{proof}
We proceed by induction. The result is trivially true for $\gamma^{\emptyset} = \gamma$. Suppose that it holds for a tuple of length $c \in \mathbb{D}_{n}$ of length $i \in n-1$ and let $d=(d_0, \dots, d_i)$ be a successor of $c$. Set 
$
\gamma^d = \rrot{n}{d_i}{i,i+1}(\gamma^c). 
$
Notice that \emph{(3)} of Lemma \ref{lem:shiftrotdef} guarantees that $\gamma^d \in R$. 

Now let $f \in 2^{n\setminus{i+1}}$ be such that $f(j) = 0$ for some $j \in i+1$, and let $f^*$ be the restriction of $f$ to the set $n\setminus{\{i, i+1\}}$.  There are two cases to consider.

\begin{enumerate}
\item[Case 1:] Suppose $j \notin i$, in which case $f(i)=0$. If 
\[
\squares_{i,i+1}(\gamma^c)_{f^*} = \Square[u][w][v][z],
\]  
then it follows from \emph{(1)} of Lemma \ref{lem:shiftrotdef} that 
\[
\squares_{i,i+1}(\gamma^d)_{f^*} = 
 \begin{tikzpicture}
    [ baseline=(center.base), font=\small,
      every node/.style={inner sep=0.25em}, scale=1 ]
    \node at (0.5,-0.5) (center) {\phantom{$\cdot$}}; % phantom = black magic
    \path (0,0)   node (nw) {$t_{d(i)}(w, \psi_{d(i)}(u,w))$}
      -- ++(3.3,0)  node (ne) {$t_{d(i)}(z, \psi_{d(i)}(u,w))$}
      -- ++(0,-1) node (se) {$t_{d(i)}(v, \phi_{d(i)}(u,w))$}
      -- ++(-3.3,0) node (sw) {$t_{d(i)}(u, \phi_{d(i)}(u,w))$};
    \draw (nw) -- (ne) -- (se) -- (sw) -- (nw);
    
    %\draw (sw) edge[bend left] (nw);
    %\draw (nw) edge[bend left] (ne);
    %\draw (ne) edge[bend left] (se);
    %\draw (se) edge[bend left] (sw);
    
    %\draw (sw) edge[bend left] (nw);
    %\draw (nw) edge[bend left] (ne);
    %\draw (ne) edge[bend left] (se);
    %\draw (se) edge[bend left] (sw);
    
    %\draw (se) edge node[below] {$=$} (sw);
    \draw (sw) edge node[left] {$\parallel$} (nw);
    \end{tikzpicture}.
\]
The left column of the above square is equal to the $(i+1)$-cross-section line of $\gamma^d$ at $f$ and it is a constant pair, as claimed.  
\item[Case 2:] Suppose $ j \in i$. In this case we apply the inductive assumption that $(2)$ holds for $\gamma^c$ and conclude that 
\[
\squares_{i, i+1}(\gamma^c)_{f^*} = \Square[u][w][u][w].
\]
Again, we apply \emph{(1)} of Lemma \ref{lem:shiftrotdef} and conclude that
\[
\squares_{i,i+1}(\gamma^d)_{f^*} = 
\begin{tikzpicture}
    [ baseline=(center.base), font=\small,
      every node/.style={inner sep=0.25em}, scale=1 ]
    \node at (0.5,-0.5) (center) {\phantom{$\cdot$}}; % phantom = black magic
    \path (0,0)   node (nw) {$t_{d(i)}(w, \psi_{d(i)}(u,w))$}
      -- ++(3.3,0)  node (ne) {$t_{d(i)}(w, \psi_{d(i)}(u,w))$}
      -- ++(0,-1) node (se) {$t_{d(i)}(u, \phi_{d(i)}(u,w))$}
      -- ++(-3.3,0) node (sw) {$t_{d(i)}(u, \phi_{d(i)}(u,w))$};
    \draw (nw) -- (ne) -- (se) -- (sw) -- (nw);
    \draw (sw) edge node[left] {$\parallel$} (nw);
     \draw (se) edge node[right] {$\parallel$} (ne);
    \end{tikzpicture}.
\]
The $(i+1)$-cross-section line of $\gamma^d$ at $f$ is either the left column or right column of the above square, and each of these columns is a constant pair. This finishes the proof. 
\end{enumerate}  
\end{proof}

\begin{prop}[Perpendicular Stage]\label{prop:perpstage}
Let $\var$ be a variety with a Taylor term $t$ and associated terms $t_0, \dots, t_{\sigma(t) -1}$. Let $\A \in \var$, let $\theta, \delta \in \Con(\A)$ and choose $n\geq 2$. Suppose $R$ is an $n$-dimensional tolerance of $\A$ such that
$M(\theta, \dots, \theta) \leq R \leq \rect(\theta, \dots, \theta)$ 
and 
$R$ has $(\delta, k)$-centrality for all $k \in n$. Let $j \in n$. Then, $R^{\circ_j}$ has $(\delta, i)$-centrality for each $ i \in n$ with $i \neq j$.
\end{prop}

\begin{proof}
First, observe that any permutation of coordinates $\sigma \in S_n$ induces an automorphism $\sigma: \A^{2^n} \to \A^{2^n}$, and an $(n)$-dimensional tolerance $R$ has $(\delta, i)$-centrality if and only if the image of $R$ under the induced automorphism has $(\delta, \gamma(i))$-centrality. Furthermore, $\sigma(R^{\circ_j}) = \sigma(R)^{\circ_{\sigma(j)}}$. Therefore, we need only to consider the case when $i = 0$ and $j=n-1$. 

So, let $\gamma \in R^{\circ_{n-1}}$ be such that every $(0)$-supporting line is a $\delta$-pair. Our task is to show that the $(0)$-pivot line of $\gamma$ is also a $\delta$-pair. By the definition of $R^{\circ_{n-1}}$, there are $\mu_0, \dots, \mu_{s-1} \in A^{2^{n-1}}$ so that 
\begin{enumerate}
\item $\faces_{n-1}^0(\gamma) = \mu_0$,
\item $ \faces_{n-1}^1(\gamma) = \mu_{s-1}$, and
\item $\glue_{\{n-1\}}(\langle \mu_r, \mu_{r+1} \rangle ) \in R$, for each 
$r \in s-1$.
\end{enumerate}

\begin{claim}\label{claim:perpclaim1}

Take $d \in \mathbb{D}_{n-1}$ to be a leaf. The sequence 
$
(\mu_0)^d, \dots, (\mu_{s-1})^d 
$
satisfies
\begin{enumerate}

\item for all $r \in s$, each $(n-2)$-supporting line of $(\mu_r)^d$ is a $\delta$-pair,
\item every $(n-2)$-cross section line of $(\mu_0)^d $ is a $\delta$-pair, and
\item $\glue_{\{n-1\}}(\langle (\mu_r)^d, (\mu_{r+1})^d \rangle) \in R$,
 for all 
$r \in s-1$.

\end{enumerate}
\end{claim}

\begin{claimproof}
Suppose $d = (d_0, \dots, d_{n-3})$ is a nonempty leaf ($\mathbb{D}_n$ is the graph consisting of a single vertex when $n=2$ and the claim holds in this case). The first property of the claim follows from \emph{(2)} of Lemma \ref{lem:treesuccessorrotations} and the fact that $\delta$ contains all constant pairs. 
To show the second property of the claim, we proceed by induction over the branch in $\mathbb{D}_{n-1}$ determined by $d = (d_0, \dots, d_{n-3})$. We assume $\mu_0 = \faces_{n-1}^0(\gamma)$ and that every $(0)$-supporting line of $\gamma$ is a $\delta$-pair. It follows that every $(0)$-cross section line of $\mu_0 = (\mu_0)^\emptyset$ is a $\delta$-pair, which establishes the basis of the induction. Now let $(d_0, \dots, d_i)$ be an ancestor of $d$ and suppose that every $(i+1)$-cross section line of $(\mu_0)^{(d_0, \dots, d_i)}$ is a $\delta$-pair. Now \emph{(5)} of Lemma \ref{lem:shiftrotdef} shows that every $(i+2)$-cross section line of $(\mu_0)^{(d_0, \dots, d_{i+1})}$ is a $\delta$-pair. In particular, every $(n-2)$-cross section line of $(\mu_0)^d$ is a $\delta$-pair as claimed. 

A similar induction using \emph{(2)} and \emph{(3)} of Lemma \ref{lem:shiftrotdef} establishes the third property of the claim.

\end{claimproof}

\begin{claim}\label{claim:perpclaim2}
If $d \in \mathbb{D}_{n-1}$ is a leaf, then the $(n-2)$-pivot line of $(\mu_{r+1})^d$ is a $\delta$-pair for all $r \in s-1$.
\end{claim}
\begin{claimproof}
The claim follows by induction on $r \in  s$. The claim holds for $r=0$ by \emph{(2)} of Claim  \ref{claim:perpclaim1}. Suppose the claim holds for $\mu_r$ for $r \in s-2$. This assumption along with \emph{(1)} and \emph{(3)} of Claim \ref{claim:perpclaim1} show that every $(n-2)$-supporting line of
\[
\glue_{\{n-1\}}(\langle (\mu_{r+1})^d, (\mu_{r+2})^d \rangle) \in R
\]
is a $\delta$-pair. We now apply the assumption that $R$ has $(\delta, n-2)$-centrality and conclude that the $(n-2)$-pivot line of this cube is also a $\delta$-pair. Because 
\[\glue_{\{n-1\}}(\langle (\mu_{r+1})^d, (\mu_{r+2})^d \rangle)
\] and $(\mu_{r+2})^d$ have the same $(n-2)$-pivot line, the claim is proved. 
\end{claimproof}

\begin{claim}\label{claim:perpclaim3}
Let $c  \in \mathbb{D}_{n-1}$ be a tuple of length $z \in n-1$. The $(z)$-pivot line of $(\mu_{s-1})^c$ is a $\delta$-pair. In particular, the $(0)$-pivot line of 
$(\mu_{s-1})^\emptyset = \mu_{s-1}$ is a $\delta$-pair.
\end{claim}
\begin{claimproof}

We proceed by an induction from the leaves of $\mathbb{D}_{n-1}$ to its root. The basis has been established by Claim \ref{claim:perpclaim2}. Suppose that the claim holds for all tuples of length $z +1 \in n-1$ belonging to $\mathbb{D}_{n-1}$ and let $c \in \mathbb{D}_{n-1}$ be a tuple of length $z$. Our assumption that every $(0)$-supporting line of $\gamma$ is a $\delta$-pair implies that every $(0)$-supporting line of $\mu_{s-1}$ is a $\delta$-pair. An induction using \emph{(4)} of Lemma \ref{lem:shiftrotdef} shows that every $(z)$-supporting line of $\mu^c$ is a $\delta$-pair. It follows that the assumptions of Lemma \ref{lem:shiftrotpivotchain} are satisfied, with 
$
\mu = \glue_{\{n-1\}}( \langle (\mu_{s-2})^c, (\mu_{s-1})^c \rangle )$, 
$i = n-1$, 
$j= z$, 
and $ l= z+1$. This completes the proof of the perpendicular stage.

\end{claimproof}

The $(0)$-pivot lines of $\gamma$ and $\mu_{s-1}$ are the same, and the conclusion of Claim \ref{claim:perpclaim3} is that the $(0)$-pivot line of $\mu_{s-1}$ is a $\delta$-pair. This is what we wanted to show, so the proof is finished.

\end{proof}

To summarize some important aspects of the proof of Proposition \ref{prop:perpstage}, we include Figure \ref{fig:perpendicularstage}. Three of the directions in $2^n$ are shown next to a picture of $\faces_{n-1}^0(\gamma)$. Each of the $(0)$-supporting lines of $\gamma$ is drawn with a solid curved line to indicate that it is a $\delta$-pair (we hope the reader will forgive us for not drawing a correct number of these.) Underneath $\gamma$ is the sequence $\mu_0, \dots, \mu_{s-1}$, where each consecutive pair glues together to form a cube belonging to $R$. This sequence is systematically rotated, with the tree $\mathbb{D}_{n-1}$ keeping track of the many possibilities. A typical sequence that is produced by a leaf $d \in \mathbb{D}_{n-1}$ is shown at the bottom of the figure with constant pairs drawn in bold (see Claim \ref{claim:perpclaim1}). The dotted curved lines on this leaf sequence indicate the application of the centrality assumption (see \ Claim \ref{claim:perpclaim2}). The final induction from the leaves of $\mathbb{D}_{n-1}$ to the root is not depicted.

\begin{figure}[tb]
\begin{center}
\includegraphics[scale=1]{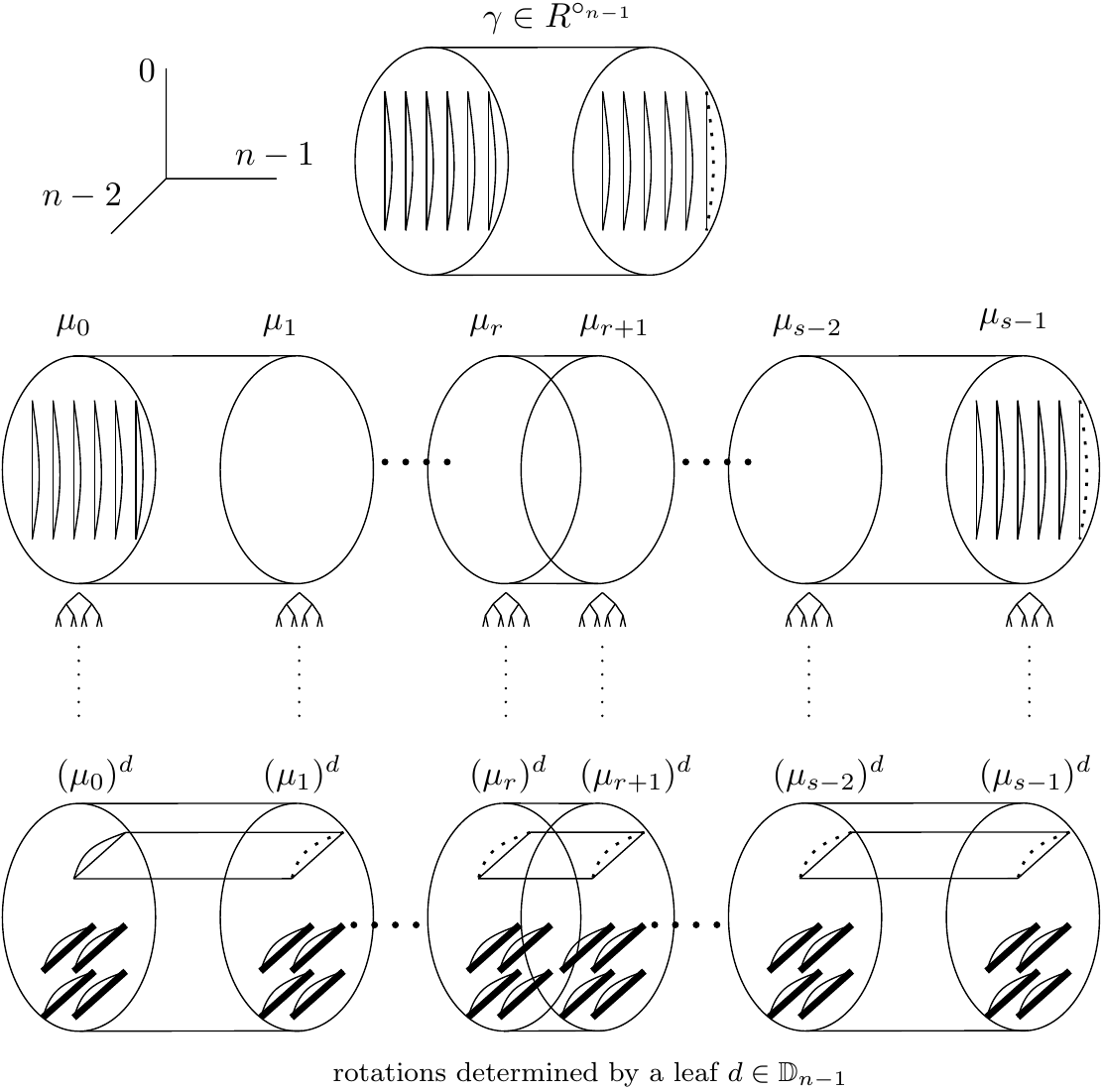}
\end{center}
\caption{Perpendicular stage }\label{fig:perpendicularstage}
\end{figure}

We now move to the parallel stage. Instead of the special compositions of rotations that we used in the perpendicular stage, we need to consider certain sequences of companions. 

\begin{prop}[Parallel Stage]\label{prop:parallelstage}
Let $\var$ be a variety with a Taylor term $t$ and associated terms $t_0, \dots, t_{\sigma(t) -1}$. Let $\A \in \var$, let $\theta, \delta \in \Con(\A)$ and choose $n\geq 2$. Suppose $R$ is an $n$-dimensional tolerance of $\A$ such that
$M(\theta, \dots, \theta) \leq R \leq \rect(\theta, \dots, \theta)$ 
and 
$R$ has $(\delta, k)$-centrality for all $k \in n$. Let $j \in n$. Then, $R^{\circ_j}$ has $(\delta, j)$-centrality.
\end{prop}

\begin{proof}
A justification similar to the one given at the beginning of the proof of Proposition \ref{prop:perpstage} allows us to consider without loss the case when $j =0$. So, let $\gamma \in R^{\circ_0}$ have the property that every $(0)$-supporting of $\gamma$ is a $\delta$-pair. We want to show that the $(0)$-pivot line of $\gamma$ is also a $\delta$-pair. This is accomplished by an induction on the tree $\mathbb{D}_n$. Set $^\emptyset \gamma = \gamma$. For $d \in \mathbb{D}_n$ a tuple of length $ i >0$, set $^d \gamma =\ccom{d_{i-1}}{0,1, i}(^c\gamma)$, where $c$ is the predecessor of $d$.

\begin{claim}\label{claim:parclaim1}
Let $d \in \mathbb{D}_n$ be a tuple of length $i \in n$. The following hold:

\begin{enumerate}

\item $^d\gamma \in R^{\circ_0}$, 
\item Every $(0)$-supporting line of $^d\gamma$ is a $\delta$-pair, and
\item $(\cut_{Q} (^d\gamma))_\textbf{1} \in M(\underbrace{\theta, \dots, \theta}_{i+1})$, where $Q = n \setminus \{0, \dots, i\}$. In particular, $^d \gamma \in M(\underbrace{\theta, \dots, \theta}_n)$ if $d$ is a leaf. 

\end{enumerate}
\end{claim}

\begin{claimproof}
We proceed by induction on the length of $d$. If $d$ is the empty tuple then $^d \gamma = \gamma$ and both \emph{(1)} and \emph{(2)} hold by assumption. Set $Q = n \setminus \{0\}$. In this case $\cut_Q(\gamma)_\textbf{1}$ is the $(0)$-pivot line of $\gamma$ and is a $\theta$-pair, because $ R \leq \rect (\theta, \dots, \theta)$. We notice that $\theta$ and $M(\theta)$ are the same $1$-dimensional relation, which establishes \emph{(3)}. Therefore, the claim holds for the root of $\mathbb{D}_n$.

Suppose that the claim holds for some $c \in \mathbb{D}_n$ and let $d = (d_0, \dots, d_{i-1})$ be a successor of $c$. We will establish the claim for $^d \gamma = \ccom{d_{i-1}}{0,1,i}(^c h)$. First, notice that it follows from our assumptions that 
\[
M(\underbrace{\theta, \dots, \theta}_n) \leq R^{\circ_0} \leq \rect (\underbrace{\theta, \dots, \theta}_n),
\]
and so \emph{(4)} of Lemma \ref{lem:companionproperties} proves \emph{(1)} of the claim. 

Next, we want to show that every $(0)$-supporting line of $^d \gamma $ is a $\delta$-pair. We assume that every $(0)$-supporting line of $^c\gamma$ is a $\delta$-pair. It follows from \emph{(3)} of Lemma \ref{lem:companionproperties} that every $(0)$-supporting line of $^d \gamma$ that is not the $(0)$-pivot line of $\faces^0_i(^d\gamma)$ is $\delta$-pair. In view of \emph{(2)} of Lemma \ref{lem:companionproperties}, we must demonstrate that the $(0)$-pivot line of $\rrot{n}{d_{i-1}}{0,1}(^c\gamma)$ is a $\delta$-pair. It follows from \emph{(3)} and \emph{(4)} of Lemma \ref{lem:shiftrotdef} that $\rrot{n}{d_{i-1}}{0,1}(^c\gamma) \in R^{\circ_0}$, and that every $(1)$-supporting line of $\rrot{n}{d_{i-1}}{0,1}(^c\gamma)$ is a $\delta$-pair. Applying Proposition \ref{prop:perpstage}, we conclude that the $(1)$-pivot line of $\rrot{n}{d_{i-1}}{0,1}(^c\gamma)$ is a $\delta$-pair. An argument similar to the one given in the proof of Lemma \ref{lem:shiftrotpivotchain} shows that the $(0)$-pivot line of $\rrot{n}{d_{i-1}}{0,1}(^c\gamma)$ is a $\delta$-pair. This proves \emph{(2)} of the claim. 

Last, we prove \emph{(3)} of the claim. We assume that \emph{(3)} holds for $^c\gamma$. Let $Q = n \setminus \{ 0, \dots, i \}$. Referring to Definition \ref{def:companions} (with $^c\gamma$ taking the place of the $\gamma$ from the definition), we compute

\begin{align*}
\cut_Q (^d\gamma)_\textbf{1} &= \cut_Q (t_{d_{i-1}}( \refl_i^1(^c\gamma), \epsilon_0, \dots, \epsilon_{\sigma(t)-2}))_\textbf{1}\\
&= t_{d_{i-1}}(\cut_Q (\refl_i^1(^c\gamma))_\textbf{1}, \cut_Q(\epsilon_0)_\textbf{1}, \dots, \cut_Q ( \epsilon_{\sigma(t) -2} )_\textbf{1}).
\end{align*}
If each of the arguments of $t_{d_{i-1}}$ in the last expression belongs to $ M(\underbrace{\theta, \dots, \theta}_{i+1}) $, then so does $\cut_Q (^d\gamma)_\textbf{1}$. We observed in the proof of Lemma \ref{lem:companionproperties} that 
$\epsilon_0, \dots, \epsilon_{\sigma(t) -2} \in M(\underbrace{\theta, \dots, \theta}_n)$, so it follows from Lemma \ref{lem:deltaprojectstodelta} that 
$\cut_Q(\epsilon_0)_\textbf{1}, \dots,\cut_Q( \epsilon_{\sigma(t) -2} )_\textbf{1} \in M(\underbrace{\theta, \dots, \theta}_{i+1})$. 

So, it remains to show that $\cut_Q (\refl_i^1(^c\gamma))_\textbf{1} \in M(\underbrace{\theta, \dots, \theta}_{i+1})$. Observe that 
\begin{align*}
\cut_Q (\refl_i^1(^c\gamma))_\textbf{1}  &= \refl_i^1(\cut_Q(^c\gamma)_\textbf{1})\\
&= \glue_{\{ i \}}( \langle \faces_i^1(\cut_Q(^c\gamma)_\textbf{1}), \faces_i^1(\cut_Q(^c\gamma)_\textbf{1}) \rangle ).
\end{align*}
The inductive assumption that \emph{(3)} holds for $^c\gamma$ implies that 
$
\faces_i^1(\cut_Q(^c\gamma)_\textbf{1})= \cut_{Q'}(^c\gamma)_\textbf{1} \in M(\underbrace{\theta, \dots, \theta}_{i}),
$
 where $Q' = n \setminus \{0, \dots, i -1 \}$. Call this $(i)$-cube $\zeta$. We have shown that 
$
\cut_Q (\refl_i^1(^c\gamma))_\textbf{1} = \glue_{i}(\langle \zeta,\zeta \rangle )
$
for $\zeta \in M(\underbrace{\theta, \dots, \theta}_i)$. This proves that $\cut_Q (\refl_i^1(^c\gamma))_\textbf{1} \in M(\underbrace{\theta, \dots, \theta}_{i+1})$.

\end{claimproof}
Using Claim \ref{claim:parclaim1}, we are now able to prove the following claim, which finishes the proof of the parallel stage.
\begin{claim}\label{claim:parclaim2}
Let $c  \in \mathbb{D}_{n}$ be a tuple of length $i \in n$. The $(0)$-pivot line of $^c\gamma$ is a $\delta$-pair. In particular, the $(0)$-pivot line of 
$^\emptyset\gamma = \gamma$ is a $\delta$-pair.
\end{claim}

\begin{claimproof}
Let $d \in \mathbb{D}_n$ be a leaf. It follows from \emph{(2)} and \emph{(3)} of Claim \ref{claim:parclaim1} along with the assumptions that $R$ has $(\delta, 0)$-centrality and $M(\theta, \dots, \theta) \leq R$ that the $(0)$-pivot line of $^dh$ is a $\delta$-pair. This establishes the basis of an induction from the leaves of $\mathbb{D}_n$ to the root. For the inductive step, suppose the claim holds for all tuples of length $i+1 \in n$ and let $c \in \mathbb{D}_n$ be a tuple of length $i$. The inductive assumption implies that the $(0)$-pivot line of 
$
^d\gamma = \ccom{d_{i}}{0,1,i+1}(^c\gamma)
$
is a $\delta$-pair for every successor $d$ of $c$. We now apply \emph{(1)} of Lemma \ref{lem:companionproperties} and conclude that the $(0)$-pivot line of $^c\gamma$ is also a $\delta$-pair. 
\end{claimproof}

\end{proof}

Some of the important aspects of the proof of Proposition \ref{prop:parallelstage} in the case $n=4$ are depicted in Figure \ref{fig:parallelstage}. We orient the coordinates $\{0, 1, 2, 3\}$ in this order: horizontal, vertical, out of page, and inward. The top left cube is a picture of $\gamma$, and the diagram illustrates a typical $^d\gamma$ as $d$ progresses through $\mathbb{D}_4$. The bold projections indicate membership in $M(\theta, \theta)$, $M(\theta, \theta, \theta)$, and $M(\theta, \theta, \theta, \theta)$. The bold solid curved lines are equal to the $(0)$-pivot line of the appropriate rotated cube and are $\delta$-related as a consequence of the perpendicular stage. The centrality assumption applies to $^d\gamma$ whenever $d$ is a leaf of $\mathbb{D}_4$, and so the $(0)$-pivot line of such cubes is a $\delta$-related pair. This is indicated with a curved dashed line. An induction from the leaves to the root of $\mathbb{D}_4$ shows that the $(0)$-pivot line of $\gamma$ is also $\delta$-related.  

\begin{figure}
\centering
\includegraphics[scale=.8]{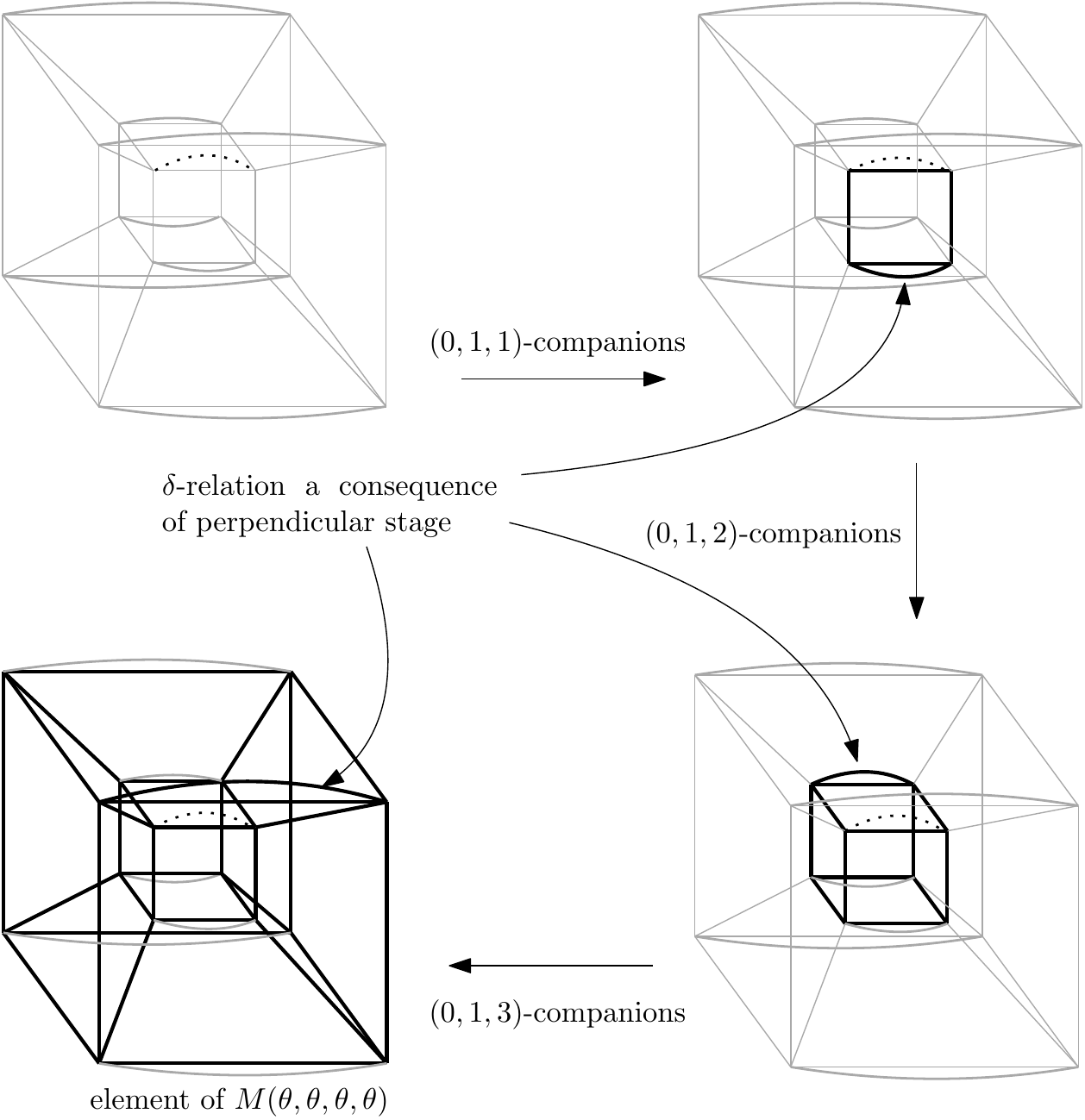}
\caption{$(4)$-dimensional parallel stage}\label{fig:parallelstage}
\end{figure}

We are now ready to prove one of our main results.

\begin{thm}\label{thm:allarityH=TC}
Let $\var$ be a Taylor variety, $\A \in \var$, $\theta \in \Con(\A)$. For every $n \geq 2$, 
\[
[\underbrace{\theta, \dots, \theta}_n ]_{TC} = [\underbrace{\theta, \dots, \theta}_n]_H.
\]
\end{thm}
\begin{proof}

By Theorem \ref{thm:basicpropertiescommutator}, 
$
[\theta, \dots, \theta ]_{TC} \leq [\theta, \dots, \theta]_H.
$
 We show that $[\theta, \dots \theta]_H \leq [\theta, \dots \theta]_{TC}$. Set $\delta = [\theta, \dots, \theta]_{TC}$. It suffices to check that $\Delta(\theta, \dots, \theta)$ has $(\delta, i)$-centrality, for each $i \in n$. 

We proceed by induction. For each $j \in \nat$ set  $R_j = \tc_j(M(\theta, \dots, \theta))$. It follows inductively from \emph{(1)} of Lemma \ref{prop:hcongenerate} that each $R_j$ is an $(n)$-dimensional tolerance such that $M(\theta, \dots, \theta) \leq R_j \leq \rect(\theta, \dots, \theta)$. Using this, it follows inductively from Propositions \ref{prop:perpstage} and \ref{prop:parallelstage} (the perpendicular and parallel stages) that each $R_j$ has $(\delta, i)$-centrality for ever $i \in n$. Because 
$
\Delta(\theta, \dots, \theta) = \Union_{j \in \nat} R_j,
$
the proof is finished.

\end{proof}

\subsection{Properties of the hypercommutator}\label{subsec:HHC8}

Now we state and prove a relational characterization of the hypercommutator that generalizes Propositions \ref{prop:binhypchar} and \ref{prop:ternaryhypchar}. Let $\A$ be an algebra, $S\finsub \nat$ with $|S| \geq 2$, and $i \in S$. We say a pair $\langle x, y \rangle \in A^2$ is \textbf{$(i)$-supported by $\gamma \in A^{2^S}$} if the $(i)$-pivot line of $\gamma$ is the pair $\langle x,y \rangle$ and every $(i)$-supporting line of $\gamma$ is a constant pair. 

If $\langle x,y \rangle$ is $(i)$-supported by $\gamma$ for every $i \in S$, then we say that $\langle x,y \rangle $ is \textbf{totally supported} by $\gamma$ (in which case $\gamma_f = x$ for all $f \neq \textbf{1}$ and $\gamma_\textbf{1} = y$.) In this situation we call $\gamma$ the \textbf{$(|S|)$-dimensional commutator cube} for the pair $\langle x,y \rangle$, and denote it by $\com_S(x,y)$. We also define 
$S(\gamma, \langle x,y \rangle) = \{ i \in S: \langle x,y \rangle \text{ is $(i)$-supported by $\gamma$} \}.
$

\begin{thm}\label{thm:charhypercomm}
Let $\A$ be an algebra, $n \geq 2$, and $(\theta_0, \dots, \theta_{n-1}) \in \Con(\A)^n$. The following are equivalent:

\begin{enumerate}
\item $\langle x,y \rangle \in [\theta_0, \dots, \theta_{n-1}]_H$,
\item $\com_n(x,y) \in \Delta(\theta_0, \dots, \theta_{n-1})$, and
\item there exists $i \in n$ so that $\langle x,y \rangle $ is $(i)$-supported by some $\gamma \in \Delta(\theta_0, \dots, \theta_{n-1})$.
\end{enumerate}
\end{thm}

\begin{proof}
First, we show that \emph{(2)} holds if and only if \emph{(3)} holds. Clearly, \emph{(2)} implies \emph{(3)}, so we prove that \emph{(3)} implies \emph{(2)}. Fix $ \langle x,y \rangle \in A^2$. Suppose that $\gamma \in \Delta(\theta_0, \dots, \theta_{n-1})$ is such that $S(\gamma, \langle x,y \rangle) \neq \emptyset$. If $|S(\gamma, \langle x,y \rangle)| = n$, then $(2)$ holds. Otherwise, there exists $\gamma'$ so that $|S(\gamma', \langle x,y \rangle )| = |S(\gamma, \langle x,y \rangle )| +1$. 

Indeed, pick $i \in S(\gamma, \langle x,y \rangle )$, $j \in n\setminus S(\gamma, \langle x, y \rangle )$, and let $\mu = \sym_j(\refl_i^0(\gamma))$. Because $\Delta(\theta_0, \dots, \theta_{n-1})$ is an $n$-dimensional congruence, it follows that 
\[
 \mu  \in \Delta(\theta_0, \dots, \theta_{n-1}).
\] 
Now set
$
\gamma' = \glue_{\{j\}}(\langle \faces_j^0(\mu) , \faces_j^1(\gamma) \rangle ).
$ 
Notice that $\faces_j^1(\mu) = \faces_j^0(\gamma)$, because every $(i)$-supporting line of $\gamma$ is constant. Therefore, $\gamma' \in \Delta(\theta_0, \dots, \theta_{n-1})$.

Let $ l \in S(\gamma, \langle x, y \rangle )$. We show that $ l \in S(\gamma', \langle x,y \rangle )$. We assume that $\langle x,y \rangle$ is $(l)$-supported by $\gamma$, so the $(l)$-pivot line of $\gamma$ (which is also the $(l)$-pivot line of $\faces_j^1(\gamma)$) is the pair $\langle x,y \rangle$. Therefore, the $(l)$-pivot line of $\gamma'$ is the pair $\langle x,y \rangle$. We must also show that every $(l)$-supporting line of $\gamma'$ is a constant pair. Because $l \neq j$, it follows that a particular $(l)$-supporting line of $\gamma'$ is either an $(l)$-supporting line of $\faces_j^1(\gamma)$ (which is assumed to be a constant pair) or an $(l)$-cross-section line of $\faces_j^0(\mu)$. By definition, $\faces_j^0(\mu) = \faces_j^1(\refl_i^0(\gamma)) = \refl_i^0(\faces_j^1(\gamma))$. Clearly, if $l = i$ then every $(l)$-cross-section line of $\refl_i^0(\faces_j^1(\gamma))$ is a constant pair. If $l \neq i$, then every $(l)$-cross-section line of $\refl_i^0(\faces_j^1(\gamma))$ comes from an $(l)$-cross-section line of $\faces_i^0(\faces_j^1(\gamma))$, each of which is assumed to be a constant pair. 

We observe that also $j \in S(\gamma', \langle x,y \rangle)$. To see this, notice that every $(j)$-cross-section line of $\gamma'$ is a row of some $(j,i)$-cross-section square of $\gamma'$. Let us analyze a generic square and take $f \in 2^{n \setminus \{j,i\}}$. Notice that $2^{n \setminus \{j,i \}}$ is the set of coordinates for $\lines_i(\faces_j^1(\gamma))$. Suppose that $\lines_i(\faces_j^1(\gamma))_f = \langle a, b \rangle $. It follows from the definition of $\gamma'$ that 

\[
\squares_{j,i}(\gamma')_f =  \Square[a][a][a][b] .
\]
If $f \neq \textbf{1}$, then the assumption that $\langle x,y \rangle$ is $(i)$-supported by $\gamma$ gives that $a =b$. In this case each row of the above square is a constant pair. If $f = \textbf{1}$, then $ \langle a,b \rangle = \langle x,y \rangle$. In this case, the above square becomes

\[
\squares_{j,i}(\gamma')_\textbf{1} =  \Square[x][x][x][y] .
\]

The bottom row of the above square is a constant pair. So, every $(j)$-supporting line of $\gamma'$ is a constant pair. The top row of the above square witnesses that the $(j)$-pivot line of $\gamma'$ is equal to the pair $\langle x,y \rangle$. Putting this together, we have shown that $|S(\gamma', \langle x, y \rangle ) | = |S(h, \langle x,y \rangle )| +1$. It follows by induction that \emph{(2)} holds whenever \emph{(3)} holds.

Now we show that \emph{(1)} holds if and only if \emph{(2)} holds. Denote by $\alpha$ the set of $\langle x,y \rangle $ that are totally supported by some $\gamma \in \Delta(\theta_0, \dots, \theta_{n-1})$. By definition, $\Delta(\theta_0, \dots, \theta_{n-1})$ has $([\theta_0, \dots, \theta_{n-1}], n-1)$-centrality, from which it follows that $\alpha \subseteq [\theta_0, \dots, \theta_{n-1}]_H$. So, \emph{(2)} implies \emph{(1)}. 

For the other direction, it is enough to show that $\alpha \in \Con(\A)$ and that \newline $\Delta(\theta_0, \dots, \theta_{n-1})$ has $(\alpha, n-1)$-centrality. Let us show that $\alpha$ is a congruence of $\A$. It is obvious that $\alpha \leq \A^2$. Because $\Delta(\theta_0, \dots, \theta_{n-1})$ contains all constant $\gamma \in A^{2^n}$, reflexivity of $\gamma$ is also immediate. For symmetry, take $\gamma \in A^{2^n}$ that totally supports the pair $\langle x,y \rangle$. The pair $\langle y,x \rangle $ is $(i)$-supported by $\sym_i(\gamma)$ for any $i \in n$, and the result now follows from the equivalence of \emph{(2)} and \emph{(3)}. 

To prove transitivity, take $\langle x, y \rangle, \langle y, z \rangle \in \alpha$. By what we have shown so far, there are $\gamma, \mu \in \Delta(\theta_0, \dots, \theta_{n-1})$ that totally support $\langle y,x \rangle $ (note the reversed order) and $\langle y, z \rangle $, respectively. Now set

\[
 \zeta = \glue_{\{0\}}(\langle \faces_0^1(\gamma), \faces_0^1(\mu) \rangle )
\]
Because $\faces_{0}^0(\gamma)$ and $\faces_{0}^0(\mu)$ are both constant cubes with value $y$, it follows that $\zeta \in \Delta(\theta_0, \dots, \theta_{n-1})$. It is easy to see that $\langle x, z \rangle$ is $(0)$-supported by $\zeta$, so we conclude that $\langle x, z \rangle \in \alpha$.

It remains to check that $\Delta(\theta_0, \dots, \theta_{n-1})$ has $(\alpha, n-1)$-centrality. Suppose that $\gamma \in \Delta(\theta_0, \dots, \theta_{n-1})$ has the property that each of its $(n-1)$-supporting lines is an $\alpha$-pair. Our aim is to show that the $(n-1)$-pivot line is also an $\alpha$-pair. This is achieved by exhibiting a systematic way of gluing various cubes together to produce a cube in $\Delta(\theta_0, \dots, \theta_{n-1})$ that totally supports the $(n-1)$-pivot line of $h$. Such a procedure is developed in the following sequence of claims.

\begin{claim}\label{claim:glueingsupports1}
Let $\gamma \in \Delta(\theta_0, \dots, \theta_{n-1})$ be such that every $(n-1)$-supporting line of $\faces_{n-2}^0(\gamma)$ is an $\alpha$-pair. Let $i \in n-1$. There exists $z_i(\gamma) \in \Delta(\theta_0, \dots, \theta_{n-1})$ with the following properties:

\begin{enumerate}
\item 
$\cut_{\{ i, \dots, n-2\}} (z_i(\gamma))_{\textbf{0}} 
= \cut_{\{i, \dots, n-2\}}( \gamma)_\textbf{0} ,
$ and
\item $ \cut_{\{i, \dots, n-2 \}}(z_i(\gamma))_f =  \cut_{\{i, \dots, n-2 \}}(\refl_{n-1}^0(\gamma))_\textbf{0}$, for every $f \in 2^{ \{ i, \dots, n-2 \}}$ such that $f \neq \textbf{0}$.
\end{enumerate}
\end{claim}

\begin{claimproof}
We proceed by induction on $i \in n-1$. To establish the basis of the induction, let $i=0$. Notice that $\cut_{\{i, \dots, n-2 \}}(\gamma)_\textbf{0} = \cut_{\{0, \dots, n-2\}}(\gamma)_\textbf{0} = \lines_{n-1}(\gamma)_\textbf{0}$ (we called this the $(n-1)$-antipivot line of $\gamma$.) Suppose that it is the pair $\langle a,b \rangle$. The $(n-1)$-antipivot line of $\gamma$ is also the $(n-1)$-antipivot line of $\faces_{n-2}^0(\gamma)$, so $\langle a, b \rangle \in \alpha$ by assumption. Therefore, $\com_n(a,b) \in \Delta(\theta_0, \dots, \theta_{n-1})$. Now set
\[
z_0(\gamma) = \sym_0(\sym_1(\dots \sym_{n-2}(\com_n(a,b) ) \dots)).
\]
It is easy to see that $\langle a, b \rangle $ is the $(n-1)$-antipivot line of $z_0(\gamma)$, and that every other cross-section line of $z_0(\gamma)$ is the constant pair $\langle a,a \rangle$, so we have established that \emph{(1)} and \emph{(2)} of the claim hold for $i=0$. 

For the inductive step, suppose that the claim holds for $ i \in n-1$ with $ i \neq n-2$. Notice that $\sym_i(\gamma)$ also satisfies the assumptions of the claim, so the claim holds for both $z_i(\gamma)$ and $z_i(\sym_i(\gamma))$. Set $\alpha =z_i(\gamma) $ and $\beta = \sym_i(z_i(\sym_i(\gamma)))$. Let $\textbf{i}=(1,0, \dots, 0) \in 2^{\{i, \dots, n-1\}}$. Items \emph{(1)} and \emph{(2)} of the claim for $\sym_i(\gamma)$ translate into the following statements about $\beta$:

\begin{enumerate}
\item[\emph{(1)$^\beta$}]
$\cut_{\{ i, \dots, n-2\}} (\beta)_{\textbf{i}} 
= \cut_{\{i, \dots, n-2\}}( \gamma)_\textbf{i} 
,$ and 
\item[\emph{(2)$^\beta$}] $ \cut_{\{i, \dots, n-2 \}}(\beta)_f =  \cut_{\{i, \dots, n-2 \}}(\refl_{n-1}^0(\gamma))_\textbf{i}$, for every $f \in 2^{ \{ i, \dots, n-2 \}}$ such that $f \neq \textbf{i}$.
\end{enumerate}

Define
$
z_{i+1}(\gamma) = \glue_{\{i\}}\left( \langle \faces_i^0(\alpha), \faces_i^1(\beta) \rangle
\right).
$ We show that $z_{i+1}(\gamma) \in \Delta(\theta_0, \dots, \theta_{n-1})$, and that \emph{(1)} and \emph{(2)} of the claim hold. Set 
\[ \epsilon =
\glue_{\{i\}}(\langle \faces_i^1(\alpha), \faces_i^0(\beta) \rangle ) 
\]
Let $f \in 2^{\{i+1, \dots, n-2\}}$. Set $_0f = \{(i,0)\} \union f $ and $_1f = \{(i,1) \} \union f$.   We compute
\begin{align*}
\cut_{\{i+1, \dots, n-2\}}(\epsilon)_f &= \cut_{\{i+1, \dots, n-2\}}
\left(\glue_{\{i\}}(\langle \faces_i^1(\alpha), \faces_i^0(\beta) \rangle ) 
\right)_f\\
&= \glue_{\{i\}} ( \cut_{\{i, \dots, n-2\}}(\alpha)_{_1f}, \cut_{\{i, \dots, n-2\}}(\beta)_{_0f})\\
&= \glue_{\{i \}} ( \cut_{\{i, \dots, n-2\}}(\refl_{n-1}^0(\gamma))_\textbf{0}, \cut_{\{i, \dots, n-2\}}(\refl_{n-1}^0(\gamma))_\textbf{i})\\
&= \cut_{\{i+1, \dots, n-2\}}(\refl_{n-1}^0(\gamma))_\textbf{0}\\
&= \refl_{n-1}^0(\cut_{\{i+1, \dots, n-2\}}(\gamma)_\textbf{0}),
\end{align*}
where the equality between the second and third lines follows from \emph{(2)} and \emph{(2)}$^\beta$. The above computation shows that 
\[
\cut_{\{i+1, \dots, n-1\}}(\epsilon)_g = \cut_{\{i+1, \dots, n-1\}}(\gamma)_\textbf{0}
\]
for every $g \in 2^{\{i+2, \dots, n-1\}}$, so we can apply Lemma \ref{lem:hyperreflexive} to conclude that $\epsilon \in \Delta(\theta_0, \dots, \theta_{n-1})$. Because $\Delta(\theta_0, \dots, \theta_{n-1})$ is $(n)$-transitive, this shows that $z_{i+1}(\gamma) \in \Delta(\theta_0, \dots, \theta_{n-1})$. 

Now we verify that \emph{(1)} holds for $z_{i+1}(\gamma)$.  We compute
\begin{align*}
\cut_{\{i+1, \dots, n-2\}}(\epsilon)_\textbf{0}
&= \cut_{\{i+1, \dots, n-2\}}
\left(\glue_{\{i\}}(\langle \faces_i^0(\alpha), \faces_i^1(\beta) \rangle ) 
\right)_\textbf{0}\\
&= \glue_{\{i\}} ( \cut_{\{i, \dots, n-2\}}(\alpha)_{\textbf{0}}, \cut_{\{i, \dots, n-2\}}(\beta)_{\textbf{i}})\\
&= \glue_{\{i \}} ( \cut_{\{i, \dots, n-2\}}(\gamma)_\textbf{0}, \cut_{\{i, \dots, n-2\}}(\gamma)_\textbf{i})\\
&= \cut_{\{i+1, \dots, n-2\}}(\gamma)_\textbf{0},
\end{align*}
where the equality between the second and third lines follows from \emph{(1)} and \emph{(2)}$^\beta$. A similar computation shows that \emph{(2)} holds for $z_{i+1}$, which completes the inductive step and the proof. 
\end{claimproof}

\begin{claim}\label{claim:gluingsupports2}
If $\gamma \in \Delta(\theta_0, \dots, \theta_{n-1})$ is such that each of its $(n-1)$-supporting lines is an $\alpha$-pair, then $\glue_{\{j\}}(\langle \faces_j^0(\refl_{n-1}^0(\gamma)), \faces_j^0(\gamma) \rangle) \in \Delta(\theta_0, \dots, \theta_{n-1})$, for every $j \in n-1$.
\end{claim}

\begin{claimproof}
We first show that the claim holds when $j=n-2$. We apply Claim \ref{claim:glueingsupports1} with $i=n-2$ and get
\[
z_{n-2}(\gamma)= \glue_{\{n-2\}}(\langle  \faces_{n-2}^0(\gamma), \faces_{n-2}^0(\refl_{n-1}^0(\gamma)) \rangle) \in \Delta(\theta_0, \dots, \theta_{n-1}).
\]
Because $\Delta(\theta_0, \dots, \theta_{n-1})$ is $n$-symmetric, the claim holds in this case. More generally, we observe that the proof of Claim \ref{claim:glueingsupports1} does not depend in any special way on the value $j=n-2$. Indeed, switching the coordinates $n-2$ and $j$, applying the same argument, and then switching the coordinates again will produce an argument that works for any value of $j \in n-1$. 
\end{claimproof}

\begin{claim}\label{claim:glueingsupports3}
Let $\gamma \in \Delta(\theta_0, \dots, \theta_{n-1})$ be such that each of its $(n-1)$-supporting lines is an $\alpha$-pair. Additionally, suppose $k\in n-1$ is such that, for all $l \in k$, every $(n-1)$-supporting line of $\faces_l^0(\gamma)$ is a constant pair. In this situation, there is $\gamma' \in \Delta(\theta_0, \dots, \theta_{n-1})$ such that
\begin{enumerate}
\item the $(n-1)$-pivot lines of $\gamma'$ and $\gamma$ are equal, and
\item every $(n-1)$-supporting line of $\faces_l^0(\gamma')$ is a constant pair, for all $l \in k+1$. 

\end{enumerate}
\end{claim}

\begin{claimproof}

Set 
$
\gamma' = \glue_{\{k\}}(\langle \faces_k^0(\refl_{n-1}^0(\gamma)), \faces_k^1(\gamma) \rangle).
$
Claim \ref{claim:gluingsupports2} with $j = k$ ensures that $\glue_{\{k\}}(\langle \faces_k^0(\refl_{n-1}^0(\gamma)), \faces_k^0(\gamma) \rangle) \in \Delta(\theta_0, \dots, \theta_{n-1})$, from which it follows that $\gamma' \in \Delta(\theta_0, \dots, \theta_{n-1})$. It is clear that \emph{(1)} holds, because $\faces_k^1(\gamma)= \faces_k^1(\gamma')$. 

To check \emph{(2)}, let $f \in 2^{n-1}$ be such that $f(l) =0$ for some $l \in k+1$. If $l \in k$, then $\lines_{n-1}(\gamma')_f = \lines_{n-1}(\gamma)_f$, and is a constant pair by assumption. If $l =k $, then $\lines_{n-1}(\gamma')_f$ must be an $(n-1)$-cross-section line of $\faces_k^0(\gamma') = \faces_k^0(\refl_{n-1}^0(\gamma))$, and is therefore also a constant pair. 
\end{claimproof}

Finally, let $\gamma \in \Delta(\theta_0, \dots, \theta_{n-1})$ be such that every $(n-1)$-supporting line is an $\alpha$-pair. Claim \ref{claim:glueingsupports3} provides a recursive procedure to replace each $(n-1)$-supporting line of $\gamma$ with a constant pair, starting with those $(n-1)$-cross-section lines that belong to $\faces_0^0(\gamma)$, and ending with those that belong to $\faces_{n-2}^0(\gamma)$. This demonstrates that the $(n-1)$-pivot line of $\gamma$ is $(n-1)$-supported by some $\zeta \in \Delta(\theta_0, \dots, \theta_{n-1})$, and the proof is finished. 

\end{proof}

\begin{cor}
Let $\A$ be an algebra, $n \geq 2$, and take $(\theta_0, \dots, \theta_{n-1}) \in \Con(\A)^n$. The hypercommutator is independent of the order of its arguments, i.e.\  
\[
[\theta_0, \dots, \theta_{n-1}]_H = [\theta_{\sigma(0)}, \dots, \theta_{\sigma(n-1)}]_H,
\]
for all permutations $\sigma \in S_n$.
\end{cor}
\begin{proof}
This is follows immediately from the equivalence of \emph{(1)} and \emph{(2)} in Theorem \ref{thm:charhypercomm}.
\end{proof}

With Theorem \ref{thm:charhypercomm} in hand, we are now able to prove that the hypercommutator satisfies what we call HHC8, which is the condition that 

\[
[\theta_0, \dots, \theta_{m-1},[\theta_{m}, \dots, \theta_{n-1}]_H]_H \leq [\theta_0, \dots, \theta_{n-1}]_H ,
\]
for any algebra $\A$, $n \geq 3$, $m \in n-1$, and $(\theta_0, \dots, \theta_{n-1}) \in \Con(\A)^n$. To prove it, we will demonstrate that $\Delta(\theta_0, \dots, \theta_{m-1}, [\theta_m, \dots, \theta_{n-1}]_H)$ is a subset of the projection of a special subalgebra of $\Delta(\theta_0, \dots, \theta_{n-1})$ onto a coordinate system with fewer dimensions. This construction will produce an $n$-dimensional commutator cube for any pair of elements belonging to the congruence defined by the nested expression on the left hand side of the HHC8 inequality. 

Let $\A$ be an algebra, $n\geq 3$, $m \in n-1$, and $(\theta_0, \dots, \theta_{n-1}) \in \Con(\A)$. We define the \textbf{$m$-nest} of $\Delta(\theta_0, \dots, \theta_{n-1})$ as
\begin{align*}
N(\theta_0, \dots, \theta_{n-1}) = \{ \gamma \in \Delta(\theta_0, \dots, \theta_{n-1}): \cut_{m}&(\gamma)_f \text{ is an $(n-m)$-dimensional}\\
&\text{ commutator cube for all $f \in 2^{m}$} 
\}.
\end{align*}
Equivalently, 
\begin{align*}
N(\theta_0, \dots, \theta_{n-1}) = \{ \gamma \in\Delta(\theta_0, \dots, \theta_{n-1}):& \cut_{\{m, \dots, n-1\}}(\gamma) \text{ is an $(n-m)$}\\
&\text{-dimensional commutator cube }\\
&\text{with vertices labeled by elements of $A^{2^{m}}$} 
\}.
\end{align*}
The $m$-nest for $n=3$ and $m=1$ was used in the proof of Theorem \ref{thm:bthhc8}. We provide a picture of a typical $2$-nest element $\gamma$ when $n=4$ in Figure \ref{fig:2nest}, where $\faces_3^1(\gamma)$ is the `inner' cube, and $\cut_{\{0,1\}}(\gamma)_\textbf{1}$ is outlined in bold. It is clear from the definition that $N(\theta_0, \dots, \theta_{n-1}) \leq \Delta(\theta_0, \dots, \theta_{n-1})$. The next lemma establishes an important property of the $m$-nest.

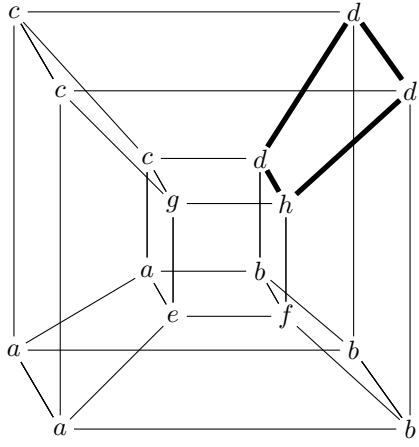
\begin{figure}[bt]
 \centering
 \scalebox{1}
 {
 \begin{tikzpicture}[scale=1.5]
	 \tikzstyle{vertex}=[circle,minimum size=10pt,inner sep=0pt]
	 \tikzstyle{selected vertex} = [vertex, fill=red!24]
	 \tikzstyle{selected edge} = [draw,line width=5pt,-,red!50]
	 \tikzstyle{edge} = [draw,-,black]
	 \node[vertex] (v0) at (0,0) {$e$};
	 \node[vertex] (v1) at (0,1) {$g$};
	 \node[vertex] (v2) at (1,0) {$f$};
	 \node[vertex] (v3) at (1,1) {$h$};
	 
	 \node[vertex] (v4) at (-0.23, 0.4) {$a$};
	 \node[vertex] (v5) at (-0.23,1.4) {$c$};
	 \node[vertex] (v6) at (.77,0.4) {$b$};
	 \node[vertex] (v7) at (.77,1.4) {$d$};
	 
	 \node[vertex] (v8) at (-1,-1) {$a$};
	 \node[vertex] (v9) at (-1,2) {$c$};
	 \node[vertex] (v10) at (2.1,-1) {$b$};
	 \node[vertex] (v11) at (2.1,2) {$d$};
	  
	 \node[vertex] (v13) at (-1.41,2.7) {$c$};
	 \node[vertex] (v12) at (-1.41,-0.3) {$a$};
	 \node[vertex] (v14) at (1.60,-0.3) {$b$};
	 \node[vertex] (v15) at (1.60,2.7) {$d$};
	
	 \draw[edge] (v0) -- (v1) -- (v3) -- (v2) -- (v0);
	 \draw[edge] (v0) -- (v4) -- (v5) -- (v1) -- (v0);
	 \draw[edge] (v2) -- (v6) -- (v7) -- (v3) -- (v2);
	 \draw[edge] (v4) -- (v6) -- (v7) -- (v5) -- (v4);
	 \draw[edge] (v8) -- (v9) -- (v13) -- (v12) -- (v8);
	 \draw[edge] (v0) -- (v4) -- (v12) -- (v8) -- (v0);
	 \draw[edge] (v1) -- (v9) -- (v13) -- (v5) -- (v1);
	 \draw[edge] (v2) -- (v10) -- (v14) -- (v6) -- (v2);
	 \draw[edge] (v8) -- (v10) -- (v14) -- (v12) -- (v8);
	 \draw[edge] (v3) -- (v11) -- (v15) -- (v7) -- (v3);
	 \draw[edge] (v10) -- (v11) -- (v15) -- (v14) -- (v10);
	 \draw[edge] (v9) -- (v11) -- (v15) -- (v13) -- (v9);
	 \draw[edge, line width=2pt] (v7) -- (v15) -- (v11) -- (v3) -- (v7);
	 
 \end{tikzpicture}
 }
 \caption{Typical $2$-nest element $\gamma$ when $n=4$.}\label{fig:2nest}
 \end{figure}

\begin{lem}\label{lem:nest1}
Let $\A$ be an algebra, $n\geq 3$, $m \in n-1$, and $(\theta_0, \dots, \theta_{n-1}) \in \Con(\A)$. Then,
$
\Delta(\theta_0, \dots, \theta_{m-1}, [ \theta_m, \dots, \theta_{n-1}]_H) \leq 
\cut_{\{m+1, \dots, n-1\}}(N(\theta_0, \dots, \theta_{n-1}))_\textbf{1}.
$
\end{lem}

\begin{proof}
Set $\alpha = [\theta_m, \dots, \theta_{n-1}]_H$. The lemma is a consequence of the following two facts:
\begin{enumerate}
\item $M(\theta_0, \dots, \theta_{m-1}, \alpha) \leq 
\cut_{\{m+1, \dots, n-1\}}(N(\theta_0, \dots, \theta_{n-1}))_\textbf{1}$, and
\item $\cut_{\{m+1, \dots, n-1\}}(N(\theta_0, \dots, \theta_{n-1}))_\textbf{1}$ is an $(m+1)$-dimensional congruence. 
\end{enumerate}
Before we proceed we point out that, although the notation $\cube_i(x,y)$ does not specify a dimension, the dimension should be discernible from the dimension of the relation to which we assert it belongs. Recall that 
\begin{align*}
M(\theta_0, \dots, \theta_{m-1}, \alpha) = \Sg_{\A^{2^{m+1}}}
\bigg(
&\Union_{i \in m}
\left\{ \cube_i(x,y): \langle x,y \rangle \in \theta_i
\right\}\\
&\union 
\{\cube_m(x,y): \langle x,y \rangle \in \alpha \}
\bigg).
\end{align*}
To establish (1) it is enough to show that these generators belong to \newline $\cut_{\{m+1, \dots, n-1\}}(N(\theta_0, \dots, \theta_{n-1}))_\textbf{1}$. There are two cases to address, the first dealing with $i \in m$, and the second dealing with the last coordinate $m$. 

For the first case, let $i \in m$ and take $\gamma = \cube_i(a,b) \in A^{2^n}$ for $\langle a, b  \rangle \in \theta_i$. Now, $\gamma$ is a generator of $\Delta(\theta_0, \dots, \theta_{n-1})$ and by definition we have that 
$
\cut_{\{0, \dots, m-1\}}(\gamma)_f 
$
is a constant cube with value either $a$ or $b$, depending on whether $f(i) = 0$ or $f(i)=1$. Therefore, $\gamma \in N(\theta_0, \dots, \theta_{n-1})$. On the other hand, observe that 
\[
\cut_{\{m+1, \dots, n-1\}}(\gamma)_\textbf{1} = \cube_i(a,b) \in A^{2^{m+1}}.
\]
This shows that 
$
\{\cube_i(x,y): \langle x,y \rangle \in \theta_i \} \subseteq \cut_{\{m+1, \dots, n-1\}}(N(\theta_0, \dots, \theta_{n-1}))_\textbf{1}.
$

Now we deal with the second case. Now, Lemma \ref{lem:deltaprojectstodelta} indicates that
\[
\Delta(\theta_m, \dots, \theta_{n-1}) = \cut_{\{0, \dots, m-1\}}(\Delta(\theta_0, \dots, \theta_{m-1}))_\textbf{1},
\]
Suppose $\langle a, b \rangle \in \alpha$. It follows from Theorem \ref{thm:charhypercomm} that $\com_{n \setminus m}(a,b) \in \Delta(\theta_m, \dots, \theta_{n-1})$. Therefore, there is $\epsilon \in \Delta(\theta_0, \dots, \theta_{n-1})$ so that 
\[
\cut_{\{0, \dots, m-1\}}(\epsilon)_\textbf{1}= \com_{n\setminus m}(a,b)
\]
We apply Corollary \ref{lem:hyperreflexive} to this situation and conclude that there is $\mu \in \Delta(\theta_0, \dots, \theta_{n-1})$ so that 
$
\cut_{\{0, \dots, m-1\}}(\mu)_f = \com_{n \setminus m}(a,b)
$
for all $f \in 2^m$. It is immediate that $\mu \in N(\theta_0, \dots, \theta_{n-1})$ . We now establish that 
$
\cut_{\{m+1, \dots, n-1\}}(\mu)_\textbf{1} = \cube_m(a,b),
$
which will finish the proof of (1). 

Indeed, for $\textbf{1} \in 2^{\{m+1, \dots, n-1\}} $ and $g\in 2^{\{0, \dots, m\}}$, we compute
\begin{align*}
\left(\cut_{\{m+1, \dots, n-1\}}\mu)_\textbf{1}\right)_{g} &= 
\mu_{g \union \textbf{1}} \\
&= \mu_{(g_0, \dots, g_{m-1}, g_m, 1, \dots, 1)}
= \begin{cases} a & \text{if $g_m = 0$}\\
				b & \text{if $g_m = 1$},
\end{cases}
\end{align*}
where the case distinction follows from the fact that the first $m$ arguments of $g$ provide coordinates for a vertex label of $\cut_{\{0, \dots, m-1\}}(\mu)$.

Now we establish (2). Set $R = \cut_{\{m+1, \dots, n-1\}}(N(\theta_0, \dots, \theta_{n-1}))_\textbf{1}$. It is immediate that $R$ is compatible and easy to see that $R$ is $(m+1)$-reflexive. Let us show that $R$ is $(m+1)$-transitive. Take $i \in m+1$ and $\lambda, \nu \in N(\theta_0, \dots, \theta_{n-1})$ satisfying
\[
\faces_i^1(\cut_{\{m+1, \dots, n-1\}}(\lambda)_\textbf{1}) 
=
\faces_i^0(\cut_{\{m+1, \dots, n-1\}}(\nu)_\textbf{1}).
\]
We want to show that 
\[
\tau = \glue_{\{i\}}(\langle \faces_i^0(\cut_{\{m+1, \dots, n-1\}}(\lambda)_\textbf{1}) 
,
\faces_i^1(\cut_{\{m+1, \dots, n-1\}}(\nu)_\textbf{1})
\in R.
\]

Either $i\in m$ or $i=m$. The first case is easier to understand, because in this situation $\faces_i^1(\lambda) = \faces_1^0(\nu)$. Indeed, take $f = (f_0, \dots, f_{i-1}, f_{i+1}, \dots, f_{n-1}) \in 2^{n \setminus \{i\}}$. Notice that
\begin{align*}
(\faces_i^1(\lambda))_f &= \lambda_{(f_0, \dots, f_{i-1}, 1, f_{i+1}, \dots, f_{n-1}) }\\
&= (\cut_m(\lambda)_{(f_0, \dots, f_{i-1}, 1, f_{i+1}, \dots, f_{m-1})})_{(f_m, f_{m+1}, \dots, f_{n-1})}.
\end{align*}
We assume that $\lambda \in N(\theta_0, \dots, \theta_{n-1})$, so $\cut_m(\lambda)_{(f_0, \dots, f_{i-1}, 1, f_{i+1}, \dots, f_{m-1})}$ is an $(n-m)$-dimensional commutator cube. If $f_m = 0$ then we finish the computation as follows:
\begin{align*}
(\faces_i^1(\lambda))_f
&= (\cut_m(\lambda)_{(f_0, \dots, f_{i-1}, 1, f_{i+1}, \dots, f_{m-1})})_{(0, f_{m+1}, \dots, f_{n-1})} \\
&= (\cut_m(\lambda)_{(f_0, \dots, f_{i-1}, 1, f_{i+1}, \dots, f_{m-1})})_{(0, 1, \dots, 1)} \\
&= (\cut_m(\nu)_{(f_0, \dots, f_{i-1}, 0, f_{i+1}, \dots, f_{m-1})})_{(0, 1, \dots, 1)} \\
&= (\cut_m(\nu)_{(f_0, \dots, f_{i-1}, 0, f_{i+1}, \dots, f_{m-1})})_{(0, f_{m+1}, \dots, f_{n-1})} \\
&= \faces_i^1(\nu)_f
\end{align*}
If $f_m=1$ but $(f_m, \dots, f_{n-1}) \neq \textbf{1}$, the computation is completed as follows:

\begin{align*}
(\faces_i^1(\lambda))_f
&= (\cut_m(\lambda)_{(f_0, \dots, f_{i-1}, 1, f_{i+1}, \dots, f_{m-1})})_{(1, f_{m+1}, \dots, f_{n-1})} \\
&= (\cut_m(\lambda)_{(f_0, \dots, f_{i-1}, 1, f_{i+1}, \dots, f_{m-1})})_{(0, 1, \dots, 1)} \\
&= (\cut_m(\nu)_{(f_0, \dots, f_{i-1}, 0, f_{i+1}, \dots, f_{m-1})})_{(0, 1, \dots, 1)} \\
&= (\cut_m(\nu)_{(f_0, \dots, f_{i-1}, 0, f_{i+1}, \dots, f_{m-1})})_{(1, f_{m+1}, \dots, f_{n-1})} \\
&= \faces_i^1(\nu)_f
\end{align*}
The case when $(f_m, \dots, f_{n-1}) = \textbf{1}$ is handled similarly. We conclude that $\zeta = \glue_{\{i\}}(\langle\faces_i^0, \faces_i^1 \rangle ) \in N(\theta_0, \dots, \theta_{n-1})$. It is easy to see that 
\[\cut_{\{m+1, \dots, n-1 \}}(\zeta)_\textbf{1} = \tau,
\]
so the case when $i \in m$ is finished. 

Suppose now that $i = m$. Notice that $\cut_{\{m, \dots, n-1\}}(N(\theta_0, \dots, \theta_{n-1}))$ is a collection of $(n-m)$-dimensional commutator cubes whose vertices are labeled by elements of $\Delta(\theta_0, \dots, \theta_{m-1})$. That is,  
\begin{align*}
\cut_{\{m, \dots, n-1\}}(\lambda) &= \com_{n\setminus m}
(a,b), \text{ and}\\
\cut_{\{m, \dots, n-1\}}(\nu) &= \com_{n\setminus m}
(c,d),
\end{align*}
where 
\begin{align*}
a &= \cut_{\{m, \dots, n-1\}}(\lambda)_{(0, 1, \dots, 1)},\\
b &= \cut_{\{m, \dots, n-1\}}(\lambda)_{(1, 1, \dots, 1)},\\
c &= \cut_{\{m, \dots, n-1\}}(\nu)_{(0, 1, \dots, 1)}, \text{ and}\\
d &= \cut_{\{m, \dots, n-1\}}(\nu)_{(1, 1, \dots, 1)}.
\end{align*}
The assumption that 
$\faces_m^1(\cut_{\{m+1, \dots, n-1\}}(\lambda)_\textbf{1})
= 
\faces_m^0(\cut_{\{m+1, \dots, n-1\}}(\nu)_\textbf{1})
$
exactly means that $b=c$. A consequence of Lemma \ref{lem:higherconofhighercon} is that 
\[
\cut_{\{m, \dots, n-1\}}(\Delta(\theta_0, \dots, \theta_{n-1}))
\in \Con_{n \setminus m}(\Delta(\theta_0, \dots, \theta_{m-1})).
\] 

We demonstrated in the proof of Theorem \ref{thm:charhypercomm} that the collection of pairs $\langle x,y \rangle $ that are totally supported by some higher dimensional congruence is a transitive relation. In the current situation this means that 
\[
\com_{n\setminus m}(a,d) \in \cut_{\{m, \dots, n-1\}}(\Delta(\theta_0, \dots, \theta_{n-1})).
\]
Set $ \rho = \glue_{\{m, \dots, n-1\}}(\com_{n\setminus m}(a,d) )$.
Evidently $\rho \in N(\theta_0, \dots, \theta_{n-1})$, and a routine computation shows that $\cut_{\{m+1, \dots, n-1\}}(\rho)_\textbf{1}= \tau$. This finishes the proof that $R$ is $(m+1)$-transitive. A similar kind of argument shows that $R$ is $(m+1)$-symmetric. 

The lemma now follows. Indeed, having established $(1)$ and $(2)$, we are now able to conclude that 
\[
\Delta(\theta_0, \dots, \theta_{m-1}, \alpha) = \tc(M(\theta_0, \dots, \theta_{m-1}, \alpha) \leq \tc(R) = R,
\]
as desired.
\end{proof}

\begin{thm}[HHC8]\label{thm:HHC8}
If $\A$ is an algebra, $n \geq 3$, and $(\theta_0, \dots, \theta_{n-1}) \in \Con(\A)^n$, then
\[
[\theta_0, \dots, \theta_{m-1}, [\theta_m, \dots, \theta_{n-1}]_H]_H \leq [\theta_0, \dots, \theta_{n-1}]_H
\]
for all $m \in n-1$. 
\end{thm}

\begin{proof}
Suppose that $\langle x,y \rangle \in 
[\theta_0, \dots, \theta_{m-1}, [\theta_m, \dots, \theta_{n-1}]_H]_H$. Applying Theorem \ref{thm:charhypercomm} shows that 
\[
\com_{m+1}(x,y) \in \Delta(\theta_0, \dots, \theta_{m-1}, [\theta_m, \dots, \theta_{n-1}]_H).
\]
Lemma \ref{lem:nest1} allows us to conclude that there is $\gamma \in N(\theta_0, \dots, \theta_{n-1})$ such that 
$\cut_{\{m+1, \dots, n-1\}}(\gamma)_\textbf{1} = \com_{m+1}(x,y)$, and the definition of $N(\theta_0, \dots, \theta_{n-1})$ forces $\gamma = \com_n(x,y)$. Applying Theorem \ref{thm:charhypercomm} yet again shows that $\langle x,y \rangle \in [\theta_0, \dots, \theta_{n-1}]_H$, and the proof is finished. 
\end{proof}
We finish the section with a corollary and the theorem promised by the title. 
\begin{cor}\label{cor:diagHC8}
Let $n \geq 3$ and $m\in n-1$. If $\theta$ is a congruence of a Taylor algebra $\A$, then 
\[
[\underbrace{\theta, \dots, \theta}_{m},[\underbrace{\theta, \dots, \theta}_{n-m}]_{TC}]_{TC}\leq [\underbrace{\theta, \dots, \theta}_{n}]_{TC}.
\]
\end{cor}
\begin{proof}
The proof is the same as the proof of Corollary \ref{cor:ternarynestingproperty}. We observe that the following sequence of congruences is increasing:
\[
[\underbrace{\theta, \dots, \theta}_{m},[\underbrace{\theta, \dots, \theta}_{n-m}]_{TC}]_{TC}
\leq 
[\underbrace{\theta, \dots, \theta}_{m},[\underbrace{\theta, \dots, \theta}_{n-m}]_{H}]_{H}
\leq 
[\underbrace{\theta, \dots, \theta}_{n}]_H
= 
[\underbrace{\theta, \dots, \theta}_{n}]_{TC}.
\]
Indeed, the first bound is a consequence of Theorem \ref{thm:basicpropertiescommutator}, the second bound is a consequence of Theorem \ref{thm:HHC8}, and the third equality is a consequence of Theorem \ref{thm:allarityH=TC}.
\end{proof}
\begin{thm}\label{thm:supniltaylorarenil}
Supernilpotent Taylor algebras are nilpotent.
\end{thm}

\begin{proof}
Let $\A$ be a Taylor algebra and $\theta \in \Con(\A)$. We show that 

\[
(\theta]_n \leq [\underbrace{\theta, \dots, \theta}_{n+1}]_{TC}.
\]
We proceed inductively over the lower central series of $\A$. The basis is clear, because $\theta \leq \theta$. Suppose the bound holds for $n$. A consequence of Theorem \ref{thm:basicpropertiescommutator} together with Corollary \ref{cor:diagHC8} is that

\[
(\theta]_{n+1} = [\theta, (\theta]_n]_{TC} \leq  [\theta, [\underbrace{\theta, \dots, \theta}_{n}]_{TC}]_{TC}
\leq [\underbrace{\theta, \dots, \theta}_{n+1}]_{TC},
\]
and the result follows.

\end{proof}

\section{A characterization of congruence meet-semidistributivity}\label{sec:sdmeet}

A variety $\var$ is said to \textbf{congruence meet-semidistributive}, or $\sd$, if each of its congruence lattices satisfies the implication

\[
(\gamma \meet \alpha = \gamma \meet \beta)
\implies 
(\gamma \meet (\alpha \join \beta) = \gamma \meet \alpha).
\]
A variety $\var$ is said to be \textbf{congruence neutral} if 
\[
[\alpha, \beta]_{TC} = \alpha \meet \beta,
\]
for all algebras $\A \in \var$ and $\alpha, \beta \in \Con(\A)$.
It is well known that every $\sd$ variety is congruence neutral, and vice versa \cite{kearnesszendreirel}. 
Along these lines, let $n \geq 2$. We say that an operation $[\underbrace{\cdot, \dots, \cdot}_n]: \mathcal{L}^n \to \mathcal{L}$ on a lattice $\mathcal{L}$ is \emph{neutral} if
\[
[\theta_0, \dots, \theta_{n-1}] =  \Meet_{i\in n} \theta_i
\]
for all $(\theta_0, \dots, \theta_{n-1}) \in \mathcal{L}^n$.

\begin{prop}\label{prop:hypneutral=tcneutral}
Let $\A$ be an algebra and $n\geq 2$. The $(n)$-ary hypercommutator is neutral on $\Con(\A)$ if and only if the $(n)$-ary term condition commutator is neutral on $\Con(\A)$.
\end{prop}

\begin{proof}

Suppose that $[\alpha_0, \dots, \alpha_{n-1}]_H = \Meet_{i \in n} \alpha_i$ for all $(\alpha_0, \dots, \alpha_{n-1}) \in \Con(\A)^n$. This assumption, along with Theorems \ref{thm:basicpropertiescommutator} and \ref{thm:allarityH=TC}, is used to produce the nondecreasing sequence of congruences
\begin{align*}
\Meet_{i \in n} \theta_i &= [\theta_0, \dots, \theta_{n-1}]_H
= [[\theta_0, \dots, \theta_{n-1}]_H, \dots, [\theta_0, \dots, \theta_{n-1}]_H]_H\\
&= \left[ \Meet_{i \in n} \theta_i, \dots, \Meet_{i \in n} \theta_i \right]_H 
= \left[ \Meet_{i \in n} \theta_i, \dots, \Meet_{i \in n} \theta_i \right]_{TC} \\
& \leq [ \theta_0, \dots, \theta_{n-1}]_{TC}\\
&\leq \Meet_{i \in n} \theta_i,
\end{align*}
which forces the $(n)$-ary term condition commutator to be neutral. The other direction is an obvious consequence of Theorem \ref{thm:basicpropertiescommutator}.
\end{proof}

We can now apply some of the theory developed in this article to extend the congruence neutral characterization of congruence meet semidistributive varieties.

\begin{thm}
Let $\var$ be a variety of algebras. The following conditions are equivalent:

\begin{enumerate}
\item $\var$ is $\sd$.
\item $\Delta(\alpha, \alpha) = \rect(\alpha, \alpha)$ for all congruences $\alpha$ of algebras $\A \in \var$.
\item $\Delta(\underbrace{\alpha, \dots, \alpha}_n) = \rect(\underbrace{\alpha, \dots, \alpha}_n)$ for every $n \geq 2$, for all congruences $\alpha$ of algebras $\A \in \var$.
\item $\Delta(\underbrace{\alpha, \dots, \alpha}_n) = \rect(\underbrace{\alpha, \dots, \alpha}_n)$ for some fixed $n \geq 2$, for all congruences $\alpha$ of algebras $\A \in \var$.
\item There exists $n\geq 2$ so that the $(n)$-ary hypercommutator is neutral across $\var$.
\item The binary hypercommutator is neutral across $\var$.
\end{enumerate}
\end{thm}

\begin{proof}
We prove that \emph{(1)}$\implies \dots \implies$\emph{(6)}$\implies$\emph{(1)}. It is well known that the binary term condition commutator is neutral (see Corollary 4.7 of \cite{kearnesszendreirel}), so \ref{prop:hypneutral=tcneutral} indicates the binary hypercommutator is also neutral. Take $\A \in \var$ and $\alpha \in \Con(\A)$. Suppose that 
\[
\mu = \Square[a][c][b][d] \in \rect(\alpha, \alpha).
\]
We want to show that $\mu \in \Delta(\alpha, \alpha)$. We are assuming that $[\alpha, \alpha]_H = \alpha$, so Theorem \ref{thm:charhypercomm} indicates that the outer two squares of the sequence 
\[
\begin{tikzpicture}[scale=1]
	 \tikzstyle{vertex}=[circle,minimum size=10pt,inner sep=0pt]
	 \tikzstyle{selected vertex} = [vertex, fill=red!24]
	 \tikzstyle{selected edge} = [draw,line width=5pt,-,red!50]
	 \tikzstyle{edge} = [draw,-,black]
	 \node[vertex] (v0) at (0,0) {$a$};
	 \node[vertex] (v1) at (0,1) {$c$};
	 \node[vertex] (v2) at (1,0) {$a$};
	 \node[vertex] (v3) at (1,1) {$a$};
	 \node[vertex] (v4) at (2,0) {$b$};
	 \node[vertex] (v5) at (2,1) {$b$};
	 \node[vertex] (v6) at (3,0) {$b$};
	 \node[vertex] (v7) at (3,1) {$d$};

	 \draw[edge] (v0) -- (v1) -- (v3) -- (v2) -- (v0);
	 \draw[edge] (v4) -- (v5) -- (v7) -- (v6) -- (v4);
	 \draw[edge] (v3) -- (v5);
	 \draw[edge] (v2) -- (v4);

 \end{tikzpicture}
\]
belong to $\Delta(\alpha, \alpha)$. The middle square is a generator of $\Delta(\alpha, \alpha)$, so an application of $(2)$-transitivity finishes the proof that \emph{(1)} implies \emph{(2)}. 

Suppose that \emph{(2)} holds. Take $\A \in \var$ and $\alpha \in \Con(\A)$. We proceed by induction on $n\geq 2$. Suppose that 
\[
\Delta(\underbrace{\alpha, \dots, \alpha}_{n-1}) = 
\rect(\underbrace{\alpha, \dots, \alpha}_{n-1})
\]
follows from \emph{(2)}, for $n-1 \geq 2$. We show that this also holds for $n$.

First define the congruence 

\begin{align*}
\zeta &= \left\{ \langle a, b \rangle \in (\A^{2^{n\setminus 2}})^2:
a, b \in \rect(\{\alpha_i = \alpha\}_{i \in n\setminus 2} ) 
\text{ and }
\langle a_\textbf{1},  b_\textbf{1} \rangle \in \alpha 
\right\}\\
&\in \Con( \rect(\underbrace{\alpha, \dots, \alpha}_{n-2})) \text{ (up to a change of coordinates).}
\end{align*}
We claim that $\Delta(\zeta, \zeta) \subseteq \cut_{\{0,1\}}(\Delta(\underbrace{\alpha, \dots, \alpha}_n))$. Indeed, it suffices to show that 
$
\cube_0(\zeta) \union \cube_1(\zeta) \subseteq \cut_{\{0,1\}}(\Delta(\alpha, \dots, \alpha))
$. Lemma \ref{lem:deltaprojectstodelta} and the inductive assumption show that  
\[
\epsilon =\faces_0^0(\Delta(\underbrace{\alpha, \dots, \alpha}_n)) = \Delta(\{\alpha_i = \alpha\}_{n\setminus 1}) =
\rect(\{\alpha_i = \alpha\}_{n\setminus 1}).
\]
However, $\faces_1(\epsilon) = \zeta$, so 
$
\cube_0(\zeta) \subseteq \cut_{\{0,1\}}(\Delta(\alpha, \dots, \alpha)).
$
A similar argument shows that 
$
\cube_1(\zeta) \subseteq \cut_{\{0,1\}}(\Delta(\alpha, \dots, \alpha)).
$
We assume that \emph{(2)} holds, so $\Delta(\zeta, \zeta) = \rect(\zeta, \zeta)$. We have demonstrated that 
\[
\glue_{\{0,1\}}
\left(
\left\{
\Square[a][c][b][d]: a,b,c,d \text{ belong to the same $\zeta$-class }
\right\}
\right)
\subseteq \Delta(\underbrace{\alpha, \dots, \alpha}_n),
\]
or equivalently, that $\rect(\underbrace{\alpha, \dots, \alpha}_n) \subseteq \Delta(\underbrace{\alpha, \dots, \alpha}_n)$.

Obviously, \emph{(3)} implies \emph{(4)}. Suppose that \emph{(4)} holds. Let $\A \in \var$ and $(\theta_0, \dots, \theta_{n-1}) \in \Con(\A)^n$. Because

\[
\rect
\left(\Meet_{i \in n} \theta_i, \dots, \Meet_{i\in n} \theta_i
\right)
= 
\Delta
\left(\Meet _{i \in n} \theta_i, \dots, \Meet_{i\in n} \theta_i
\right)
\leq 
\Delta(\theta_0, \dots, \theta_{n-1}),
\]
it follows that $\com_n(a,b) \in \Delta(\theta_0,  \dots, \theta_{n-1})$ for all $\langle a,b \rangle \in \Meet_{i\in n} \theta_i$. In view of Theorem \ref{thm:charhypercomm}, this shows that $[\theta_0, \dots, \theta_{n-1}]_H = \Meet_{i\in n} \theta_i$. Therefore, \emph{(5)} holds.

The remaining implications are consequences of Theorem \ref{thm:basicpropertiescommutator} and Proposition \ref{prop:hypneutral=tcneutral}, respectively. 

\end{proof}

\section{Acknowledgments}
The author wishes to acknowledge Keith Kearnes and Alexander Wires for stimulating conversations regarding this topic.

\bibliographystyle{amsplain}   % {{{1
\bibliography{refs.bib}
\begin{center}
  \rule{0.61803\textwidth}{0.1ex}   % 1/(golden ratio)
\end{center}
\subjclass{MSC 08A40 (08A05, 08B05)}
\end{document}